\documentclass[12pt,reqno]{amsart}%%%%%%%%%%%%{article}
\usepackage {amssymb}
\usepackage {amsmath}
\usepackage {bbm}
\usepackage{amsthm}
\usepackage{mathtools}
\usepackage{graphicx}
\usepackage {amscd}
\usepackage {epic}
\usepackage[alphabetic]{amsrefs}
\usepackage[colorlinks, citecolor=blue]{hyperref}
\usepackage{enumerate}
\usepackage{tikz}
\usepackage{tikz-cd}
\usepackage{verbatim,color,geometry}
\usepackage[all]{xy}
\usepackage[new]{old-arrows}
\usepackage{appendix}
\newcommand{\Aut}[0]{\operatorname{Aut}}

\newcommand{\PP}{{\mathbb{P} }}
\newcommand{\Z}{{\mathbb{Z} }}
\newcommand{\bA}{{\mathbb{A} }}
\newcommand{\bN}{{\mathbb{N} }}

\newcommand{\bG}{{\mathbb{G} }}
\newcommand{\bR}{{\mathbb{R} }}
\newcommand{\bZ}{{\mathbb{Z} }}

\newcommand{\tv}{{\tilde{v}}}
\newcommand{\tI}{{\tilde{I}}}

\newcommand{\rd}{\mathrm{d}}

\newcommand{\cR}{{\mathcal{R}}}

\newcommand{\cX}{{\mathcal{X}}}
\newcommand{\cL}{{\mathcal{L}}}
\newcommand{\cF}{\mathcal{F}}

\newcommand{\cW}{{\mathcal{W}}}
\newcommand{\cY}{{\mathcal{Y}}}

\newcommand{\kX}{{\mathfrak{X}}}

\newcommand{\DNA}{\mathbf{D}^{\mathrm{NA}}}
\newcommand{\JNA}{\mathbf{J}^{\mathrm{NA}}}
\newcommand{\LNA}{\mathbf{L}^{\mathrm{NA}}}
\newcommand{\bfJ}{\mathbf{J} }

\newcommand{\mult}{{\rm mult}}
\newcommand{\Supp}{{\rm Supp}}
\newcommand{\Hom}{{\rm Hom}}

\newcommand{\ord}{{\rm ord}}
\newcommand{\gr}{{\rm gr}}

\newtheorem{thm}{Theorem}[section]
\newtheorem*{theorem*}{Theorem}
\newtheorem*{conj*}{Conjecture}
\newtheorem*{ques*}{Question}

\newtheorem{lem}[thm]{Lemma}
\newtheorem{cor}[thm]{Corollary}

\newtheorem{prop}[thm]{Proposition}

\newtheorem{conj}[thm]{Conjecture}
\theoremstyle{definition}
\newtheorem{defn}[thm]{Definition}

\newtheorem{say}[thm]{}
\newtheorem{exmp}[thm]{Example}

\newtheorem{ques}[thm]{Question}    %!!!!!!!!!!!!!!!!!!!!
\newtheorem{rem}[thm]{Remark}

\newtheorem{defn-thm}[thm]{Definition--Theorem}  %!!!!!!!!!!!!!!!!!!!!!!!!
\newtheorem{defn-prop}[thm]{Definition--Proposition}  %!!!!!!!!!!!!!!!!!!!!!!!!
\newtheorem{defn-lem}[thm]{Definition--Lemma}  %!!!!!!!!!!!!!!!!!!!!!!!!
\newtheorem{keyidea}[thm]{Key idea}

\newtheorem*{thm*}{Theorem}

\theoremstyle{remark}
\newtheorem{claim}[thm]{Claim}

\newcommand{\Fut}{{\rm Fut}}
\newcommand{\Ding}{{\rm Ding}}
\newcommand{\Proj}{{\rm Proj}}
\newcommand{\Val}{{\rm Val}}
\newcommand{\vol}{{\rm vol}}
\newcommand{\hvol}{\widehat{\rm vol}}
\newcommand{\cD}{{\mathcal{D}}}
\newcommand{\cM}{{\mathcal{M}}}
\newcommand{\cG}{{\mathcal{G}}}
\newcommand{\cO}{{\mathcal{O}}}

\newcommand{\fa}{{\mathfrak{a}}}
\newcommand{\fm}{{\mathfrak{m}}}
\newcommand{\lct}{{\rm lct}}
\newcommand{\im}{{\rm im}}
\newcommand{\bQ}{{\mathbb{Q}}}
\newcommand{\Q}{{\mathbb{Q}}}
\newcommand{\N}{{\mathbb{N}}}
\newcommand{\fb}{{\mathfrak{b}}}

\newcommand{\ra}{{\rangle}}
\newcommand{\wt}{{\rm wt}}

\newcommand{\Spec}{{\rm Spec}}
\newcommand{\Gr}{{\rm Gr}}

\renewcommand{\d}{\delta}

\newcommand{\la}{\lambda}
\newcommand{\cI}{\mathcal{I}}

\newcommand{\R}{{\mathbb{R}}}
\begin{document}
\title{K-stability of Fano varieties: an algebro-geometric approach}
\date{\today}

\author {Chenyang Xu}

\address   { Current address: Princeton University, Princeton NJ 08544, USA}
\email     {chenyang@princeton.edu}

\address   { MIT, Cambridge, MA 02139, USA}
\email     {cyxu@mit.edu}

\address   { BICMR, Beijing 100871, China}
\email     {cyxu@math.pku.edu.cn}

\begin{abstract}{
We give a survey of the recent progress on the study of K-stability of Fano varieties by an algebro-geometric approach.}
\end{abstract}

\maketitle{}
\setcounter{tocdepth}{2}
\tableofcontents

{\let\thefootnote\relax\footnotetext{
CX is partially supported by a Chern Professorship of the MSRI (NSF No. DMS-1440140) and by the
National Science Fund for Distinguished Young Scholars (NSFC 11425101) `Algebraic Geometry' (2015-2018), NSF DMS-1901849 (2019-now). }
\marginpar{}
}

%\begin{center}
 %{\bf This is a note which is still under construction!}
%\end{center}

\emph{Throughout, we work over an algebraically closed  field $k$ of characteristic 0.}

\section{Introduction}\label{s-intro}

The concept of K-stability is one of the most precious gifts differential geometers brought to algebraic geometers. It was first introduced in \cite{Tia97} as a criterion to characterize the existence of K\"ahler-Einstein metrics on Fano manifolds, which is a central topic in complex geometry. The definition involves the sign of an analytic invariant, namely {\it the generalized Futaki invariant}, of all possible normal $\mathbb{C}^*$-degenerations of a Fano manifold $X$.  Later, in \cite{Don02}, the notion of K-stability was extended to general polarized manifolds $(X,L)$, and  {the generalized Futaki invariant} was reformulated in completely algebraic terms, which then allows arbitrary flat degenerations.

In this survey, we will discuss the recent progress on the algebraic study of K-stability of Fano varieties, using the ideas developed in higher dimensional geometry, especially the techniques centered around the minimal model program (MMP). While there is a long history for complex geometers to study  K\"ahler-Einstein metrics on Fano varieties, algebraic geometers, especially higher dimensional geometers, only started to look at the K-stability question relatively recently.  One possible reason is that only until the necessary background knowledge from the minimal model program was developed (e.g. \cite{BCHM10}), a systematic study in general could rise above the horizon.  Nevertheless, there has been spectacular progress from the algebro-geometric side in the last a few years, which we aim to survey in this note. 

\subsection{K-stability in algebraic geometry} 
Unlike many other stability notions, when K-stability was first defined in \cite{Tia97} and later formulated using purely algebraic geometric terms in \cite{Don02}, it was not immediately clear to algebraic geometers what this really means and whether it is going to be useful in algebraic geometry.  The definition itself is clear: one considers all one parameter subgroup flat degenerations $\cX/\bA^1$, and attach an invariant $\Fut(\cX)$ to it, and being K-stable is amount to saying that $\Fut(\cX)$ is always positive (except the suitably defined trivial ones). Following the philosophy of Donaldson-Uhlenbeck-Yau Theorem, which established the Hitchin-Kobayashi Correspondence between stable bundles and Einstein-Hermitian bundles, one naturally would like compare it to the geometric invariant theory (GIT) (see e.g. \cite{PT04, RT07} etc.). This may make algebraic geometers feel more comfortable with the concept. However, an apparently similar nature to the asymptotic GIT stability also makes one daunted, since the latter is notoriously known to be hard to check. Later people found examples which are K-stable but not asymptotically GIT stable varieties (see \cite{OSY12}), which made the picture even less clear. 

\smallskip

However,  it is remarkable that  in \cite{Oda13},  concepts and technicals from the minimal model program were first noticed to be closely related to K-stability question. In particular, it was shown K-semistability of a Fano variety implies at worst it only has Kawamata log terminal (klt) singularities, which is a measure of singularities invented in the minimal model program (MMP) theory.  Then by running a meticulously designed MMP process, in \cite{LX14} we show to study K-stability of Fano varieties, one only needs to consider $\cX/\bA^1$ where the special fiber over $0\in \bA^1$ is also a klt Fano variety. These works make it clear that to study K-stability of Fano varieties, the MMP would play a prominent role. In fact, as a refinement of Tian's original perspective, \cite{Oda13} and \cite{LX14} suggested that in the study of K-stability of Fano varieties, we should focus on Kawamata log terminal (klt) Fano varieties, since this class of varieties is equipped with the necessary compactness. The class of klt Fano varieties is a theme that has been investigated in higher dimensional geometry for decades, however such compactness was not noticed before by birational geometers. It was then foreseeable that not far from the future, there would be an intensive interplay between two originally disconnected subjects, and a purely  \emph{algebraic K-stability theory} would be given a birth as a result.

\smallskip

Nevertheless, such a view only fully arose after a series of intertwining works (contributed by K. Fujita, C. Li and others) which established a number of equivalent characterizations of notations in K-stability, including the ones using invariants defined on valuations (instead of the original definition using one parameter subgroup degenerations $\cX/\bA^1$). More specificly, in \cite{Ber16}, inspired by the analytic work \cite{Din88}, Berman realized that one can replace the generalized Futaki invariants by Ding invariants and define the corresponding notions of Ding stability. Soon after that, it was noticed in \cite{BBJ15, Fuj19b} that one can apply the argument in \cite{LX14} for Ding invariants, and conclude that Ding stability and K-stability are indeed the same for Fano varieties. A key technical advantage for Ding invariant, observed by Fujita in \cite{Fuj18}, is that we can define $\Ding(\cF)$ more generally for a linearly bounded multiplicative filtration $\cF$ on the anti-pluri-canonical ring $R=\oplus_{m\in \mathbb N} H^0(-mrK_X)$ which satisfies that if we approximate a filtration by its finitely generated $m$-truncation $\cF_m$, then $\lim_m\Ding(\cF_m)\to \Ding (\cF)$ (this is not known for the generalized Futaki invariants!). Therefore, Ding semistability implies the non-negativity of any $\Ding(\cF)$.
Based on this, a key invariant for the later development  was independently formulated by  Fujita and Li (see \cite{Fuj19b,Li18}), namely $\beta(E)$ for any divisorial valuation $E$ over $K(X)$. They also show that there is a close relation between $\Ding$ and $\beta$, and conclude for instance $X$ is Ding-semistable if and only if $\beta(E)\ge 0$ for any divisorial valuation.   Another remarkable conceptual progress is the definition of the normalized volumes  in \cite{Li18} which provides the right framework to study the local K-stability theory for klt singularities. As a consequence, one can investigate the local-to-global principle for K-stability as in many other higher dimensional geometry problems.   In a different direction, a uniform version of the definition of K-stability was introduced in \cite{BHJ17} and \cite{Der16} independently, which is more natural when one consider the space of all valuations instead of only divisorial valuations (see \cite{BJ20, BX19}).   With all these progress on our foundational understanding of K-stability, we can then turn to study various questions on K-stable Fano varieties using purely algebraic geometry.  

\smallskip
 
One of the most important reasons for algebraic geometers' interests in K-stability of Fano varieties is the possibility of using it to construct a well behaved moduli space, called {\it K-moduli}. 
Constructing moduli spaces of Fano varieties once seemed to be out of reach for higher dimensional geometers, as one primary reason is that degenerations of a family of Fano varieties are often quite complicated. Nevertheless,  it is not completely clueless to believe that by adding the K-stability assumption, one can overcome the difficulty. In fact, explicit examples of K-moduli spaces, especially parametrizing families of surfaces, appeared in \cite{Tia90, MM93, OSS16}.  Then general compact moduli spaces of smoothable K-polystable Fano varieties were first abstractly constructed in \cite{LWX19} (also see partial results in \cite{SSY16, Oda15}). However, the proofs of all these heavily rely on analytics input e.g.  \cites{DS14, CDS,Tia15}. To proceed, naturally people were trying to construct the moduli space purely algebraically, and therefore remove the `smoothable' assumption. However, this only became plausible after the progress of the foundation theory of K-stability described in the previous paragraph. Indeed,  together with major results from birational geometry (e.g. \cite{BCHM10, HMX14, Bir19}) and general moduli theory (e.g. \cite{AHH18}), by now we have successfully provided an algebraic construction of the moduli space and established a number of fundamental properties, though there are still some missing ones. More precisely, in \cite{Jia17} it was shown that the boundedness of all K-semistable Fano varieties with a fixed dimension $n$ and volume $V$ can be deduced from \cite{Bir19}. After the work of \cite{XZ20}, the same boundedness can be also concluded using a weaker result proved in \cite{HMX14}. Then in \cite{BLX19, Xu19}, two proofs of the openness of K-semistability in a family of Fano varieties  are given, both of which rely on the boundedness of complements proved in \cite{Bir19}. Putting this together, it implies that there exists an Artin stack $\mathfrak{X}^{\rm Kss}_{n,V}$ of finite type, called K-moduli stack, which parametrizes all $n$-dimensional K-semistable Fano varieties with the volume $V$. Next, we want to proceed to show $\mathfrak{X}^{\rm Kss}_{n,V}$ admits a good moduli space (in the sense of \cite{Alp13}). This was done in a trilogy of works: It was first realized in \cite{LWX18} that  for the closure of a single orbit, the various notions in K-stability share the same nature as the intrinsic geometry in GIT stability.  By further developing the idea in \cite{LWX18}, in \cite{BX19} we prove that the K-semistable degeneration of a family of Fano varieties is unique up to the S-equivalence, which amounts to saying that the good moduli space, if exists is separated. Based on this, in \cite{ABHX19}, applying the general theory developed in \cite{AHH18}, we eventually prove the good moduli space $\mathfrak{X}^{\rm Kss}_{n,V}\to {X}^{\rm Kps}_{n,V}$ does exist as a separated algebraic space. The major remaining challenge is to show the properness of ${X}^{\rm Kps}_{n,V}$. In \cite{CP18, XZ19}, it is also shown that CM line bundle is ample on ${X}^{\rm Kps}_{n,V}$, provided it is proper and an affirmative answer to the conjecture that K-polystability is equivalent to the reduced uniformly K-stability.  

\medskip

Another major question of K-stability theory is to verify it for explicit examples, which has been intensively studied, started from the very beginning when people were searching for K\"ahler-Einstein Fano manifolds. A famous sufficient condition found by Tian in \cite{Tia87} is $\alpha(X)>\frac{n}{n+1}$ where $\alpha(X)$ is the alpha-invariant (see \cite{OS12} for an algebraic treatment).  This criterion and its variants have been applied for a long time by people to verify K-stability, although there are many cases which people expect to be K-stable but have an alpha-invariant not bigger than $\frac{n}{n+1}$. After the new equivalent characterizations of K-stability were established, we can define the $\delta$-invariant $\delta(X)$, which satisfies $\delta(X)\ge 1$ (resp. $>1$) {\it if and only if} $X$ is K-semistable (resp. uniformly K-stable) (cf. \cite{FO18, BJ20}), where $\delta(X)$ can be calculated as  the limit of the infima of the log canonical thresholds for $m$-basis type divisors. Calculating $\delta(X)$ has a somewhat similar nature with calculating $\alpha(X)$ but could be more difficult. However, it is also much more rewarding, as it carries the precise information about K-stability. Since then, many new examples have been verified by estimating $\delta(X)$.   
%A first work along this line is in \cite{Fuj19}, where Fujita used $\delta(X)$ to show for Fano manifolds $X$ of dimension at least 2, if $\alpha(X)=\frac{n}{n+1}$, then it is K-stable. After \cite{SZ19} where a birationally superrigid Fano variety $X$ with $\alpha(X)>\frac{1}{2}$ is shown to be K-stable, in a recent work \cite{AZ20}, a remarkable induction strategy has been invented by using inversion of adjunction. As an application, all smooth Fano index 2 hypersurfaces of dimension at least 3 are shown to be K-stable. 
Another approach is using the moduli space to continuously identify the K-(semi,poly)stable Fano varieties from the deformation and degeneration of one that we know to be K-(sem-poly)stable. While this approach was implicitly contained in \cite{Tia90} and first explicitly appeared in \cite{MM93}, it became more powerful only after combining with the recent progress, especially explicit estimates of the normalized volume of singularities and a connection between local and global stability.

\begin{rem}There is a huge body of complex geometry study on this topic. We deliberately avoid any detailed discussion on them, except occasionally referring as a background. For readers who are interested, one could look at \cite{Sze18} for a recent survey. 

While we try to explain  various aspects of the recent progress on the algebro-geometric theory of K-stability of Fano varieties, the choice of the materials is of course based on the author's knowledge and taste.  
\end{rem}

\subsection{Organization of the paper} The paper is divided into three parts. 

In Part \ref{p-what}, we will discuss algebraic geometers' gradually evolved understanding of K-stability. As we mentioned, although in \cite{Don02}, the formulation was already algebraic, the more recent equivalent characterization using valuations turns out to fit much better into higher dimensional geometry. Therefore, we focus on explaining this new characterization of K-stability. In Section \ref{s-degeneration}, we will first briefly review the definition of K-stability given in \cite{Tia97, Don02}, then we will discuss the main result on special degenerations in \cite{LX14}. In the rest of Part \ref{p-what}, we will concentrate on proving the valuation criteria of  K-(semi)stability established by \cite{Fuj19b, Li17} and others, as well as introduce more variants of the notion of K-stability. In Section \ref{s-twoinvariants}, we will first introduce Fujita's $\beta$-invariant. It is then easy to deduce from the special degeneration theorem discussed in Section \ref{s-degeneration} that the positivity (resp. non-negativity) of $\beta$ implies K-stability (resp. K-semistability). We will also discuss two equivalent definitions of $\delta(X)$ for a Fano variety $X$ from \cite{FO18, BJ20}, which is an invariant precisely telling whether a given Fano is K-(semi)stable. To finish the converse direction that K-stability (resp. K-semistability) implies $\beta> 0$ (resp. $\beta\ge0$), we will present two approaches. First, in Section \ref{s-specialvaluation}, we follow the approach in \cite{BLX19}, which develops a theory on special divisors, corresponding to special degenerations. Then in Section \ref{s-Ding}, we discuss Fujita's work in \cite{Fuj18} of extending the definition of Ding-invariant for test configurations as in \cite{Ber16} to more general filtrations. This more general situation contains the filtration induced by a valuation as a special case, and then one just needs to compare Ding invariants and $\beta$ as in \cite{Li17, Fuj19b}. In  Section \ref{s-local}, we also discuss \cite{Li17,LX20}, which uses a concept introduced by Chi Li, called {\it the normalized volume} (see \cite{Li18}). This notion initiates a local stability theory on klt singularities. 

\medskip

In Part \ref{p-moduli}, we will focus on the program of constructing a projective scheme which parametrizes all K-polystable Fano varieties with the fixed numerical invariants, called {\it the K-moduli space  of Fano varieties}. 
The construction consists of several steps. In Section \ref{s-ksemi}, we will show there is a Artin stack which parametrizes all K-semistable Fano varieties with the fixed numerical invariants (see \cite{Jia17, BLX19, Xu19, XZ20}). Then in Section \ref{s-kpoly}, we show it admits a separated good moduli space (see \cite{BX19, ABHX19}). The properness of such good moduli space is still unknown, but by assuming that, we can essentially conclude the projectivity (see \cite{CP18, XZ19}). This is discussed in Section \ref{s-proj}. The proofs of these results interweave with our understanding of K-stability as discussed in Part \ref{p-what}, and also rely on some recent progress in algebraic geometry e.g. \cite{HMX14, Bir19, AHH18} etc.

\medskip

In Part \ref{p-example}, we will discuss how our new knowledge on K-stability as established in Part \ref{p-what} and Part \ref{p-moduli} can be used to get many new  examples of K-stable Fano varieties. K-stability can be verified either by studying the singularities of the $\mathbb{Q}$-linear system $|-K_X|_{\mathbb{Q}}$ or by establishing explicit  K-moduli spaces. Both of these two methods have older roots in works like \cite{Tia87,Tia90,MM93, OSS16} etc. Nevertheless, the recent progress provides us much stronger tools. In Section \ref{s-sing}, we will focus on how to estimate $\delta$-invariants in some explicit examples of Fano varieties, by following the works in \cite{Fuj19, SZ19, AZ20}. Then in Section \ref{s-moduli}, we will discuss how local estimates of the volume of singularities can be used to give explicit descriptions of K-moduli spaces (see \cite{LX19, Liu20,ADL19} etc). We will also discuss a wall crossing phenomenon of these moduli spaces in the log setting, as in \cite{ADL19}.

\subsection{Conventions} We will use the standard terminology of higher dimensional geometry, see e.g. \cite{KM98, Kol13, Laz04}.  A variety $X$ is {\it $\mathbb{Q}$-Fano} if it is projective, has klt singularities, and $-K_X$ is ample. 
A pair $(X,\Delta)$ is {\it log Fano} if $X$ is projective, $(X,\Delta)$ is klt, and $-K_X-\Delta$ is ample. 

\bigskip

\noindent{\bf Acknowledgement: } We are grateful for helpful conversations with Harold Blum, Ivan Cheltsov, Chi Li, Yuchen Liu, Xiaowei Wang, Chuyu Zhou and Ziquan Zhuang. A number of arguments in this note that are different with the original ones in the published papers, are indeed out of discussions with Harold Blum, Yuchen Liu or Ziquan Zhuang, which we owe our special thanks to. 
Part of the survey was written while the author enjoyed the hospitality of the MSRI, which is gratefully acknowledged. The survey was  used as the note for my class on 2020 Fall at Princeton University, and I would like to thank the participants of the class. 
\clearpage

\part{What is K-stability?}\label{p-what}
Unlike smooth projective varieties which are canonically polarized or Calabi-Yau type, Fano manifolds do not necessarily have a K\"ahler-Einstein (KE) metric. 
It had been speculated for a long time that the existence of KE metric on a Fano manifold should be equivalent to certain algebraic stability. After searching for a few decades, the concept of {\it K-stability} was eventually invented in \cite{Tia97} and reformulated in algebraic terms in \cite{Don02}. Since then the Yau-Tian-Donaldson Conjecture, which predicts that the existence of a KE metric on a Fano manifold $X$ is equivalent to $X$ being K-polystable, prevailed in complex geometry. Eventually, the solution was published in \cite{CDS, Tia15}, though the corresponding version for singular Fano varieties is still open, and has attracted lots of  recent interests. Actually, there is a new approach, called {\it the variational approach}, that aims to solve the singular case. The variational approach is probably conceptually closer to the algebraic geometry, because it is tightly related to non-achimedean geometry (see \cite{BoJ18}). Important progress along this line has recently been made in \cite{BBJ15, LTW19, Li19} (see Remark \ref{r-YTDnew}). 

In this part, we will discuss our gradually evolved understanding of the notions of K-stability of Fano varieties,  using only algebraic terms. We started with the original definitions introduced in \cite{Tia97, Don02}, by considering all $\mathbb{C}^*$-degenerations of a Fano variety $X$. Then we will explain that, one can apply the MMP to put a strong restriction on the allowed degenerations, namely we only need to consider $\mathbb{Q}$-Fano degenerations. 
With all these preparations, we will introduce a new (but equivalent) criterion of K-stability of Fano varieties using valuations over $X$. This is the central topic in Part \ref{p-what}. In Appendix \ref{s-local}, we will also briefly discuss a study of unexpectedly deep properties of klt singularities, which can be considered as a local model of K-stability theory for Fano varieties. 

\section{Definition of K-stability by degenerations and MMP}\label{s-degeneration}

\subsection{One parameter group degeneration}\label{ss-definition}

In this section, we will introduce the original definition of K-stability in \cite{Tia97, Don02}. It was related to the geometric invariant theory (GIT) stability, or more precisely the asymptotic version. See Remark \ref{r-GITvsK}. In GIT theory, by Hilbert-Mumford criterion, we know to test GIT stability, it suffices to compute the weight of the linearization on all possible one parameter subgroup degenerations. 

Here we first consider an abstract  one parameter subgroup degeneration, which is called a {\it test configuration}.
\begin{defn}
Let $X$ be an $n$-dimensional normal $\mathbb{Q}$-Gorenstein variety such that $-K_X$ is ample. Assume that $-rK_X$ is
Cartier for some fixed $r\in \mathbb{N}$. A test configuration of $(X, -rK_X)$ consists of
\begin{enumerate}
\item[ $\cdot$]
a variety $\cX$ with a $\mathbb{G}_m$-action,
\item[ $\cdot$]
a $\mathbb{G}_m$-equivariant ample line bundle $\cL\rightarrow
\cX$,
\item[ $\cdot$]
a flat $\mathbb{G}_m$-equivariant map $\pi: (\cX,
\mathcal{L})\rightarrow \mathbb{A}^1$, where $\mathbb{G}_m$
acts on $\mathbb{A}^1$ by multiplication in the standard way
$(t,a)\to ta$,
\end{enumerate}
such that over $A^{1}\setminus\{0\}$, there is an isomorphism 
$$\phi\colon (\cX, \cL)\times_{\mathbb{A}^1} (\mathbb{A}^1\setminus\{0\})\to (X, -rK_X)\times (\mathbb{A}^1\setminus\{0\}).$$
\end{defn}

%\begin{rem}
%In the following, by the abuse of notation, if we do not want to specify the exponent $r$, we will
%call $(\mX, \mL)$ a test configuration for both cases in the above definition.  Therefore by rescaling the polarization, any test configuration can be considered obtained from a $\mathbb{Q}$-test configuration of $(X, -K_X)$.
%\end{rem}

For any $\mathbb{Q}$-test configuration, we can define {\it the generalized Futaki
invariant} as follows. First   for a sufficiently divisible $k\in\mathbb{N}$, we have
\[
d_k=\dim H^0(X, \mathcal{O}_X(-kK_{X}))=a_0k^n+a_1k^{n-1}+O(k^{n-2})
\]
for some rational numbers $a_0$ and $a_1$. Let
$(\cX_0,\cL_0)$ be the restriction of $(\cX,
\cL)$ over $\{0\}$. Since $\mathbb{G}_m$ acts on $(\cX_0,\cL^{{\rm tc}\otimes k/r}_0)$, it also acts on
$H^0(\cX_0,\cL^{{\rm tc}\otimes k/r}_0)$. We denote the total
weight of this action by $w_k$. By the equivariant Riemann-Roch
Theorem,
\[
w_k=b_0k^{n+1}+b_1k^{n}+O(k^{n-1}).
\]
So we can expand
\[
\frac{w_k}{kd_k}=F_0+F_1k^{-1}+O(k^{-2}).
\]
\begin{defn}\label{d-futaki}
Under the above notion, the generalized Futaki invariant of the test configuration $(\cX, \cL)$ is defined to be
\begin{equation}\label{deffut}
\Fut(\cX,\cL)=-{F_1}=\frac{a_1b_0-a_0b_1}{a_0^2}
\end{equation}
\end{defn}

We easily see for any  $a\in \mathbb{N}$, 
$\Fut(\cX,\cL^{{\rm tc}\otimes a})= \Fut(\cX,\cL)$, therefore when $\cL$ is only a $\mathbb{Q}$-line bundle, we can still define
$\Fut(\cX,\cL):=\Fut(\cX,\cL^{{\rm tc} \otimes a})$ for some sufficiently divisible $a$.

\begin{rem}\label{r-tcvsstc}
Test configurations were first introduced in \cite{Tia97}, where the special fiber was required to be normal, and the generalized Futaki invariant was defined in analytic terms. 

Later in \cite{Don02}, any degeneration was allowed (indeed instead of a Fano variety, \cite{Don02} considered test configurations of any given polarized projective variety), and the generalized Futaki invariant was defined in algebraic terms as above. Therefore, in some literature, the generalized Futaki invariant is also called {\it the Donaldson-Futaki invariant.} Since in our current note, we will mostly restrict ourselves to a even smaller class (see Definition \ref{d-stc}) than in Tian's setting, to avoid any confusion, we will only use the terminology in \cite{Tia97}. 

In different literatures,  the definition of the generalized Futaki invariant may differ by a (positive) constant. As we will see, in Definition \ref{d-kstable} of K-stability, it is only the sign of the generalized Futaki invariant that matters. %However, to compare it with other constants, later we will carefully track the normalization.  
\end{rem}

\begin{rem}\label{r-GITvsK}
The most important feature for testing K-stability is that we have to look at all  $r$. In other words, we need to consider higher and higher re-embeddings given by $|-rK_X|$ and all their flat degenerations under  one parameter group.   This is similar to the asymptotic Chow stability. See \cite[Section 3]{RT07} for a detailed discussion, especially for the implications among different notions of stability. 

Nevertheless, the directions which were not addressed in \cite{RT07} were more subtle. While it is proved in \cite{Don02} that any polarized manifold $(X,L)$ with a cscK metric and finite automorphism group is asymptotically Chow stable, however, this is known to not hold for either $X$ has an infinite automorphism group (see \cite{OSY12}) or $X$ is singular (see \cite{Oda12}). 
%Although the connection between GIT and K-stability is an important topic, we will not pursue it. 
\end{rem} 

\begin{rem}The notions of K-stability can be defined for a log Fano pair $(X,\Delta)$ (See \cite[(30)]{Don12}). In fact, though for the purpose of making the exposition simpler, we will only discuss $\bQ$-Fano varieties, all the K-stability results we discussed in this survey can be generalized from a $\bQ$-Fano variety $X$ to a log Fano pair $(X,\Delta)$, and for most of the time the generalization is merely a book-keeping. %with a couple of exceptions in Section \ref{s-moduli}. %The readers who want to find the general version, can consult the original paper.
\end{rem}

\begin{defn}\label{d-kstable}
Let $X$ be an $n$-dimensional normal $\mathbb{Q}$-Gorenstein variety such that $-K_X$ is ample, then
\begin{enumerate}
\item
$X$ is K-semistable if for any test
configuration $(\cX, \cL)$ of $(X, -K_X)$, we
have $\Fut(\cX, \cL)\ge0$.
\item
$X$ is  K-stable (resp. K-polystable) if for any test configuration $(\cX, \cL)$ of $(X,
-K_X)$, we have $\Fut(\cX, \cL)\ge 0$, and
the equality holds only if $(\cX, \cL)$ is trivial (resp. only if
$\cX$ and $X\times \mathbb{A}^1$ are isomorphic) outside a codimension 2 locus on $\cX$.
\end{enumerate}
\end{defn}

\begin{exmp}Consider a smooth Fano variety $X$, such that there is an effective torus action $T(\cong \mathbb{G}_m^r)$ on $X$. Then for any (integral) coweight $\mathbb{G}_m\to T$, we can define a  test configuration $ \mathcal{X}\cong X\times \mathbb{A}^1$ and $\mathcal{L}\cong  -K_{X\times \mathbb{A}^1} $ with the $\mathbb{G}_m$-action given by 
\[
t\cdot (x,a)\to (t(x), t\cdot a).
\]
This kind of test configuration is called {\it a product test configuration}. Since if we reverse the action of $\mathbb{G}_m$, the total weight will change the sign, we conclude that if $X$ is K-semistable, then ${\rm Fut}( \mathcal{X},\mathcal{L})=0$ for all product test configurations. This condition was first introduced in \cite{Fut83}. 
\end{exmp}

There is an intersection formula description of the generalized Futaki invariants (see \cite{Wan12, Oda13b}) for any given test configuration $(\cX, \cL)$.
\begin{lem}[{\cite{Wan12, Oda13b}}]\label{l-intersection} 
Assume a $(\cX,\cL)$ is a normal test configuration. If we glue $(\cX, \cL)$ with $(X
\times (\mathbb{P}^1\setminus \{0\}) ,p_1^*(-rK_{X}))$ over $\mathbb{A}^1\setminus \{0\}$ by $\phi$ to get a proper family $(\bar{\cX}, \bar{\cL})$ over $\mathbb{P}^1$,
then we have the following equality:
\begin{eqnarray}\label{e-inter}
 \Fut(\cX,{\cL})&= &\frac{1}{2(n+1)(-K_X)^n}\left(n (\frac{1}{r}\bar{\cL})^{n+1}+(n+1)K_{\bar{\cX}/\mathbb{P}^1}\cdot(\frac{1}{r}\bar{\cL})^{n}\right)
\end{eqnarray}
\end{lem}
\begin{proof}See e.g. \cite[Page 224-225]{LX14}.
\end{proof}

\subsection{MMP on family of Fano varieties}\label{ss-mmp}
In this section, we will start to uncover the connection between K-stability and MMP. 

The following theorem proved in \cite{Oda13} is the first intersection of K-stability theory and the minimal model program in algebraic geometry.
\begin{thm}[{\cite{Oda13}}]Let $X$ be an $n$-dimensional normal $\mathbb{Q}$-Gorenstein variety such that $-K_X$ is ample. If $X$ is K-semistable, then $X$ has (at worst) klt singularities, i.e., $X$ is a $\bQ$-Fano variety.
\end{thm}
The proof of the above theorem is a combination of \eqref{e-inter} and a MMP construction called {\it the lc modification} whose existence relies on the relative minimal model program (see \cite{OX12}).

It was probably not a big surprise that the notion of K-stability should have some restriction on the singularities, however, it was really a remarkable observation that the right category of singularities should be the one from the minimal model program theory.  From now on, we will only consider the K-stability problem for $\mathbb{Q}$-Fano varieties. 

\begin{rem}It is natural to ask whether we can restrict ourselves to an even smaller category of singularities than klt singularities. The answer is likely to be negative (see Section \ref{s-local}). However, the global invariant of the volume $(-K_X)^n$ for the K-semistable Fano variety will post more restrictive conditions on the possible local singularities (see Theorem \ref{t-localtoglobal}). 
\end{rem}

Now we introduce a smaller class of test configurations called {\it special test configuration}, which will play a crucial role in our study.

\begin{defn}[Special test configurations]
\label{d-stc}
A test configuration $(\mathcal{X},\mathcal{L})$ of $(X,-rK_{X})$ is called {\it a special test configuration} if $\mathcal{L}\sim_{\mathbb{Q}}-rK_{\cX}$ and the special fiber $X_0$ is a $\mathbb{Q}$-Fano variety. By inversion of adjunction, this is equivalent to saying $\mathcal{X}$ is $\mathbb{Q}$-Gorenstein and $-K_{\mathcal{X} }$ is ample and $(\mathcal{X},X_0)$ is plt. 

We also call a test configuration $(\cX,\cL)$ satisfying that $(\cX,\cX_0)$ is log canonical and $\cL\sim_{\mathbb{Q}}-rK_{\cX}$ to be a {\it weakly special test configuration}.
\end{defn}

The next theorem shows the difference in the definition of K-stability for Fano varieties in \cite{Tia97} and in \cite{Don02} (see Remark \ref{r-tcvsstc})  does not really play any role.
%Historically, there was indeed a difference between the definitions of K-stability for Fano varieties in \cite{Tia97} and in \cite{Don02}: in \cite{Tia97} only the degeneration with a normal special fiber $X_0$ is considered, while in \cite{Don02} all degenerations are considered. However, we will show this difference does not really play any role, by proving that we can indeed look at the even smaller class of special test configurations. 

%Although in retrospective, MMP naturally leads to consider the construction of special test configurations

\begin{thm}[{\cite{LX14}}]\label{t-specialdegeneration}
Let $(\mathcal{X},\mathcal{L})\to \mathbb{A}^1$ be a test configuration $(X,-rK_X)$, then there exists a special test configuration $(\mathcal{X}^{{\rm st}},\mathcal{L}^{{\rm st}})\to \mathbb{A}^1$ which is birational to $(\mathcal{X},\mathcal{L})\times_{\mathbb{A}^1, z\to z^d} \mathbb{A}^1$ over $\mathbb{A}^1$, such that $\Fut(\mathcal{X}^{{\rm st}},\mathcal{L}^{{\rm st}})\le d\cdot \Fut(\mathcal{X},\mathcal{L})$.

Moreover, the equality holds if and only if the birational map 
$$(\mathcal{X}^{{\rm st}},\mathcal{L}^{{\rm st}}) \dasharrow (\mathcal{X},\mathcal{L})\times_{\mathbb{A}^1, z\to z^d} \mathbb{A}^1$$ is an isomorphism outside codimension 2.
\end{thm}
\begin{proof}[Sketch of the proof]Started from any test configuration $\cX$, we will use the the minimal model program to modify $\cX$ such that at the end we obtain a special test configuration $\cX^{\rm s}$, and during the process the generalized Futaki invariants decrease. The modification consists of a few steps.

\noindent {\it Step 0:} If we replace $\cX$ by the normalization $n\colon \cX^n\to \cX\times_{z\to z^d}\mathbb{A}^1$, then we have 
$$d\cdot \Fut(\cX,\cL )\ge \Fut(\cX^n,\cL^n:=n^*\cL)$$ with the equality holds if and only if $\cX^n\to\cX\times_{z\to z^d}\mathbb{A}^1$ is an isomorphism outside codimension at least 2. This can be seen by directly applying the intersection formula \eqref{e-inter}.

\begin{keyidea}After Step 0, we can always assume the special fiber is reduced. The following steps will all involve minimal model program constructions. The main observation is the following calculation: denote by $\cL$ the polarization on $\cX$, such that $\cL|_{X_t}\sim -rK_{X_t}$ for $t\neq 0$, so we can write $\cL+rK_{X_t}\sim E$ which is a divisor supported over $0$. Then let $t>0$ such that $\cL_t=\cL+tE$ is still ample, then applying \eqref{e-inter}, we have
\begin{eqnarray}\label{e-decreasing}
\frac{d}{dt}\Fut(\cX,\cL_t) =\frac{n}{2(-K_X)^n}(\frac{1}{r}\bar{\cL}_t)^{n-1}\cdot (\frac{1}{r}E)^2\le 0\ \ .
\end{eqnarray}
\end{keyidea}

\noindent {\it Step 1:} From $\cX^n$, we can construct the log canonical modification of $f^{\rm lc}\colon \cX^{\rm lc}\to (\cX,X^n_0)$ (see \cite[Theorem 1.32]{Kol13}), where $X^n_0$ is the special fiber. By a suitable base change, we can assume that the special fiber $X^{\rm lc}_0$ of $\cX^{\rm lc}$ is also reduced. Let $F$ be the reduced exceptional divisor. Then by the definition of the lc modification, 
$$K_{\cX^{\rm lc}}+f^{\rm lc}_*(X^n_0)+F=K_{\cX^{\rm lc}}+X^{\rm lc}_0\sim K_{\cX^{\rm lc}}$$
 is ample over $\cX^n$.   So $E=\frac{1}{r}f^{{\rm lc}*}\cL^n+K_{\cX^{\rm lc }}$ is ample over $X^n$ and $\cL^{\rm lc}_t:=\frac{1}{r}f^{{\rm lc}*}\cL^n+tE$ is ample  for some $0<t\ll 1$. Therefore, \eqref{e-decreasing} implies
$ \Fut(\cX^{\rm lc},\cL^{\rm lc}_t)\le \Fut(\cX^n,\cL^n),$
and the equality holds if and only if $(\cX^n,X^n_0)$ is log canonical. 
\medskip

\noindent {\it Step 2:} Replacing $\cL^{\rm lc}$ by its power, we can assume that $H:=\cL^{lc}-K_{\cX^{\rm lc}}$ is ample. Then we run $K_{\cX^{\rm lc}}$-MMP with the scaling of $H$ (see \cite{BCHM10}), which is automatically $\mathbb{G}_m$-equivariant in each step. Thus we get a sequence of numbers $t_0=1>t_1\ge t_2\ge...\ge t_{m-1}> t_m=\frac{1}{r}$, with a sequence of models
\[
\cX^{\rm lc}=Y_0\dasharrow Y_1\dasharrow \cdots \dasharrow Y_{m-1}
\]
such that if we let $H_i$ be the pushforward of $H$ on $Y_i$, $K_{Y_i}+sH_i$ is nef for any $s\in [t_{i-1},t_i]$. Moreover, we have $K_{Y_{m-1}}+t_mH_{m-1}\sim _{\mathbb{Q}} 0$. Thus 
\[
K_{Y_{m-1}}+t_{m-1}H_{m-1}\sim_{\mathbb{Q}}(t_{m-1}-t_m)H_{m-1}
\]
is big and nef. Let $X^{\rm an}$ be the ample model of $H_{m-1}$ and $\cL^{\rm an}$ the ample $\mathbb{Q}$-divisor induced by $H_{m-1}$. If we use \eqref{e-inter} to define the generalized Futaki invariant  when $\cL$ is a big and nef line $\mathbb{Q}$-bundle, then \eqref{e-decreasing} implies 
\begin{eqnarray*}
\Fut(\cX^{\rm lc},\cL^{\rm lc})&=&\Fut(Y_0, K_{Y_0}+H)\\
&\ge &\Fut(Y_0, K_{Y_0}+t_1H_0)\\
&=& \Fut(Y_1, K_{Y_1}+t_1H_1)\\
&\ge& \cdots\\
&=& \Fut(Y_{m-1}, K_{Y_1}+t_{m-1}H_{m-1})\\
&=& \Fut(\cX^{\rm an}, \cL^{\rm an}),
\end{eqnarray*}
 and the equality holds if and only if $\cX^{\rm lc}=\cX^{\rm an}$. We note that since $-K_{\cX^{\rm an}}$ is proportional to $\cL^{\rm an}$, we indeed have 
 \begin{eqnarray}\label{e-anti}
  \Fut(\cX^{\rm an}, \cL^{\rm an})=-\frac{1}{2(n+1)(-K_X)^n}\left(K_{\bar{\cX}^{\rm an}/\mathbb{P}^1}\right)^{n+1}.
\end{eqnarray}
\medskip
(In particular, $(\cX^{\rm an}, \cL^{\rm an})$ is a weakly special test configuration.)

\noindent {\it Step 3:} In the last step, by a tie-breaking argument, we can show after a possible base change, by running a suitable minimal model program, we can construct a model such that $(\cX^{\rm s}, X^{\rm s}_0)$ is plt and the discrepancy of $X^{\rm s}_0$ with respect to $(\cX^{\rm an},X_0^{\rm an})$ is $-1$. By an intersection number calculation, we have 
$$ -\frac{1}{2(n+1)(-K_X)^n}\left(K_{\bar{\cX}^{\rm an}/\mathbb{P}^1}\right)^{n+1}\ge -\frac{1}{2(n+1)(-K_X)^n}\left(K_{\bar{\cX}^{\rm s}/\mathbb{P}^1}\right)^{n+1}$$
and the equality holds if and only if $\cX^{\rm an}=\cX^{\rm s}.$
\end{proof}

\begin{say}While special degenerations are indeed quite special, however, even simple Fano varieties could have many special degenerations. An easy example is in Example \ref{ex-quadratic}. One could think of the stack $\kX^{\rm Fano}_{n,V}$ of all klt Fano varieties with fixed numerical invariants similar to the stack $\mathfrak{Sh}_f$ of all coherent sheaves with the fixed Hilbert polynomial. %See more discussion in Remark \ref{r-theta}.
\end{say}
\begin{exmp}\label{ex-quadratic}
The family $(x^2+y^2+z^2+tw^2=0)\subset \mathbb{P}^2\times \mathbb{A}^1$, gives a special degeneration of $\mathbb{P}^1\times \mathbb{P}^1$ to the cone over a conic curve. 
\end{exmp}
%We indeed have a characterization on which $\mathbb{Q}$-Fano varieties only have trivial specialization. 
\begin{exmp}A $\mathbb{Q}$-Fano does not have any nontrivial weakly special test configuration if and only if $X$ is exceptional, that is
$$\alpha(X):=\inf \{\lct(X,D)| \ D\sim_{\mathbb{Q}}-K_X\}> 1.$$
i.e., if we define $T(E)=\sup_D \{\ord_{E}(D)\   | \ D\sim_{\mathbb{Q}}-K_X\}$, then $ \frac{A_{X}(E)}{T(E)}> 1$ for all divisors $E$ over $X$ as $\alpha(X)=\inf_{E} \frac{A_{X}(E)}{T(E)}$. See Theorem \ref{t-weakspecial} for a proof.
\end{exmp}

However, it is known for a $\mathbb{Q}$-Fano variety $X$ %(resp. a Fano manifold $X$ with $\dim(X)\ge 2$), 
$$\alpha(X)>\frac{\dim(X)}{\dim(X)+1} %\ \  (\mbox{resp. } \alpha(X)\ge \frac{\dim(X)}{\dim(X)+1})
$$
is sufficient to imply that $X$ is K-stable (see Example \ref{e-alpha}). 

So it is clear that Theorem \ref{t-specialdegeneration} alone is not strong enough to verify the K-stability of a general Fano variety.

\section{Fujita-Li's valuative criterion of K-stability}\label{s-twoinvariants}

In this section, we will discuss the characterization of K-stability using valuations. The viewpoint of using valuations to reinterpret a one parameter group degeneration was introduced in \cite{BHJ17}, based on earlier works by \cite{Nys12, Sze15}. A key definition, made in a series of remarkable works \cite{Fuj18, Fuj19b, Li17}, is the invariant $\beta(E)$ for divisorial valuations $E$ which will play a prominent role in this survey, based on Theorem \ref{t-valkstable} which says that one can use them to precisely characterize various K-stability notions. This change of viewpoint will be our major topic in the rest discussion of  Part \ref{p-what}.

Let $X$ be a normal variety such that $K_X$ is $\mathbb{Q}$-Cartier, we define the {\it log discrepancy} for any divisor $E$ over $X$
$$A_{X}(E)=a(E,X)+1$$ where $a(E,X)$ is the discrepancy (see \cite[Definition 2.25]{KM98}). So $X$ being klt is equivalent to saying that $A_X$ is positive for any $E$.

\subsection{$\beta$-invariant}\label{ss-invariant}
The $\beta$-invariant $\beta_{X}(E)$  was first defined in \cite{Fuj19b, Li17}, after an earlier attempt in \cite{Fuj18}.
\begin{defn}\label{d-betadivisor}
Let $X$ be an $n$-dimensional   $\mathbb{Q}$-Fano variety and $E$ a divisor over $X$. 
 We define
\begin{equation}\label{e-beta}
\beta_{X}(E)= A_{X}(E) -\frac{1}{(-K_X)^n} \int_0^\infty  \vol(\mu^*(-K_X)- t E)\, dt,
\end{equation}
where $E$ arises a prime divisor on a proper normal model $\mu:Y \to X$. We also denote $\frac{1}{(-K_X)^n} \int_0^\infty  \vol(\mu^*(-K_X)- t E)$ by $S_X(E)$.
\end{defn}

When $X$ is clear we will omit the decoration to simply write $\beta(E)$ etc..

\medskip

The  importance of $\beta$-invariant can be seen from the following theorem.

\begin{thm}[The valuative criterion for K-(semi)stability, {\cite{Fuj19b, Li17}}]\label{t-valkstable}
A $\mathbb{Q}$-Fano variety $X$  is 
\begin{enumerate}
\item  \emph{K-semistable} if and only if $\beta_{X}(E)\ge 0$ for all divisors  $E$ over $X$;
\item (together with \cite{BX19}) \emph{K-stable} if and only if $\beta_{X}(E)> 0$ for all divisors  $E$ over $X$.
 \end{enumerate}
 \end{thm}

The rest of the section will be devoted to prove one direction of Theorem \ref{t-valkstable}. Then different proofs of another direction will be completed in Section \ref{s-specialvaluation}, Section \ref{s-local} and Section \ref{s-Ding}.  

\bigskip
 
 Consider a special test configuration $\cX$, and denote by its special fiber $X_0$. Then the restriction of $\ord_{X_0}$ on $K(\cX)=K(X\times \mathbb{A}^1)$ to $K(X)$ yields a valuation $v$. It is easy to see when $\cX$ is a trivial test configuration, then $v$ is trivial, we also have
\begin{lem}[{\cite[Lemma 4.1]{BHJ17}}]\label{l-induced}
If $\cX$ is not a trivial test configuration, then $v$ is a divisorial valuation, i.e. $v=c\cdot \ord_E$ for some $c\in \mathbb{Z}_{>0}$ and $E$ over $X$. 
\end{lem} 
\begin{proof}Since ${\rm tr. deg}(K(\cX)/K(X))=1$, by Abhyankar's inequality, we know that
\begin{eqnarray*}
& &{\rm tr. deg}(K(v))+ {\rm rank}_{\mathbb{Q}}(v)\\
&\ge &{\rm tr. deg}(K(\ord_{X_0}))+ {\rm rank}_{\mathbb{Q}}(\ord_{X_0})-1\\
&=&\dim (X).
\end{eqnarray*}
This implies $v$ is an Abhyankar valuation, whose value group is nontrivial and contained in $\mathbb{Z}$, which implies the assertion.  
\end{proof}

While the above lemma is very general and cannot be easily used to trace the correspondence geometrically, see Theorem \ref{t-weakspecial} for a much more precise characterization of $v$.

\begin{lem}\label{l-beta=w}
For a nontrivial special test configuration $\cX$ of a $\mathbb{Q}$-Fano variety $X$,  denote by $v$ the valuation defined as above. Then we have 
\[
2\cdot {\rm Fut}(\cX)= \beta_X(v):=c\cdot \beta_X(\ord_E).
\]
\end{lem}
\begin{proof}Consider a section $s\in H^0(-mK_X)$ for $m$ sufficiently divisible. Let $D_s$ be the closure of $(s)\times \mathbb{A}^1$ on $X\times \mathbb{A}^1$. Fix a common log resolution $\widehat{\cX}$ of $\cX$ and $X\times \mathbb{A}^1$ 
\begin{center}
\begin{tikzcd}
  & \widehat{\cX}  \arrow[dr, "\psi'"] \arrow[dl,swap,"\psi"]& \\
 X\times \mathbb{A}^1  \arrow[rr, dashrightarrow, "\phi'"] & & \cX.
   \end{tikzcd} 
    \end{center}
Denote by $X_0$ the special fiber of $\cX$.     So 
\begin{eqnarray}\label{e-weight}
\psi^*(D_s)=\widehat{D_s}+(\ord_{X_0}(\bar{s}))\widetilde{X_0}+E\in H^0(	-mK_{\widehat{\cX}}+m\cdot a(X_0,X\times \mathbb{A}^1)\widetilde{X}_0+F) 
\end{eqnarray} where $\widehat{D}_{s}$ and $\widetilde{X_0}$ are the birational transforms of $D_{s}$ and $X_0$ on $\widehat{\cX}$ and ${\rm Supp}(E)$ as well as ${\rm Supp}(F)$ supporting over $0$ do not contain the birational transform of $X\times\{0\}$ and $X_0$. We know 
$$a(X_0,X\times \mathbb{A}^1)=A(X_0,X\times \mathbb{A}^1)-1=c\cdot A_{X}(E ) \mbox{\ \ and \ \ }\ord_{X_0}(\bar{s})=c\cdot \ord_E(s),$$
so if we denote by $D_s'=\psi'_*\big(\widehat{D_s})$, pushforward \eqref{e-weight} under $\psi'$, it becomes 
$$D_s'+(c\cdot\ord_E(s))X_0\in H^0(-mK_{\cX}+mc\cdot A_X(E){X}_0).$$
 We choose  a basis $\{s_1,..., s_{N_m}\}$ $(N_m=\dim H^0(-mK_X))$ of $H^0(-mK_{X})$ which is compatible with the filtration, that is, if we let $j_i=\dim \cF^i(H^0(-mK_X))$, then 
$\{s_1,..., s_{j_i}\}$ form a basis of $\cF^i(H^0(-mK_X))$. The above computation says
the total weight $w_m$ of $\mathbb{G}_m$ on $H^0(-mK_{\cX})$ is
\[
 w_m=\sum_ic\cdot(j_{i}-j_{i+1})i-rmc,
\]
and if we divide by $\frac{m^{n+1}}{n!}$ and let $m\to \infty$, we have 
\[
\lim_{m\to \infty}\frac{w_m}{m^{n+1}/n!}=c\cdot \int^{\infty}_{0}\vol(-K_X-t\cdot E)dt-cA_X(E)(-K_X)^n=-c(-K_X)^n\beta_X(E),
\]
where we use
\[
\lim_m \frac{1}{m^{n+1} /n!}\sum_i(j_{i}-j_{i+1})i=\lim_m \frac{1}{m^{n+1} /n!}\sum_ij_i=\int^{\infty}_{0}\vol(-K_X-t\cdot E)dt 
\]
by the dominated convergence theorem. 

On the other hand, a simple Riemann-Roch calculation implies 
$$\lim_{m\to \infty}\frac{w_m}{m^{n+1}/n!}=\frac{1}{n+1}(-K_{\bar{\cX}/\mathbb{P}^1})^{n+1}=-2(-K_X)^n\cdot \Fut(\cX)$$ by \eqref{e-anti}.
 \end{proof}

Combining with Theorem \ref{t-specialdegeneration}, a direct consequence of Lemma \ref{l-beta=w} is one direction of Theorem \ref{t-valkstable}.  
\begin{cor}\label{c-kimpliesbeta}
If $\beta_{X}(E)\ge 0$ (resp. $\beta_{X}(E) >0$) for any divisor $E$ over $X$, then $X$ is K-semistable (resp. K-stable). 
\end{cor}

The above discussion can be extended using the following construction.

\begin{say}[Rees construction]\label{s-rees}
The following general construction is from \cite[\S 2]{BHJ17}:

Fix $r$ such that $-rK_X$ is Cartier. Let $v$ be a valuation on $X$ which yields a filtration $\cF:=\cF_v$ on $R(X)=\bigoplus_{m\in \bN}R_m=\bigoplus_{m\in \bN} H^0(-mrK_X)$ with 
\begin{equation}\label{e-VtoF}
\cF^{\la}_vR_m:=\{s\in H^0(-mrK_X) \ |\ v(s)\ge \la\}, \ \forall m\in \bN.
\end{equation}
Then we can construct the \emph{Rees algebra}  of $\cF$ 
\[
{\rm Rees}(\cF): = \bigoplus_{m \in \N}  \bigoplus_{p \in \Z} t^{-p} \cF^p R_m  \subseteq R[t,t^{-1}].\]
as a $k[t]$-algebra.

The \emph{associated graded ring} of $\cF$ is
\[
\gr_{\cF} R : =  \bigoplus_{m \in \N} \bigoplus_{p \in \Z} \gr_{\cF}^p R_m,
\quad\quad \text{ where } \gr_{\cF}^p R_m = \frac{\cF^{p}R_m}{\cF^{p+1}R_m}
.\]
Note that 
\begin{equation}\label{e-Rees}
{ \rm Rees}(\cF) \otimes_{k[t]} k[t,t^{-1}] \simeq R[t,t^{-1}] \quad 
\text{ and } 
\quad\quad
\frac{{\rm Rees}(\cF)}{ t\, {\rm Rees}(\cF)}  \simeq \gr_{\cF}R .\end{equation}

We say that an $\N$-filtration $\cF$ is \emph{finitely generated} if ${\rm Rees}(\cF)$ is a finitely generated $k[t]$-algebra.
 When $v$ is given by a divisor $E$, this finite generation condition is equivalent to the finite generation of
  the double graded ring
\begin{eqnarray}\label{e-doublegraded}
\bigoplus_{m,n\in \bN}H^0(-m\mu^*K_X-nE),
\end{eqnarray}
i.e. the a divisor is {\it dreamy} (see \cite[Definition 1.2(2)]{Fuj19c}). 
Assuming $\cF$ is finitely generated, we set  $\cX: = \Proj_{\mathbb{A}^1} \left({\rm Rees}(\cF)\right)$. 
By \eqref{e-Rees},
\[
\cX_{ \mathbb{A}^1\setminus \{0 \}} \simeq X \times ( \mathbb{A}^1\setminus \{0 \} ) 
\quad \text{ and }
\cX_0 \simeq \Proj ( \gr_{\cF} R ). 
\]
\end{say}

\begin{lem}\label{l-integraldreamy}
Under the correspondence between Lemma \ref{l-induced} and  \eqref{e-doublegraded}, there is one-to-one correspondence 
 \[
\left\{
  \begin{tabular}{c}
\mbox{Nontrivial test configurations $\cX$}\\
of $X$ with an integral special fiber
  \end{tabular}
\right\} 
\ \longleftrightarrow
 \left\{
  \begin{tabular}{c}
\mbox{a divisorial valuation $a\cdot\ord_E$}\\
\mbox{with dreamy $E$ and $a\in \mathbb{N}_{+}$}\\
  \end{tabular}
\right\}.
\]
Moreover, we have
$a\cdot \beta(E)=2 \cdot \Fut(\cX)$.
\end{lem}
\begin{proof}If $ E$ is a dreamy divisor, then it is clear the Rees construction  ${\rm Rees}(\cF_v)$ for $v=a\cdot \ord_E$ is finitely generated. And the fiber is integral, since the associated graded ring of any valuation is integral. 

To see the last statement, when the special fiber of a test configuration $\cX$ is integral, then $\cL\sim_{\mathbb Q}-K_{\cX}$. Thus the calculation in Lemma \ref{l-induced} 
can be verbatim extended to this setting. 
\end{proof}

We also make the following definitions which give smaller classes of divisorial valuations than being dreamy. 
 \begin{defn}\label{d-specialdivisor}
 A divisorial valuation $E$ over a $\mathbb{Q}$-Fano variety is called {\it special} (resp. weakly special) if there is a special test configuration $\cX$ (resp. weakly special test configuration $\cX$ with an integral fiber), such that $\ord_{X_0}|_{K(X)}=a\cdot \ord_E$.  
 \end{defn}

\begin{say}[The reverse direction]\label{say-betaimpliesK}

We still need to prove that the K-(semi)-stability implies the (semi)-positivity of $\beta$. 

Note that the main reason we can prove Corollary \ref{c-kimpliesbeta} is that instead of considering all test configurations, which are not easy to connect to $\beta$-invariants, Theorem \ref{t-specialdegeneration} allows us to only look at special test configurations. 
The difficulty for the reverse direction then lies on the fact that not every divisor $E$ over $X$ arises from a special test configuration, and more generally it is hard to precisely characterize which divisors are dreamy.  In our note, we will present different proofs of the reverse direction in Section \ref{s-specialvaluation} and Section \ref{s-Ding}. 

The first proof was independently given in \cite{Fuj19b} and \cite{Li17}.  There they use the notion of Ding invariant (first introduced in \cite{Ber16})  to define notions of Ding-stability. %The major advantage of Ding invariants is that it can be sensibly extended to define for any general bounded multiplicative filtration, as in \cite{Fuj18}. 
Then one can use MMP as in Section \ref{ss-mmp} but for Ding-invariants to show the notions of Ding stability is the same as the K-stability ones, since they are the same on special test configurations.  One major advantage of using Ding invariant is that its definition can be sensibly extended to any general bounded multiplicative filtration, which contain the setting of both test configurations and valuations,  by establishing an approximation result (proved in \cite{Fuj18}). See Theorem \ref{t-maintheorem2}. We will postpone the detailed discussions in Section \ref{s-Ding}.

%By using a limit process, it was shown that $\beta\ge 0$ is implied by Ding-semistability (introduced in \cite{Ber16}). In the definition, Ding-semistability is stronger than K-semistability, however, they coincides on special test configurations, and it was shown in \cite{BBJ15} and \cite{Fuj19b} that a similar process as in Theorem \ref{t-specialdegeneration} can be applied to Ding-semistability to prove that to test Ding-semistability we only need to test on special test configurations, therefore it is indeed the same as K-semistability.   We will postpone the discussion of this to \cite{s-Ding}.

The second proof, which was first found in \cite{LX20} indeed reduces everything to a `special' setting. This is not straightforward for $\beta$, however, we follow \cite{Li17} and interpret $\beta$ as the derivative of the normalized volume function over the cone singularity. In Section \ref{defn-complements}, following \cite{BLX19} we will present an argument only using global invariants (see  especially Theorem \ref{t-kimpliesbeta}).
In Appendix \ref{s-local}, we discuss the proof in \cite{LX20} using normalized volumes.
\end{say}

\begin{exmp}[Boundedness of volume, {\cite{Fuj18}}]\label{e-bounded} 
One of the first striking consequences of the valuative criterion Theorem \ref{t-valkstable} is the following statement proved by Fujita in \cite{Fuj18}: The volume of an $n$-dimensional K-semistable $\mathbb{Q}$-Fano variety is at most $(n+1)^n$. 

Let $X$ be a K-semistable $\mathbb{Q}$-Fano variety. Pick up a smooth point $x\in X$, and blow up $x$, we get $\mu\colon Y\to X$ with an exceptional divisor $E$. 
Since 
\[
0\to \fm_x^k\otimes \mathcal{O}_X(-mK_X)\to \mathcal{O}_X(-mK_X) \to  \mathcal{O}_X(-mK_X)\otimes (\mathcal{O}_X/ \fm_x^k)\to 0,
\]
we know that 
$$h^0(\mathcal{O}_Y(\mu^*(-mK_X)-kE))\ge h^0(\mathcal{O}_X(-mK_X))-h^0(\mathcal{O}_X(-mK_X)\otimes (\mathcal{O}_X/ \fm_x^k).$$
Thus we have
\begin{eqnarray*}
0\le (-K_X)^n\cdot\beta(E)&=&n\cdot (-K_X)^n-\int^{\infty}_{0}\vol(\mu^*(-K_X)-tE)dt\\
&\le &n\cdot (-K_X)^n-\int^{((-K_X)^n)^{\frac1n}}_{0}((-K_X)^n-t^n),
\end{eqnarray*}
which immediately implies $(-K_X)^n\le (n+1)^n. $

If $X$ is singular, and we choose $x$ to be a singularity, then an even stronger restriction on the volume in terms of the local volume of the singularity  is obtained in \cite{Liu18}. See Theorem \ref{t-localtoglobal}. 
\end{exmp}

\subsection{Stability threshold}\label{sss-delta}
It is also natural to consider a variant, 
$$
\delta_X(E)=_{\rm defn}\frac{A_X(E)}{S_X(E)}=\frac{(-K_X)^n \cdot A_{X}(E)}{ \int^{\infty}_{0}\vol(-K_X-tE)dt}.
$$

\begin{defn}[{\cite{FO18,BJ20}}]\label{d-delta}
For a $\mathbb{Q}$-Fano variety, we define the \emph{stability threshold} of $(X,\Delta)$ to be
$$ \delta(X):=\inf_E \ \delta_X(E),$$
where the infumum through all divisors $E$ over $X$.
\end{defn}

In fact, the $\delta$-invariant $\delta(X)$ was first defined in \cite{FO18} %as the infimum of log canonical thresholds of basis type divisors 
in the following way:
Let $X$ be a $\mathbb{Q}$-Fano variety. 
Given a sufficiently divisible $m \in \N$, we say $D \sim_{\Q} -K_X$ is \emph{an $m$-basis type $\Q$-divisor} of $-K_X-\Delta$ if there exists a basis 
$\{s_1,\ldots, s_{N_m}  \}$ of $H^0(X, \cO_{X}(m(-K_X))$ such that \[
D = \frac{1}{m N_m} \bigg(  \left\{ s_1=0 \right\} + \cdots + \left\{ s_{N_m}=0 \right\} \bigg) .\]
We set 
\[
\delta_{m}(X,\Delta): = \min \{ \lct(X;D) \, \vert \, D \sim_{\Q} -K_X \text{ is  $m$-basis type}\}.\]
The original definition of $\delta(X)$ in \cite{FO18} is  $\limsup\limits_{m \to \infty } \delta_{m} (X)$. 
Then it is shown in \cite[Theorem 4.4]{BJ20} that the limit exists and
\begin{eqnarray}\label{e-delta}
 \lim\limits_{m \to \infty } \delta_{m} (X)=\inf_E \ \delta_X(E)
\end{eqnarray}
(see Proposition \ref{p-Sapproximation}).
This way of computing $\delta(X)$ as the infimum of the log canonical thresholds for a special kind of \emph{complements} (see Definition \ref{defn-complements}) is important both conceptually and computationally, as it connects to more birational geometry tools.  See e.g.  Section \ref{ss-complements}.
\medskip

 Theorem \ref{t-valkstable}(1) can be translated into the following theorem
\begin{thm}[\cite{FO18, BJ20}]\label{t-delta}
A $\mathbb{Q}$-Fano variety $X$ is K-semistable if and only if $\delta(X)\ge 1$.
\end{thm}

\begin{say}[Uniform K-stability]\label{say-uniform}
We call $X$ to be {\it uniformly K-stable} if $\delta(X)>1$. For other equivalent descriptions of this concept, see Theorem \ref{t-uniform}. This concept was first introduced in \cite{BHJ17, Der16}. The equivalence between the original definition and the current one follows from the work of \cite{Fuj19b, BJ20} (see Theorem \ref{t-uniform}). Theorem \ref{t-valkstable}(2) says uniform K-stability implies K-stability. However, the converse is much subtler, as we do not know the infimum is attained by a divisorial valuation. This is Conjecture \ref{c-maxidd} in the case $\delta=1$.
\end{say}

%A divisor $\mathbb{Q}$-divisor $D\sim_{\mathbb{Q}}-K_X$ such that $(X,D)$ is log canonical is called {\it a complement}. After Shokurov introduced this concept in \cite{Sho00}, it is a well studied topic in birational geometry. Combining with the description using the basis type divisor (see \eqref{e-delta}), Theorem \ref{t-delta} relates the K-semistability to a special type of complements. 

% The importance of $\delta$ will only fully appear after we enlarge the definition to all valuations. 
See Section \ref{ss-valuation} for the extension of the definition of $\delta$ to the space of more general valuations. 

\begin{rem}[Twisted K\"ahler-Einstein metric]It turns out that the invariant $\delta(X)$ also has an older origin from the differential geometry in terms of twisted K\"ahler-Einstein metrics:
For a Fano manifold $X$, in \cite{Tia92} and then in \cite{Rub08, Rub09}, an invariant called {\it the greatest Ricci lower bound } of $X$ was first defined and studied  as
\[
\small{
\sup\{t \in [0, 1]\  | \mbox{ a K\"ahler metric } \omega \in c_1(X) \mbox{ such that } {\rm Ric}(\omega) > t\omega\}. 
}
\]

Later this invariant was  further studied  in \cite{Sze11, Li11, SW16} etc.
It is shown in \cite{BBJ15, CRZ18}  that for a Fano manifold $X$,  the greatest Ricci lower bound is equal to $\min\{1, \delta(X)\}$.
\end{rem}

One crucial conjecture remaining is the following (see \cite{BX19, BLZ19}).
 \begin{conj}\label{c-maxidd}
Let $X$ be a $\mathbb{Q}$-Fano variety. If $\delta(X)\le 1$, then there is a special test configuration $\cX$ with the induced divisorial valuation  $E$ (see Lemma \ref{l-induced})  such that $\delta(X)=\delta_X(E)=\delta(X_0)$, where $X_0$ is the degeneration.
\end{conj}

For a sketch of the proof of Theorem \ref{t-deltaquasi}(2) as well as a further discussion of Conjecture \ref{c-maxidd}, see Section \ref{ss-complements}.

\begin{rem}[Yau-Tian-Donaldson Conjecture]\label{r-YTDnew}
Consider the case $\delta(X)=1$. In this case, Conjecture \ref{c-maxidd} says that a Fano variety is K-stable if and only if it is uniformly K-stable. We can also formulate an equivariant version of the conjecture which predicts that a Fano variety K-polystable if and only if it is reduced uniformly K-stable (see Definition \ref{d-reduceduniform}). It has been proved recently  that any reduced uniformly K-stable Fano variety has a K\"ahler-Einstein metric (\cite{BBJ15, LTW19, Li19}). Therefore Conjecture \ref{c-maxidd} (resp. its equivariant version Conjecture \ref{c-reducedk}) predicts that any K-stable (resp. K-polystable) $\bQ$-Fano variety admits a K\"ahler-Einstein metric, hence provide a new proof of Yau-Tian-Donaldson Conjecture which would work even for \emph{singular} Fano varieties. 
\end{rem}

\begin{subappendices}
\subsection{Appendix: General valuations}\label{ss-valuation}
In this section, we will first extend the above definitions from divisors to all valuations over $X$. The key point is that this enlargement allows us to study our minimizing question in a space with certain compactness.

\medskip

Let $X$ be a variety. A valuation on $X$ will mean a valuation $v: K(X)^\times \to \R$  that is trivial on the ground field and has center $c_X(v)$ on $X$. 
We denote by $\Val_{X}$ the set of valuations on $X$, equipped with the weak topology. 
 
To any valuation $v\in \Val_{X}$ and $t\in \mathbb{R}$, there is an associated \emph{valution ideal sheaf} 
$\fa_{t}(v)$: For an affine open subset $U\subseteq X$, $\fa_t(v)(U) = \{ f\in \cO_X(U) \, \vert \, v(f) \geq t \}$ if $c_X(v) \in U$ and $\fa_t(v) (U)= \cO_X(U)$ otherwise.

\begin{exmp}[Divisors over $X$]
Let $X$ be a variety and $\pi:Y \to X$ be a proper birational morphism, with $Y$ normal. 
A prime divisor $E \subset Y$ defines a valuation $\ord_E: K(X)^\times \to \Z$ given by order of vanishing at $E$. Note that $c_X(\ord_E)$ is the generic point of $\pi(E)$ and, assuming $X$ is normal, $\fa_p(v) = \pi_*\cO_X(-pE)$.  
\end{exmp}

\begin{exmp}[Quasi-monomial valuations]\label{e-quasi} Denote $Z\to X$ a log resolution with simple normal crossing divisors $E_1$,...., $E_r$ on $Z$. Denote by $\alpha=(\alpha_1$,...., $\alpha_r)\in \mathbb{R}^r_{\ge 0}$. Assume $\bigcap^r_{i=1} E_i\neq \emptyset$; and there exists a component $C\subset \cap E_i$, such that around the generic point $\eta$ of $C$, $E_i$ is given by the equation $z_i$ in $\mathcal{O}_{Z,\eta(C)}$.
We define a valuation $v_{\alpha}$ to be
$$v_{\alpha}(f)=\min \{\sum \alpha_i\beta_i|\ c_{\beta}(\eta)\neq 0\},$$
and all such valuations are called {\it quasi-monomial valuations}. It is precisely the valuations satisfying the equality case in Abhyankar's inequality, therefore, it is another description of Abhyankar valuation. 

We call the dimension of the $\mathbb{Q}$-vector space spanned by $\{\alpha_1$,...., $\alpha_r\}$ the rational rank of $v_{\alpha}$, and one can show $v_{\alpha}$ is a rescaling of a divisorial valuation if and only if $r=1$. For fixed $C\subset (Z,E)$ as above, the valuations $v_{\alpha}$ for all $\alpha$ gives a  simplicial cone which is a natural subspace in $\Val_{X}$. 
\end{exmp}

\begin{exmp}Given a valuation $v$, and a simple normal crossing (but possibly non-proper) model $(Z,E=\sum E_i)$ over $X$ such that the center of $v$ on $Z$ is non-empty, we can define a valuation $v_{\alpha}=\rho_{(Z,E)}(v)$, where the corresponding component  $\alpha_i$ defined to be $v(z_i)$. Started from a simple normal crossing model, by successively blowing up the center of $v$ and (possibly shrinking), we get a sequence of models $\phi_i\colon Z_i\to Z_{i-1}$ where $Z_0=  X $ such that the center of $v$ on $Z_i$ is not empty. Define $E_i=\phi^{-1}_{i*}(E_{i-1})+{\rm Ex}(\phi_i)$. Denote by $v_{\alpha,i}=\rho_{Z_i,E_i}(v)$, then $v=\lim_{i\to \infty}(v_{\alpha,i})$.
\end{exmp}

\begin{defn}[{Log discrepancy function on $\Val(X)$, see \cite{JM12,BdFFU15}}] When $X$ is klt, the log discrepancy function $A_{X}$ can be extended to a function 
$$A_{X}: \Val(X)\to (0, +\infty]$$
in the following way: we have already defined $A_{X}(E)$ for a divisorial valuation. For a quasi-monomial valuation as in Example \ref{e-quasi}, we define 
$A_{X}(v_{\alpha})=\sum_i \alpha_iA_{X}(E)$.  And for a general valuation $v$, we define $A_{X}(v)=\sup_{Y,E} A_{X}(\rho_{Y,E}(v))$.

\end{defn}

Let $X$ be a $\bQ$-Fano variety. For any $t\in \mathbb{R}_{\ge 0}$, we can also define a volume function:
$$\vol(-K_X-tv)=\lim_{k\to  \infty}\frac{\dim H^0(\mathcal{O}_X(-kK_X)\otimes \fa_{tk})}{k^n/n!},$$
where $\fa_{tk}$ is the associated ideal sheaf of $v$.

Then for any valuation $v$ with $A_{X,\Delta}(v)< +\infty$ we can similarly define,
$$\beta_{X}(v) :=   A_{X}(v) -\frac{1}{(-K_X)^n} \int^{\infty}_{0}\vol(-K_X-tv)dt.$$
and
$$
\delta_X(v)=_{\rm defn}\frac{(-K_X)^n \cdot A_{X}(v)}{ \int^{\infty}_{0}\vol(-K_X-tv)dt}.
$$
It is easy to see $\delta(v)=\delta(\lambda\cdot v)$ for any $\lambda>0$. 

The advantage of extending the definition to all valuations can be seen by the next theorem. 
\begin{thm}\label{t-deltaquasi}
We have the following two facts:
\begin{enumerate}
\item  {\cite{BJ20}} For a $\mathbb{Q}$-Fano variety $X$, 
$$ \delta(X):=\inf_v \ \delta_X(v),$$
where the infimum runs through all valuations $v$ with $A_X(v)<+\infty$.
\item \cite{BLX19} (see Theorem \ref{t-quasimini}) When $\delta(X)\le 1$, then the infimum is attained by a quasi-monomial valuation.
\end{enumerate}
\end{thm}

\end{subappendices}

\section{Special valuations}\label{s-specialvaluation}

In this section, we aim to connect degenerations and valuations in a more straightforward manner.  For this,  a key observation is to use  \emph{complements}. As a consequence, we will  see in Theorem \ref{t-weakspecial}, which gives a geometric characterization of weakly special divisorial valuations as precisely the set of log canonical places of complements.

\subsection{Complements}\label{ss-complements}
The following notion first appeared in \cite{Sho92}. 
\begin{defn}\label{defn-complements}
For a Fano variety $X$, we say that an effective $\bQ$-divisor $D$ is an {\it $N$-complement} for some $N\in \bN_+$, if $N(K_X+D)\sim 0$, and  a {\it $\bQ$-complement}, if it is an $N$-complement for some $N$.   
\end{defn}
Any valuation $v$ is said to be \emph{an lc place} if it satisfies that $A_{X,D}(v)=0$.

\subsubsection{Log canonical places of complements}
 We first recall results established in \cites{BC11, BJ20} to approximate $S(E)$ by invariants defined in a finite level.  
\begin{defn}For a valuation $v$, we define 
$$S_m(v):=\{\sup v(D) \ | D \mbox{ is a $m$-basis type divisor} \}.$$
\end{defn}
It can be easily seen that the above supremum is indeed a maximum. Two less nontrivial facts are the following, which combined can easily yield \eqref{e-delta}.
\begin{prop}\label{p-Sapproximation}
Notation as above. 
\begin{enumerate}
\item For any valuation $v$ with $A_X(v)<+\infty$,  $\lim_mS_m(v)=S(v)$.
\item For every $\epsilon>0$, there exists $m_0 >0$ such that
$S_m(v) \le (1 + \epsilon )S(v)$ for any $m\ge m_0$ and any $v$ with $A_X(v)<+\infty$. 
\end{enumerate}
\end{prop}
\begin{proof}For (1), see \cite{BC11}*{Theorem 1.11}. For (2), see \cite{BJ20}*{Corollary 2.10}. 
\end{proof}

The following approximating result can be considered to be a version of \cite{LX14} for valuations. 
\begin{prop}[{\cite{BLX19, BLZ19}}]\label{p-deltaapproximation}
For a given Fano variety $X$, if $\delta(X)\le 1$, then $\delta(X)=\inf_E \delta_X(E)$ for all $E$ which is an lc place of a $\bQ$-complement.
\end{prop}

\begin{proof}  By \eqref{e-delta}, we can pick a sufficiently large $m$, such that
$$\delta_m(X)\le (1+\epsilon)\delta(X)\le (1+\epsilon)^2\delta_m(X).$$  
We may also assume $S_m(E)\le (1+\epsilon) S(E)$ for any divisor $E$ by Proposition \ref{p-Sapproximation}(2). There exists  an $m$-basis type $\bQ$-divisor $D_m$ and a divisor $E$ over $X$ with
$$\delta_m(X)=\lct(X; D_m)=\frac{A(E)}{\ord_E(D_m)}=\frac{A(E)}{S_m(E)}.$$

We first assume $\delta(X)<1$. So we can also assume $\delta_m(X)<1$. Thus we can find a general $\bQ$-divisor $H\sim_{\bQ} -(1-\delta_m)K_X$, such that $(X, \delta_m D_m+H)$ is log canonical, and $E$ is an lc place of such divisor. Thus
$$\delta(E)=\frac{A(E)}{S(E)}\le (1+\epsilon)\frac{A(E)}{S_m(E)}\le (1+\epsilon)^2\delta(X).$$
Then we can pick a sequence $\epsilon\to 0$ and a corresponding sequence of $E$. 

\medskip

When $\delta(X)=1$, we need some extra perturbation argument. Consider the pseudo-effective threshold 
\[
\mbox{$T(E):=\sup\{ t\ |-\mu^*K_X-tE \mbox{ is pseudo-effective} \}$}
\]
Then by \cite{Fuj19c}*{Proposition 2.1}, we know 
$T(E)\ge \frac{n+1}{n}S(E)$ for any $E$. Thus we can find a divisor $D\sim_{\mathbb{Q}} -K_X$, such that 
\begin{equation*}
\begin{split}
\mult_E(D)&\ge(1+\frac{1}{n}-\epsilon_0)S(E)\ge (1+\frac{1}{n}-\epsilon_0)\frac{S(E)}{S_m(E)}S_m(E) \\ 
&\ge (1+\frac{1}{n}-\epsilon_0)\frac{1}{1+\epsilon}\frac{A(E)}{\delta_m} \ge (1+\frac{1}{2n})A(E),
\end{split}
\end{equation*}
if we choose $\epsilon_0, \epsilon$ sufficiently small.
%if we choose $m$ sufficiently large and $\epsilon$ sufficiently small such that $\delta_m(X)\cdot (1+\epsilon)\le \frac{2n+2}{2n+1}$. 

Fix $t\in (0, \alpha(X))$, then the pair $(X,t D)$ is klt, and 
$$A_{X,t D}(E)=A_X(E)-t\cdot  \mult_E(D)  \qquad \mbox{and} \qquad  S_{X,tD}=(1-t)S(E),$$
which implies that
\begin{equation*}
\begin{split}
\delta(X,t D)&\le\frac{A_{X,t D}(E)}{S_{X,t D}(E)}\le  \frac{1-t(1+\frac{1}{2n})}{1-t}\delta(E) \\
&\le\frac{(1-t(1+\frac{1}{2n}))(1+\epsilon)}{1-t} \delta_m<1,
\end{split}
\end{equation*}
if we choose $\epsilon$ sufficiently small (depending on $t$).

Thus for any such $t$, by the case of $\delta(X,tD)<1$, we know that there exists an lc place $F$ of $(X,tD+D')$ for a $\bQ$-complement $D'$ with $\delta_{X,tD}(F)$ sufficiently close to $\delta(X,t D)$.
On the other hand, since $\alpha(X)\le \lct(X,D)\le  \frac{A_X(F)}{\mult_F(D)}$, 
$$A_{X,tD}(F)=A(F)-t\cdot \mult_FD\ge (1-\frac{t}{\alpha(X)})A(F),$$ 
which implies that
$$\delta(F)=\frac{A(F)}{S(F)}\le \frac{(1-t)A_{X,tD}(F)}{(1-\frac{t}{\alpha(X)})S_{X,tD}(F)}=\frac{1-t}{1-\frac{t}{\alpha(X)}}\delta_{X,tD}(F). $$
Since $\delta(X,tD)\le \frac{1}{1-t}$,  we can choose a sequence $t_m\to 0$ and a corresponding sequence of lc places $F_m$ with $\lim \delta(F_m)\to 1$.
\end{proof}

\subsection{Degenerations and lc places}
In the rest of this section, we will proceed to establish the correspondence between lc places of complements and weakly special  divisorial valuations (see Theorem \ref{t-weakspecial}).
Built on the  observations in \cite{Li17} to use the cone construction (see \ref{ss-cone}) to study valuations on a Fano variety $X$, and later in \cite{LWX18} to relate the valuations over a cone  with the degenerations of $X$, eventually in \cite{BLX19}, we realize that in the correspondence in Lemma \ref{l-integraldreamy}, if we consider the smaller class of all weakly special test configurations with an integral central fiber they correspond to lc places of complements, i.e. the latter are precisely weakly special divisorial valuations.

We first show the (easy) direction which says a divisor $E$ that is an lc place of $(X,D)$ for a $\bQ$-complement $D$ is weakly special. We will establish the reverse direction and therefore complete the correspondence in Theorem \ref{t-weakspecial}.

\begin{prop}\label{p-lctotest}
Let $E$ be an lc place of $(X,D)$ where $D$ is a $\bQ$-complement, then $E$ is weakly special. Let $\cX$ be the corresponding test configuration in Lemma \ref{l-integraldreamy} for $\ord_E$. Then $\cX$ is weakly special with irreducible components and $2\cdot \Fut(\cX)=\beta(E)$. 
\end{prop}
\begin{proof}
If $E$ is an lc place of $(X,D)$ where $D$ is a $\bQ$-complement, then there is a model $Y\to X $which precisely extracts $E$ and $Y$ is of Fano type. Thus $E$ is dreamy. In particular, the Rees construction is finitely generated, thus we obtain a test configuration $\cX$ with an integral special fiber $X_0$. Next we show the test configuration is weakly special.

Consider the trivial family $(X,D)\times \mathbb{A}^1$, then we know $E_{\mathbb{A}^1}:=E\times \mathbb{A}^1$ is an lc place of the pair. Therefore, $E_{\mathbb{A}^1}$ and $X_0$ are lc places of $(X\times \mathbb{A}^1,X\times\{0\}+D\times \mathbb{A}^1)$ which implies the blowup  $E_1$ of $E_{\mathbb{A}^1}$ and $X\times\{0\}$ is also a lc place of $(X\times \mathbb{A}^1,X\times\{0\}+D\times \mathbb{A}^1)$. Since $(X\times \mathbb{A}^1,D\times \mathbb{A}^1)$ a (trivial) family of  log Calabi-Yau pairs. Therefore, we can apply MMP to construct a model $\cX\to \mathbb{A}^1$ such that the special fiber of $\cX$ is the irreducible divisor $E_1$, $-K_{\cX}$ is ample and $(\cX, \cD)$ is also a family of log Calabi-Yau pairs i.e., $\cX$ is the test configuration constructed in the first paragraph, which is weakly special. 
\end{proof}

This is enough for us to prove the reverse direction of Theorem \ref{t-valkstable}(1).
\begin{thm}\label{t-kimpliesbeta}
$(\mbox{$X$ is K-semistable})\Longrightarrow(\beta_X(E)\ge 0, \forall E).$
\end{thm}
\begin{proof} 
If $\delta(X)<1$, then by Proposition \ref{p-deltaapproximation}, we know that there exist a $\bQ$-complement $D$ and a divisor $E$ which is an lc place for $(X,D)$ such that $\delta(E)<1$.
By Proposition \ref{p-lctotest}, we could construct a test configuration $\cX$ with $\Fut(\cX)<0$, so $X$ is not K-semistable. 
\end{proof}
\begin{rem}Comparing the above argument to the one in \cite{LX20}, we do not use the cone construction. However, we still need it for the reverse direction, see Theorem \ref{t-weakspecial}.  
%while here we still use the cone construction to see lc places of $(X,D)$ for a $\bQ$-complement yield  weakly special degenerations, this is an entirely geometric problem unrelated to the stability concern. On the other hand, in \cite{LX20}, we have to use the normalized volume function of the cone singularities, and relate the global and local K-stability questions.   See Theorem \ref{t-maintheorem1}.
\end{rem}

 To finish the proof of Theorem \ref{t-valkstable}(2), it suffices to show if $X$ is K-semistable, and $E$ is a divisor such that $\beta_X(E)=0$, then $E$ induces a special test configuration. This is a special case of Conjecture \ref{c-maxidd}.
By the discussion in \ref{s-rees}, a crucial thing is to prove that the double graded ring in \eqref{e-doublegraded} is finitely generated. We will use a global argument slightly simpler than \cite[Section 4.1]{BX19}.

\begin{thm}\label{t-delta=1}
Let $X$ be a K-semistable log Fano pair and $E$ a divisor over $X$. If  
\[
1 = \delta(X) =\delta_X(E),
\] then $E$ is dreamy and induces a non-trivial special test configuration $\cX$ such that $\Fut(\cX)=0$. In particular, $X$ is not K-stable.  
\end{thm}
\begin{proof}
Denote by $a=A_{X}(E)=S_X(E)$. 
We have 
$$\lim_m S_m(E)=\lim S(E) \qquad \mbox{and} \qquad \lim_m \delta_m(X)= \delta(X)=1 \mbox{ by \eqref{e-delta}}.$$
Fix $\epsilon\le \frac{1}{2 a+2}$. We can pick up a sufficiently large  $m$  such that 
$$a+1> S_m(E)\ge a-\frac{1}{2}\qquad \mbox{and }\qquad \delta_m> 1-\epsilon.$$ 
Thus there is an $m$-basis type divisor $D_m$ which computes $S_m(E)$, satisfying that $(X, (1-\epsilon)D_m)$ is klt (since $1-\epsilon<\delta_m$) and 
$$b:=A_{X,(1-\epsilon)D_m}(E)=a-(1-\epsilon)S_m(E)=(a-S_m(E))+\epsilon\cdot S_m(E)<1.$$ 
Thus by \cite{BCHM10}, we can construct a birational model $\mu\colon Y\to (X,(1-\epsilon)D_m)$, which precisely extracts $E$ with $-E$ ample over $X$. Then we have $(Y,(1-\epsilon)\mu_*^{-1}D_m+(b-\epsilon_0)E)$ is a log Fano pair for some sufficiently small $\epsilon_0$. In particular, it is a Mori dream space, and we know the double graded ring
\eqref{e-doublegraded} is finitely generated. 

Then we can construct the test configuration $(\cX,\cL)$ with an irreducible central fiber $\cX$. This implies $\cL\sim_{\bA^1} -rK_{\cX}$ and therefore  
$2\cdot \Fut(\cX)=\beta(E)=0.$ By Theorem \ref{t-specialdegeneration}, $\cX$ has to be a special test configuration. 
\end{proof}

The above theorem has been extended to the case that $\d(X)\le 1$.
\begin{thm}[{\cite{BLZ19}}]\label{t-blzdegeneration} If a $\mathbb{Q}$-Fano variety $X$ satisfies $\delta_X(E)=\delta(X)\le 1$, then $E$ comes from a special degeneration. Moreover, the central fiber $X_0$ satisfies $\delta(X_0)=\delta(X)$, and it is computed by the valuation induced by the $\bG_m$-action. 
\end{thm}

The following theorem completes the correspondence between lc places of complements and weakly special test configurations. See Theorem \ref{t-Nweakspecial} for a strengthening, using Birkar's Theorem of the boundedness of complements in \cite{Bir19}.  
\begin{thm}[{\cite{BLX19}*{Thm A.2}}]\label{t-weakspecial}
For a $\bZ$-value divisorial valuation $v\in \Val_X$ over an $n$-dimensional $\bQ$-Fano variety $X$, it  arises from a nontrivial weakly special test configuration with a irreducible central fiber if and only if there exists an $\mathbb{Q}$-complement $D$ of $X$ such that $v$ is an lc place of $X$. In other words, a divisorial valuation is weakly special if and only if it is an lc place of an $\mathbb{Q}$-complement. 
\end{thm}
\begin{proof} In Proposition \ref{p-lctotest}, we have seen from an lc place $E$ of a $\mathbb Q$-complement $D$, we can obtain a weakly special test configuration.
Then if we consider a valuation $v=a\cdot \ord_E$ for an lc place $E$ of a $\mathbb{Q}$-complement and $a\in \mathbb N_+$, then in the proof of Proposition \ref{p-lctotest}, we take the weighted blowup of $E_{\mathbb{A}^1}$ and $X\times\{0\}$ with weight $(1,a)$, and run an MMP as before, we obtain a weakly special test configuration as desired. 

%Let $Y=C(X,-rK_X)$ be the cone and $D_Y=C(D,-rK_X)$ the cone over $D$. We take the construction as in \eqref{e-rayvaluation} again. For any divisorial valuation whose rescaling is $v_t$ for some $t\in \bQ_{+}$, it is a lc place of $(Y,D_Y)$. This implies the corresponding Rees construction induced by $E$, as in Paragraph \ref{s-rees}, will give a weakly special test configuration.

\smallskip

Conversely, from a weakly special test configuration $\cX$ with an irreducible central fiber, we want to prove the induced $\bZ$-valuation $v$ is an lc place of a $\bQ$-complement $D$. The original proof in \cite{BLX19} used the cone construction. Here we gave a (simpler) global argument\footnote{It is suggested by Harold Blum.}, relying on some latter results in Section \ref{s-Ding}.

Let $\ord_E$ be the induced divisorial valuation by the weakly special test configuration.  Consider the valuation of $\ord_E$ on $R:=\bigoplus_{m\in \bN}H^0(-mrK_X)$. Then we know 
$$A(E)-S(E)=\beta(E)\ge \beta(\cF_E)=\frac{\mu_1(\cF_{E})-S(\cF_E)}{r}\ge \DNA(\cF_E)=A(E)-S(E).$$
See Definition \ref{d-filbeta} for the invariants of $\cF_E$; the inequalities follow from Proposition \ref{l-betading};
and the last equality follows from the fact that $\cX$ is weakly special test configuration. This implies that $rA(E)=\mu_1({\cF_{E}})$ as $rS(E)=S(\cF_E)$. The rest follows from \cite[Theorem A.7]{XZ19} which we will include here for reader's convenience.

Let $I_{m,p}(\cF):={\rm Im} \left(  \cF^p R_m \otimes \cO_{X}(-rmK_X) \to \cO_X\right)$ and $I^{(t)}_{\bullet}=(I_{m,mt})$. The function $t\mapsto \lct(X;I^{(t)}_\bullet)$ is  piecewise linear  on $(0,rT_{X}(E)]$ due to the finite generation assumption.
Then
\[
\lct(X,;I^{(r\cdot A_{X}(E))}_\bullet)\le \frac{A_{X}(E)}{\ord_E(I^{(r\cdot A_{X}(E))}_\bullet)}\le \frac{A_{X}(E)}{r\cdot A_{X}(E)} = \frac{1}{r},
\]
thus $\lct(X;I^{(r\cdot A_{X}(E))}_\bullet) = \frac{1}{r}$.
Since ${\rm gr}_{\cF_E}R$ is finitely generated, then
\[
\lct(X;I^{ (r\cdot A_{X}(E))}_\bullet)=m\cdot \lct(X;I_{m, mrA_{X}(E))})= \frac{1}{r}.
\]
for some sufficiently divisible $m$. This means there is a divisor $D\in |-mrK_X|$ with $\ord_E(D)\ge mrA_{X}(E)$ and $(X,\frac{1}{mr}D)$ is log canonical. Thus $E$ is an lc place of $(X,\frac{1}{mr}D)$.  
\end{proof}
\begin{rem}\footnote{This comes out from a discussion with Yuchen Liu.} An interesting consequence of Theorem \ref{t-weakspecial} is that a $\mathbb{G}_m$-equivariant Fano variety has a $\bG_m$-equivariant $\mathbb{Q}$-complement. In fact, let $\cX$  be the product test configuration given by the $\mathbb{G}_m$-action on $X$. By Theorem \ref{t-weakspecial},  we know this yields a valuation which is an lc place of a $\mathbb{Q}$-complement $D$. Then the special fiber $X_0$ of $\cX$ over $0$ is an lc place of $(X\times \mathbb{A}^1,D\times \mathbb{A}^1+X\times \{0\})$. Thus the closure $\mathcal D$ of $D\times \mathbb{A}^1$ on $\cX$ yields a family of log CY pairs, i.e. $D_0:=\mathcal{D}\times_{\mathbb A^1}\{0\}$ is a $\bG_m$-equivariant $N$-complement. 
\end{rem}

We have the following description of special divisors (see Definition \ref{d-specialdivisor}).
\begin{thm}[Zhuang]In Theorem \ref{t-weakspecial}, the following are equivalent
\begin{enumerate}
\item a divisor $E$ over $X$ is special;
\item $A_X(E)<T(E)$ and there exists a $\mathbb{Q}$-complement $D^*$, such that $E$ is the only lc place of $(X,D^*)$; and
\item there exsits a divisor $D\sim_{\mathbb {Q}}-K_X$ and $t\in (0,1)$ such that $(X,tD)$ is lc and $E$ is the only lc place for $(X,tD)$.
\end{enumerate}
\end{thm}
\begin{proof}We first prove $(1)\Longleftrightarrow (2)$. A test configuration $\cX$ is special if and only if for any effective $\mathbb{Q}$-divisor $D\sim_{\mathbb Q}-K_X$, there exists a positive $\epsilon$, such that $(\cX, \epsilon\cD)$ is a weakly special test configuration where $\cD$ is the closure of $D\times \mathbb{G}_m$ in $\cX$. This is equivalent to saying that for $D$, there are a positive $\epsilon$ and  an effective $\mathbb{Q}$-divisor $D'\sim_{\mathbb{Q}}-K_X$ of $X$ such that $E$ is an lc place of  the lc pair $(X,\epsilon D+(1-\epsilon)D')$.

Now we assume $E$ satisfies the conditions in (2). Then we have $\mathbb{Q}$-divisors $D'\sim_{\mathbb Q}D''\sim_{\mathbb Q}-K_X$ such that $A_{X,D'}(E)>0$ and $A_{X,D''}(E)<0$. By choosing one of them as $D_1$ and an appropriate $a\in [0,1)$, we have $A_{X,aD_1+(1-a)D}(E)=0$. 
Thus for a sufficiently small $\epsilon$, $\Big(X,\epsilon\big((1-a)D_1+aD\big)+(1-\epsilon)D^*)\Big)$ is lc and has $E$ as its lc place.

Conversely, if we take $D$ to be a general $\mathbb{Q}$-divisor whose support does not contain ${\rm Cent}_X(E)$, then since for some $\epsilon>0$, $E$ is an lc place of $(X, \epsilon D+(1-\epsilon)D')$ from our assumption, we have 
$$A_X(E)=A_{X,\epsilon D}(E)\le T_{X,\epsilon D}(E)=(1-\epsilon)T_X(E).$$ 
To see the second property, we denote by $\mu\colon Y\to X$ the model precisely extracting $E$.  Then we run MMP for $-K_Y-E$ to get a model $Y'$, and we claim $(Y',E')$ is plt. Otherwise, we can find an effective $\bQ$-divisor $D_{Y'}\sim_{\mathbb{Q}}-K_{Y'}-E'$ such that $(Y', E'+ \epsilon D_{Y'})$ is not log canonical for any $\epsilon>0$, as $-K_{Y'}-E'$ is big. This yields an effective $\mathbb{Q}$-divisor $D\sim_{\mathbb{Q}}-K_X$ on $X$ violating our assumption on $E$. Now we pick up a general $\bQ$-complement of $(Y', E')$, it induces a $\bQ$-complement $D^*$ with $E$ the only lc place of $(X,D^*)$.  

Finally, from the above discussions, one can easily see $(1)\Longrightarrow (3)\Longrightarrow (2)$. We leave the details to the reader.
\end{proof}

\subsection{Bounded complements} Following \cite{BLX19}, to further extend the correspondence established in Proposition \ref{p-corresp}, we need the difficult theorem on the boundedness of complements proved in \cite{Bir19}.  In fact, the following lemma, which is a key step for our argument, is a consequence of Birkar's theorem on the boundedness of complements.
\begin{lem}[{\cite{BLX19}*{Theorem 3.5}}]\label{l-complementbounded}
Let $n$ be a positive integer. Then there exists a positive integer $N = N(n)$, such that for any $n$-dimensional $\bQ$-Fano variety, if $E$ is an lc place of a $\bQ$-complement of $X$, then it is indeed an lc place of an $N$-complement. 
\end{lem}
Then a consequence is the following improvement of Theorem \ref{t-weakspecial}.
\begin{thm}[{\cite{BLX19}*{Thm A.2}}]\label{t-Nweakspecial}
Let $n$ be a positive integer and $N = N(n)$ as above. A divisorial valuation is weakly special if and only if it is an lc place of an $N$-complement. 
\end{thm}

An application of the boundedness of complements (see Lemma \ref{l-complementbounded}) is to use it to understand the valuation which computes $\delta(X)$, when $\delta(X)\le 1$.  

Combining Proposition \ref{p-deltaapproximation} and
 Lemma \ref{l-complementbounded}, we know $\delta(X)=\lim_i\delta_X(E_i)$, where $E_k$ is an lc place of an $N$-complement for any $k$. Then we have the following construction.

\begin{say}\label{say-stratification}
Given a $\bQ$-Fano variety. For any fixed $N$, there is a variety $B$ of finite type, and a family of $\mathbb{Q}$-Cartier divisors $\cD\subset X\times B$, such that for any $b\in B$, $(X,\frac{1}{N}D_b)$ is strictly log canonical and $N(K_X+\frac{1}{N}D_b)\sim 0$. Moreover, any $N$-complement can be written as $\frac{1}{N}D_b$ for some $b\in B$. In particular, after a further stratification of $B$ and a base change, we can assume there exists $\cW\to (X\times B,\cD)\to B$, which gives fiberwise resolutions over points in $B$. In particular, for each component $B_i$ of $B$, we can consider all the prime divisors $E_{i,j}$ on $\cW$ which have log discrepancy 0 with respect to $(X\times B_i, \cD|_{B_i})$. From the assumption on $\cW$, we know that restricting $E_{i,j}$ over each $b\in B_i$ yields a divisor satisfying $A_{X,D_b}(E_{i,j})=0$.  Thus  for any $b\in B_i$, the dual complex $W:=\cD\big((E_i)_b:=\sum_{j}(E_{i,j})_b\big)$ can be canonically identified with the dual complex consisting of all lc places of $(X,D_b)$.
\end{say}

\begin{say}\label{say-locallyconstant}
Applying the construction as in Paragraph \ref{say-stratification}, after taking a subsequence, we can assume all corresponding points $b_k\in B$ of $E_k$ belonging to the same component $B_i$ of $B$.  By \cite[Thm. 1.8]{HMX13}, we can argue that the function $S(E_b)=S(E_{b'})$ if $b$ and $b'$ belong to the same component $B_i$ of $B$, and the prime divisors $E_b, E_{b'}$ correspond to the same point of $W_{\bQ}$.
Thus a divisor $E$ corresponding to a fixed point on the dual complex $W$, the function
$b\to \delta_X(E_b)$
does not depend on $b$ (in $B_i$). Therefore, the function $\delta_X(\cdot)$ can be considered to be a function on $W$, which is continuous. So it attains a minimum at a point corresponding to a quasi-monomial valuation $v$. (A priori, $v$ is not necessarily a divisorial valuation since it may not correspond to a rational point on the dual complex.)
\end{say}

 To recap on the above discussion, thus we show 
\begin{thm}\label{t-quasimini}
If $\delta(X)\le 1$, then it is always computed by a quasi-monomial valuation which is an lc place of an $N$-complement. 
\end{thm}
In fact, one can prove if $\delta(X)\le 1$, then any valuation computing $\delta(X)$ is an lc place of an $N$-complement. See \cite{BLX19}*{Thm A.7}.

To prove Conjecture  \ref{c-maxidd}, it suffices to produce a divisorial valuation which computes $\delta$, out of a quasi-monomial one (see Theorem \ref{t-blzdegeneration}). The key point is to verifying the following conjecture.
\begin{conj}[{\cite{Xu20}*{Conj. 1.2}}]\label{c-finitegeneration}
Let $X$ be a $\bQ$-Fano variety with $\delta(X)\le 1$, and $v$ a (quasi-monomial) valuation compute $\delta(X)$, then the graded ring ${\rm gr}_vR$ is finitely generated.
\end{conj}
The finite generation conjecture above  will then imply that for a sufficiently small rational perturbation $v'$ of $v$ within the rational face where $v$ is an interior point (see Example \ref{e-quasi}), we have
$${\rm gr}_vR\cong {\rm gr}_{v'}R\mbox{ (\cite{LX18}*{Lem. 2.10}) }\ \ \ {\rm and}\ \ \  \delta_X(v)=\delta_X(v')\mbox{  (\cite[Prop 3.9]{Xu20}) }.$$
See \cite{Xu20} for more discussions.  In \cite{AZ20}*{Theorem 4.15}, an example of a quasi-monomial valuation is given, which is an lc place of a $\mathbb Q$-complement but with a non-finitely generated associated graded ring.
\begin{subappendices}

\subsection{Appendix: Normalized volumes and local stability}\label{s-local}

In this appendix, we will discuss a local K-stability picture. This is established using Chi Li's definition of normalized volumes. For reader who wants to know more background on this, see the survey paper \cite{LLX18}. In Secion \ref{ss-normv}, we give the definition of normalized volumes as well as a very quick sketch of {\it the Stable Degeneration Conjecture}. Then in Section \ref{ss-cone}, we will explain the cone construction. In Section \ref{ss-localglobal},  we discuss only a small (known) part of the Stable Degeneration Conjecture for cone singularities, which has consequences on the K-stability question of the base Fano variety of a cone. See Theorem \ref{t-maintheorem1}, which provides a different proof of Theorem \ref{t-valkstable}.

\subsubsection{Normalized volume}\label{ss-normv}
The following notion of {\it normalized volume} was first introduced in \cite{Li18}. It shares some similar flavor with the $\delta$-invariant, but is defined in a local setting.

Let $Y$ be an $n$-dimensional klt singularity and $x\in Y={\rm Spec}(R)$ a closed point. The {\it non-archimedean link} of $Y$ at $x$ is defined  as
$$\Val_{Y,x} := \{ v\in \Val_{X} \, \vert\, c_Y(v) = \{x \} \, \} \subset \Val_Y.$$

\begin{defn}[\cite{Li18}]\label{d-normvol}
The \emph{normalized volume function} $$\hvol_{Y,x}:\Val_{Y,x}\to(0, +\infty]$$
 is defined by
  \[
  \hvol_{Y, x}(v)=\begin{cases}
            A_{Y}(v)^n\cdot\vol(v) & \textrm{ if }A_{Y}(v)<+\infty;\\
            +\infty & \textrm{ if }A_{Y}(v)=+\infty.
           \end{cases}
 \]
 The \emph{volume of the singularity} $x\in Y$ is defined as
 \[
  \hvol(x, Y):=\inf_{v\in\Val_{Y,x}}\hvol_{Y,x}(v).
 \]
 The previous infimum is a minimum by the main result in \cite{Blu18} (see also \cite[Remark 3.8]{Xu19}).
 \end{defn}

In the recent study of the normalized volume function, the  guiding question is the {\it Stable Degeneration Conjecture} (see \cite[Conjecture 7.1]{Li18} and \cite[Conjecture 1.2]{LX18}). There has been a lot of progress in it (see \cite{LLX18}). The following results, which are part of the Stable Degeneration Conjecture, have been shown.
 \begin{thm}\label{t-singularityvolume}
 For any klt singularity $x\in Y$,
 \begin{enumerate}
 \item \cites{Blu18,Xu19} The infimum of the normalized function $ \hvol_{Y,x}(v)$ is always attained by a quasi-monomial valuation.
 \item \cite{XZ20} The minimizer is unique up to rescaling.
 \end{enumerate} 
 \end{thm}
 
 The main remaining open part is an analogue version of Conjecture \ref{c-finitegeneration}.
 \begin{conj}\label{c-localfinite}
 For a klt singularity $x\in Y$, if $v$ is a minimizer of $\hvol_{Y,x}$, then the associated graded ring $R_0:={\rm gr}_vR$ is finitely generated. 
 \end{conj}
 This can be easily proved using \cite{BCHM10} in the case when $v$ is of rational rank 1, but the higher rational rank case remains open. 
 
  \begin{thm}[{\cites{Li17, LX20, LX18}}]\label{t-singularityvolume}
 For any klt singularity $x\in Y$, if the associated graded ring of the minimizer is finitely generated, then the induced degeneration $(Y_0={\rm Spec}(R_0),  \xi_v)$ is K-semistable.  In particular, it holds when the minimizer is divisorial.
 \end{thm}
 % \begin{conj}[{Stable Degeneration Conjecture}]\label{c-SDC}
% Given any arbitrary klt singularity $x\in Y={\rm Spec}(R)$. The minimizer $v$ is unique up to a rescaling. Moreover, it is quasi-monomial with a finitely generated associated graded ring, and %\end{conj}
Here $(Y_0,\xi_v)$ being K-semistable is in the sense of  {\it a Fano cone singularity}. See \cite{CS18} or \cite[Section 2.5]{LLX18} for the definition of a Fano cone singularity and its K-semistability. 

%We call a valuation in $\Val_{x,X}$ to be a {\it K-semistable valuation} if  $v$ is quasi-monomial with a finitely generated associated graded ring, and the induced degeneration $$(Y_0={\rm Spec}(R_0),\Gamma_0, \xi_v)$$ is K-semistable. 
 
 The converse was proved in \cite{Li17, LX20, LX18}, i.e. any valuation $v$ which is quasi-monomial with a finitely generated associated graded ring, and satisfies that the induced degeneration $(Y_0, \xi_v)$ is a K-semistable Fano cone must be a minimizer of the normalized volume function. 
  As a consequence, this answers the K-semistable part of \cite[Conjecture 3.22]{DS17} (see \cite{LX18}). 

In the below, we will only concentrate on a small part of the local theory, namely, we connect the K-(semi)stability of a $\bQ$-Fano variety with the minimizing problem on the cone singularity over the Fano variety. This was initiated in \cite{Li17}. 

\subsubsection{Cone construction}\label{ss-cone}

Let us first recall the cone construction. 
\begin{say}[Cone construction]\label{say-cone}
Denote by $Y:=\Spec (R)=C(X,-rK_X)$ the cone over $X$ where 
$$R=\bigoplus_m R_m:=\bigoplus_{m\in \bN} H^0(-mrK_X).$$
 Blowing up the vertex $x$ of the cone $Y$, we get an exceptional divisor $X_{\infty}\cong X$, which is called {\it the canonical valuation}.

Let $E$ be any divisor over $X$ that arises on a proper normal model $\mu:Z\to X$. Following  \cite{Li17,LX20}, $E$ gives rise to a ray of valuations 
\begin{eqnarray}\label{e-ray}
\{v_{t}\ |\  t\in [0,\infty) \subset \Val_{Y,x}\}.
\end{eqnarray}
  Since the blowup of $Y$ at $0$ is canonically isomorphic to the total space of the line bundle $\cO_X(-L)$, there is a proper birational map from 
 $Z_{L^{-1}} \to Y$, where $Z_{L^{-1}}$ denote the total space of $\mu^*\cO_X(-L)$. 
   Now, $$v_0=\ord_{X_{\infty}}\mbox{\ \ \   and  \ \ \ }v_{\infty}=\ord_{E_{\infty}},$$
  where  $E_{\infty}$ denotes the pullback of $E$ under the map 
  $Z_{L^{-1}} \to Y$ and $X_{\infty}$ denote the zero section of $Z_{L^{-1}}$. 
  Furthermore, we define
  \begin{eqnarray}\label{e-rayvaluation}
  v_t:=  \mbox{ the quasi-monomial valuation with weights $(1,t)$ along $X_\infty$ and $E_{\infty}$.}
  \end{eqnarray}
  Denote by $E_k$ the divisor corresponding to $k\cdot v_{\frac{1}{k}}$.
  \end{say}

\begin{defn}
Let $x\in Y$ be a klt singularity. 
We recall that a prime divisor $E$ over a klt point $x\in Y$ is called a {\it Koll\'ar component} (resp. {\it weak Koll\'ar component}) if there is a birational morphism $\mu\colon Y'\to Y$ isomorphic over $X\setminus\{x\}$, such that $E={\rm Ex}(\mu)$, $(Y',E)$ is plt (resp. log canonical) and $-K_{Y'}-E$ is ample over $Y$. In the case of Koll\'ar component, a morphism $\mu$ is called a {\it plt blow-up}. These notions can be considered as a local birational version of special degenerations (resp. weakly special degenerations with irreducible fibers). 
\end{defn}

If we start with a nontrivial  weakly special degeneration $\cX$ of $X$ with an irreducible central fiber,  then we can take the induced $\bZ$-valuation $v$ as in Lemma \ref{l-induced},  the Rees algebra construction gives
$\cR:=\bigoplus_{m\in \bN} H^0(-mrK_{\cX})$ as a $k[t]$-algebra.  Since there are two gradings given by $m$ and $p$, we indeed have a $T=(\mathbb{G}_m)^2$-actions on $\cR$: the relative cone structure corresponds to the action by the coweight $(1,0)$, and the $\mathbb{G}_m$-action from the test configuration corresponds to the coweight $(0,1)$. %We denote the $\cR/t\cR$ by $R_*$ which is also a double graded algebra. 

We can take the (weighted) blow up of $Y$ with respect to the filtration induced by the valuation with weight $(1,1)$ with a exceptional divisor $E_Y$.
 Then the exceptional divisor $E_Y$ is given by 
$${\rm Proj}(\gr_{E_Y} R)={\rm Proj}\bigoplus_{d \in \N} \big(\bigoplus_{m+p=d} \gr_{\cF_{v}}^p R_m\big)$$
and $\gr_{E_Y} R\cong \gr_EY$. Thus $E_Y$ is a $\bG_m$-quotient of $\Spec(R_*)\setminus\{0\}$
and if we denote $\Delta_{E_Y}$ the orbifold divisor, then $(E_Y, \Delta_{E_Y})$ is semi-log-canonical and irreducible with $-K_{E_{Y}}-\Delta_{E_Y}$ being ample, since  ${\rm Spec}(R_*)$  is a cone over $X_0$. By the inversion of adjunction, $\mu'\colon(Y',E_Y)\to Y$ is a weakly Koll\'ar component. 

Conversely, if we start with a $\mathbb{G}_m$-equivariant weak Koll\'ar component $E_Y$ over $x\in Y$ with $\ord_{E_Y}(t)=1$, 
%(that is there exists a morphism $Y_E\to Y$ which only extracts the prime divisor $E$ with $-E$ being ample over $Y$ and $(Y_E,E)$ log canonical), 
we know it arises from a weighted blow up of the pullback of $v$ and $X_{\infty}$ with weight $(1,1)$, where $v$ is  a divisorial valuation $v:=\ord_E$ on $X$. 
Thus the filtration of $R$ by $\ord_E$ is finitely generated, as the associated graded ring  
$$\bigoplus_{m \in \N} \bigoplus_{p \in \Z} \gr_{v}^p R_m \cong \bigoplus_{d\in \N} \mu_*(\mathcal{O}_{Y'}(-dE_Y)/ \mu_*(\mathcal{O}_{Y'}\big(-(d+1)E_{Y})\big). $$
Hence we can take the Rees algebra ${\rm Rees}(\cF_E)$ to get the test configuration $\cX$.

% For a $\mathbb{Q}$-Fano variety $X$, we define the cone $$Y={\rm Spec}\bigoplus_{m\in \mathbb{Z}_{\ge 0}}R_{m}, \mbox{\ \  where \ \ } R_m=H^0(-mrK_X) \mbox{ for some }r\in \mathbb{N}_+.$$

To summarize,  the above discussion gives the following correspondence. 

\begin{prop}\label{p-corresp}
There is a one to one correspondence between 
 \[
\left\{
  \begin{tabular}{c}
\mbox{weakly special degenerations of $X$}\\
with an integral central fiber
  \end{tabular}
\right\} 
\ \longleftrightarrow
 \left\{
  \begin{tabular}{c}
\mbox{$\mathbb{G}_m$-equivariant  weak Koll\'ar}\\
\mbox{components $E_Y$ with $\ord_E(t)=1$}\\
  \end{tabular}
\right\}.
\]
%Moreover, if we denote by $v$ the valuation induced by the irreducible central fiber, then $E$ is the valuation that arises from the weighted blow up of $X_{\infty}$ and the pullback of $v$ with weight $(1,1)$.

Moreover, if we restrict to $\bG_m$-equivariant Koll\'ar components over $Y$, then they precisely correspond to special test configurations for $X$.
\end{prop}

\subsubsection{Local and global K-stability}\label{ss-localglobal}

In \cite{Li17},  by considering  the normalized volume function of the vertex, this construction was related to the study of K-stability question.
At first sight, using the normalized volume function to study the K-stability of $\bQ$-Fano varieties may seem indirect. However, working on $Y$ encodes all information of the anti-canonical ring. A number of new results were first established through this approach, e.g. Theorem \ref{t-maintheorem1} and Theorem \ref{t-delta=1}. We will now explain the main ideas. 

\medskip

 Since the degeneration induced by the canonical valuation $X_{\infty}$ is just $Y$ itself,  one established part of the stable degeneration conjecture implies the following statement, which contains Theorem \ref{t-valkstable}(1). 
 \begin{thm}[{\cite{Li17,LL19, LX20}}]\label{t-maintheorem1} 
 We have the following equivalence
\[
(X\mbox{ is K-semistable})\Longleftrightarrow(\mbox{$v_0:=\ord_{X_{\infty}}$ is a minimizer })\Longleftrightarrow (\beta_X(E)\ge 0, \forall E /X).
\]
 \end{thm}

We already see $(\beta(E)>0, \forall E)\Longrightarrow (X\mbox{ is K-semistable})$ (see Corollary \ref{c-kimpliesbeta}). Now we discuss the other implications in Proposition \ref{p-mimpliesbeta} and Proposition \ref{p-kimpliesm}. 

\subsubsection*{$v_0$ minimizing implies $\beta\ge 0$}

\begin{prop}[{\cite{Li17}}]\label{p-mimpliesbeta}
$(\mbox{$v_0$ is a minimizer })\Longrightarrow(\beta_X(E)\ge 0, \forall E).$
\end{prop}
\begin{proof}
 Let $E$ be a divisor over $X$ that arises on a proper normal model $\mu:Z\to X$. Following  \cite{Li17,LX20} (see \eqref{e-rayvaluation}), $E$ gives rise to a ray of valuations 
\begin{eqnarray}\label{e-ray}
\{v_{t}\ |\  t\in [0,\infty) \subset \Val_{Y,x}\}.
\end{eqnarray}

We have
  \[
  A_{Y} (v_t) =   A_{Y} (\ord_{X_\infty}) + tA_{Y} (\ord_{E_\infty}) = 1/r + at 
  .\]
For $t>0$, 
  \[ \fa_p(v_t) = \underset{m\geq 0 }{\oplus} {\cF_E}^{(p-m)/t}R_m \subseteq R \quad
  \text{ and } 
  \quad
   \fa_p(v_0) = \underset{m\geq p } {\oplus} R_m  \subseteq R
  .\] 
When $k\in \mathbb{N}$, $v_{\frac{1}{k}}=\frac{1}{k}\ord_{E_k}$, where $E_k$ is a divisor over $X$. 

Since $ v_0$ is a minimizer,  $\frac{d \  \hvol(v_t)}{dt}\big\vert_{t=0^+}\ge 0$, which implies $\beta(E)\ge 0$ by Lemma \ref{l-dervolume}.
\end{proof}

\begin{lem}[C. Li's derivative formula]\label{l-dervolume}
\begin{eqnarray}\label{e-betader}
\frac{d }{dt} \hvol(v_t)\bigg\vert_{t=0^+} = (n+1) \beta_{X}(E).
\end{eqnarray}
\end{lem}
\begin{proof}
Let $\fa_{t,p}:=\fa_p(v_t)$. So $\fa_{t,p}$ contains $ \oplus_{m\geq p} R_m$. We claim the following hold:
\begin{eqnarray}\label{e-volume}
\mbox{ $\mult(\fa_{t,\bullet}) =  r^n (-K_X)^n - (n+1) \int_{0}^{\infty} \vol (\cF_E R^{(x)}) \frac{ t\, dx}{(1+tx)^{n+2}}$}.
\end{eqnarray}

This follows from the argument in \cite[(18)-(25)]{Li17}. For the reader's convenience, we give a brief proof.
For $t\in \R_{>0 }$, we have 
\begin{align*}
\mult(\fa_{t,\bullet })&=\lim_{p\to \infty}  \frac{(n+1)!}{p^{n+1}}\dim_k (R/\fa_{t,p})\\
&=\lim_{p\to \infty}  \frac{(n+1)!}{p^{n+1}}\sum^{\infty}_{m=0}\dim_k  (R_m/ \cF_E^{{(p-m)}/{t}} R_m ) \\
&=\lim_{p\to \infty}  \frac{(n+1)!}{p^{n+1}}\sum^{p}_{m=0}\left(\dim_k R_m - \dim_k  \cF_E^{{(p-m)}/{t}} R_m \right)\\
& = \vol(L)  - \lim_{p\to \infty}  \frac{(n+1)!}{p^{n+1}} \sum_{m=0}^p \dim_k \cF_E^{{(p-m)}/{t}} R_m .
\end{align*}
Then we can identify the limit of the summation with the integral in \eqref{e-volume}, where a change of the variable is needed (see \cite[(25)]{Li17}, where  one chooses $c_1=0$, $\alpha=\beta=\frac{1}{t}$). 

Computing the derivative, we have 
$$
\frac{d}{dt}\hvol(v_t) \Bigr \rvert_{t=0^+} =a(n+1)\left(\frac{1}{r}\right)^n\mult(\fa_{0,\bullet} )+ \frac{1}{r^{n+1}} \frac{d}{dt} \left(\mult (\fa_{t, \bullet} )\right)\Bigr\rvert_{t=0^+}.$$
From \eqref{e-volume}, we know $\mult(\fa_{0,\bullet})=r^n(-K_X-\Delta)^n$
and
\begin{align*}
 \frac{d}{dt} \left( \mult(\fa_{t,\bullet} ) \right) \Bigr\rvert_{t=0^+}&=-
 (n+1)\int_0^\infty   \left( \vol(\cF_E R^{(x)})\left( \frac{ 1-tx(n+1)}{(1+tx)^{n+3} } \right) \right) \Bigr\rvert_{t=0^+}dx
.  \end{align*}
Since the latter simplifies to $-(n+1)\int^{\infty}_{0} \vol (\cF_E R^{(x)})dx$, thus
\[
\frac{d}{dt} \left( \mult(\fa_{t,\bullet} ) \right) \Bigr\rvert_{t=0^+}=(n+1)\beta_X(E).
\]
\end{proof}

\subsubsection*{K-semistable implies $v_0$ minimizing}
%\begin{rem} The above construction with the derivative formula is a powerful tool to study the case when $\beta_X(E)=0$. For instance, it plays an essential role in \cite{LWX18} when we prove an K-semistable Fano $X$ has a unique K-polystable degeneration. \end{rem}
For the last implication that $v_0$ is a minimizer if $X$ is K-semistable, we will take the proof from \cite{LX20}. There we try to `regularize' the minimizing valuation to conclude that we only need to compare the normalized volume $v_0:=\ord_{X_{\infty}}$ and those arisen from a special test configuration. 

\begin{prop}[{\cite[Proposition 4.4]{LX20}}]\label{p-kollar}
We have $\hvol(x,Y)=\inf_E \hvol_{Y,x}(E)$ where $E$ run through all $\mathbb{G}_m$-equivariant Koll\'ar components.
\end{prop}
\begin{proof}By \cite[Theorem 7]{Liu18}, we know that 
\begin{eqnarray}\label{e-liu}
\hvol(x,Y)=\inf_{\fa}\mult(\fa)\cdot\lct(Y,\fa)^n,
\end{eqnarray}
where the right hand side runs through all $\fm_x$-primary ideals. For any $\fa$, if we consider the initial degeneration $\fb_k$ of $\fa^k$, then $\fb_{\bullet}$ forms an ideal sequence, and $\mult(\fa)=\mult(\fb_{\bullet})=\lim\frac{\mult(\fb_k)}{k^n}$, and $\lct(Y;\fa)\ge \frac{1}{k}\lct(Y;\fb_k)$. Therefore,
$$ \mult(\fa)\cdot\lct(Y,\fa)\le \inf_k  \mult(\fb_k)\cdot\lct^n(Y,\fb_k) .$$
Therefore, we can restrict the right hand side of \eqref{e-liu} by only running through all $\mathbb{G}_m$-equivariant $\fm_x$-primary ideals. 

Started from any $\mathbb{G}_m$-equivariant $\fm_x$-primary ideal $\fa$, let $c=\lct(Y;\fa)$. Let $Y'\to Y$ be a dlt modification of $(X,c\cdot \fa)$ with reduced exceptional divisor $\Gamma$. We can mimic the last step in the proof of Theorem \ref{t-specialdegeneration}, to show that there is a $\mathbb{G}_m$-equivariant Koll\'ar component $E$ (with the model $Y_E\to Y$), such that $a(E, X, c\cdot\fa)=-1$,
and 
\[
\hvol(E)=((-K_{Y_E}-E)|_E)^{n-1}\le ((-K_{Y'}-\Gamma)|_\Gamma)^{n-1}\le \mult(\fa)\cdot\lct(Y,\fa)^n.
\]
(see \cite[Section 3.1]{LX20}). %Thus we can conclude.
\end{proof}

\begin{prop}\label{p-kimpliesm}
$(X\mbox{ is K-semistable})\Longrightarrow(\mbox{$v_0$ is a minimizer})$.
\end{prop} 
 \begin{proof}Assume $X$ is K-semistable. To show $v_0=\ord_{X_{\infty}}$ is a minimizer of $\hvol_{Y,x}$, by Proposition \ref{p-kollar}, it suffices to show that for any Koll\'ar component $E$ over $x\in Y$, we have $\hvol(E)\ge \hvol(v_0).$ By Proposition \ref{p-corresp}, we know that $E$ will induces a ray $v_t$ containing $a\cdot \ord_E$ for some $a\in \mathbb{Q}$, which corresponds to a special test configuration $\cX$. We can rescale $t$ such that $v_t$ is defined to be the quasimonomial valuation with weights $(1,t)$ along $X_\infty$ and $c\cdot E_{\infty}$, where $c\cdot E_{\infty}$ is the pull back of the divisorial valuation induced by $\cX$ (see Lemma \ref{l-induced}). 
 
Since $X$ is K-semistable, by Lemma \ref{l-beta=w} and \ref{l-dervolume}, we know that 
$$\frac{d}{dt}\hvol{(v_t)}|_{t=0}= (n+1)\cdot \Fut(\cX)\ge 0.$$
Since $\hvol(v_t)$ is a convex function on $t$ by \eqref{e-volume} (or see \cite[Section 3.2]{LX18}),
\[\hvol(\ord_E)=\hvol(a\cdot \ord_E)\ge \hvol(v_0).
\] 
 \end{proof}
 
\end{subappendices}

\section{Filtrations and Ding stability}\label{s-Ding}

In this section, we will discuss the Ding invariants and related stability notions. Unlike the generalized Futaki invariant which can be defined for any polarized variety, the Ding invariant was only defined in the K\"ahler-Einstein setting,  i.e.,  $K_X=\lambda \cdot L$. It was first formulated by Berman in \cite{Ber16} based on the original analytic work of Ding in \cite{Din88}. One key observation made by Fujita in \cite{Fuj18} is that Ding invariant satisfies a good approximation property, and therefore we can define it in a more general context, namely the linearly bounded filtration. This yields another way of proving Theorem \ref{t-valkstable} which was first given in \cites{Fuj19b, Li17}, see Theorem \ref{t-maintheorem2}. 
\subsection{Ding stability}
First we recall the definition of Ding invariant for test configurations introduced in \cite{Ber16}.
\begin{defn}[Ding invariant]\label{d-ding}
Let $\cX$ be a normal test configuration and the notation as in Lemma \ref{l-intersection}. Denote by $ \cD_{\cX,\cL}\sim_{\bQ} -\frac{1}{r} \bar{\cL}-K_{\bar{\cX}/\mathbb P^1}$, and $\cX_0$ the central fiber of $\cX$ over $0$.  Then we define 
\begin{eqnarray}\label{e-ding}
 \Ding(\cX,{\cL})&= &-\frac{ (\frac{1}{r}\bar{\cL})^{n+1}}{(n+1)(-K_X)^n}-1+\lct(\cX, \cD_{\cX,\cL};\cX_0).
\end{eqnarray}
\end{defn}
With the definition of the Ding invariant, we can define various Ding stability notions the same way as K-stability, replacing $\Fut(\cX,\cL)$ by $\Ding(\cX,\cL)$. 
We easily see the following.
\begin{lem}[{\cite{Ber16}}]\label{l-DvsK}Let $(\cX,\cL)$ be a normal test configuration, then $2\cdot \Fut(\cX,\cL)\ge \Ding(\cX,\cL)$, and the equality holds if and only if $(\cX,\cL)$ is a weakly special test configuration.  
\end{lem}

\begin{thm}[{\cite{BBJ15, Fuj19b}}]\label{t-d=k}
For a Fano variety $X$, the notion of Ding-stability (resp. Ding-semistability, Ding-polystability) is equivalent to K-stability (resp. K-semistability, K-polystability).
\end{thm}
\begin{proof}Started from any normal test configuration $(\cX,\cL)$, we can exactly follow the steps in the proof of Theorem \ref{t-specialdegeneration}, but replacing the generalized Futaki invariant by Ding invariants, and show that there is a special test configuration $\cX^{\rm s}$ such that $\Ding(\cX^{\rm s})\le d\cdot \Ding(\cX,\cL)$ for some base change degree $d$. Moreover, the equality holds if and only if $\cX^{\rm s}$ is isomorphic to the normalization of the base change of $(\cX,\cL)$. See \cite{Fuj19b}*{Theorem 3.1}.

Thus Ding-stability (resp. Ding-semistability, Ding-polystability) is the same if we only test on special test configurations, as well as  K-stability (resp. K-semistability, K-polystability) by Theorem \ref{t-specialdegeneration}. Thus by Lemma \ref{l-DvsK}, we know they give the same conditions.
\end{proof}

\subsection{Filtrations and non-Archimedean invariants}

In this section, we try to generalize various invariants to a setting including both test configurations and valuations, namely {\it (multiplicative) linearly bounded  filtrations}. In fact, from a non-Archimedean geometry viewpoint, there is an even more general notion which is {\it the non-Archimedean metric}. See \cite{BHJ17,BBJ15, BoJ18} for a study on this. In this note, we will not discuss this topic. 

\begin{defn}[Filtration]\label{d-filtration}
Let $r$ be a sufficiently divisible positive integer such that $L:=-rK_X$ is Cartier.  Considering graded multiplicative decreasing filtrations $\cF^{t}$ $(t\in \mathbb{R})$ where $R$ is the section ring
\[R =R(X) =   \bigoplus_{m \in \N} R_{m}  = \bigoplus_{m \in \N}  H^0(X, \cO_X(-mrK_X))
\] 
satisfying $\cF^\la R_m = \cap_{\la ' < \la} \cF^{\la'} R_m$  (e.g. $\cF^\la R_m = \cF^{\lceil \la \rceil} R_m$) for  all $\la$, $\cF^{\lambda'} R_m =R_m$ for some $\lambda'\ll 0$ and $\cF^\la R_m=0$ for $\la \gg0$.
All the filtrations we consider are \emph{linearly bounded}, that is to say there exists $e_-\le e_+\in \mathbb R$ so that for all $m \in \N$, $\cF^{xm} R_m=R_m$ for $x\le e_-$ and  $\cF^{xm} R_m=0$ for $x\ge e_+$. 
\end{defn}

\begin{exmp}For a valuation $v$ over $X$, if $A_{X}(v)<+\infty$, the induced filtration $\cF^t_{v}$ as defined in \eqref{e-VtoF} is linearly bounded.
\end{exmp}

\begin{exmp}[{\cite{Nys12, BHJ17}}]\label{expl:filtration from tc}
From any test configuration $(\cX,\cL)$ of $(X,L)$. Assume $r\cL$ is Cartier. We can associate linearly bounded filtrations on $R$ as follow, 
\begin{eqnarray}\label{e-testfiltration}
\cF^pR_m=\{s\in H^0(X,L^{\otimes rm}) \ | \ \ t^{-p}\bar{s}\in H^0(\cX,\cL^{\otimes rm})\}, 
\end{eqnarray}
where $\bar{s}$ is the pull back of $s$ by $X_{\mathbb A^1}\to X$ considered as a meromorphic section of $\cL^{\otimes rm}$; and $t$ is the parameter on $\bA^1$.  
We know $\bigoplus_{p\in \bZ} \cF^pR$ is finitely generated.
\end{exmp}

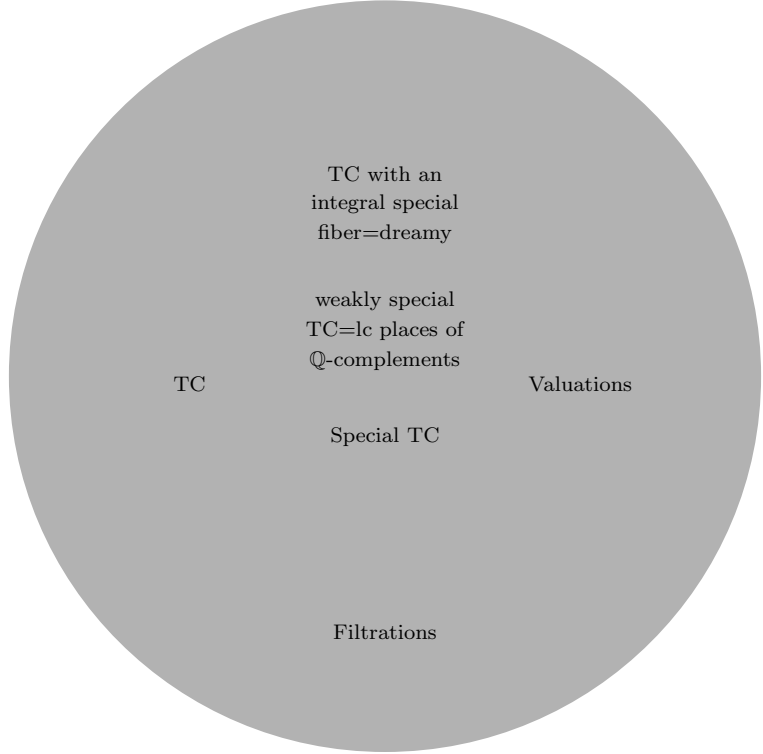
\begin{center}

\begin{tikzpicture}
Use a scope environment to apply a style to a part of the drawing. Here, we apply color blending:

  \begin{scope}[blend group=soft light]

  \fill[black!30!white]  (270:1.4) circle (5);
    \fill[black!30!white]  (270:2.2) circle (0.8);
        \fill[black!30!white]  (270:1.5) circle (1.62);
    \fill[black!30!white] (210:1.7) circle (3.3);
    \fill[black!30!white]  (330:1.7) circle (3.3);
End the scope environment. At the end of the environment, the blending effect will end, because environments keep the settings local:

 \end{scope}
Add nodes with text for the descriptions:

  \node at (210:3)    {{\tiny  TC}};
  \node at (330:3)    {{\tiny  Valuations}};
  \node at (270:2.2) {{\tiny  Special TC}};
    \node at (270:-1.3) {{\tiny  TC with an}};
    \node at (270:-0.9) {{\tiny  integral special}};
        \node at (270:-0.5) {{\tiny  fiber=dreamy}};
 \node at (270:0.4) {{\tiny  weakly special}};
  \node at (270:0.8) {{\tiny  TC=lc places of}};
    \node at (270:1.2) {{\tiny $\mathbb{Q}$-complements}};
 \node at (270:4.8) {{\tiny  Filtrations}};

\end{tikzpicture}

\end{center}

\begin{rem}
Therefore, linearly bounded filtrations give a natural common generalization of test configurations and valuations.  More importantly, in the study of K-stability of Fano varieties, there are natural filtrations appearing, which {\it a priori} arise from neither valuations nor test configurations (e.g. see \cite{BX19, XZ19}). 
\end{rem}

Let $\cF$ be a linearly bounded multiplicative filtration on $R$. Let 
$$\Gr^\lambda_\cF R_m = \cF^\lambda R_m / \bigcup_{\lambda'>\lambda} \cF^{\lambda'} R_m.$$
We define (c.f. \cite{BJ20}*{\S 2.3-2.6})
\[
S_m(\cF):=\frac{1}{m\dim R_m}\sum_{\lambda\in\bR} \lambda \dim\Gr^\lambda_\cF R_m
\]
and $S(\cF)=\lim_{m\to \infty} S_m(\cF)$. Note that the above expression is a finite sum since there are only finitely many $\lambda$ for which $\Gr^\lambda_\cF R_m\neq 0$ and the limit exists by \cite{BC11}. For $x\in \bR$, we set
\[
\vol(\cF R^{(x)})=\lim_{m\to \infty} \frac{\dim \cF^{mx}R_m}{m^n/n!}
\]
where $n=\dim X$ (the limit exists by \cite{LM09}). Then 
$$\nu:=-\frac{1}{(L^n)}\frac{\rd}{\rd x}\vol(\cF R^{(x)})$$ is {\it the Duistermaat-Heckman measure} of the filtration (see \cite{BHJ17}*{\S 5}) and we denote by $[\lambda_{\min}(\cF),$ $\lambda_{\max}(\cF)]$ its support. We also have 
\[
S_m(\cF) = e_- + \frac{1}{\dim R_m}\int_{e_-}^{e_+} \dim \cF^{mx}R_m {\rm d}x
\]
and
\[
S(\cF) = \frac{1}{(L^n)} \int_{\lambda_{\min}(\cF)}^{\lambda_{\max}(\cF)} \vol(\cF R^{(x)}) {\rm d}x = \int_{\bR} x\ \rd \nu.
\]

Then we can generalize other invariants from test configurations to more general linearly bounded filtrations. 
\begin{defn}[Non-Archimedean invariants] \label{defn:D^NA}
Let $X_{\bA^1}=X \times \bA^1$ and $X_0=X\times \{0\}$. Let $\cF$ be a filtration on $R$ and choose $e_-$ and $e_+$ as in Definition \ref{d-filtration} such that $e_-,e_+\in\bZ$.  Let $e=e_+-e_-$ and for each $m\in\bN$. Define {\it the base ideal sequence $I_{m,p}(\cF)$} for a given filtration as following:  $ I_{m,p} $ is the base ideal of the linear system $\cF^{p}H^0(-mrK_X)$, i.e., $I_{m,p}(\cF):={\rm Im} \left(  \cF^p R_m \otimes \cO_{X}(-rmK_X) \to \cO_X\right)$. Then we set
\[
\cI_m := \cI_m(\cF) := I_{m,me_+} + I_{m,me_+-1}\cdot t + \cdots + I_{m,me_-+1}\cdot t^{me-1} + (t^{me}) \subseteq \cO_{X\times \bA^1}.
\]
It is not hard to verify that $\cI_\bullet$ is a graded sequence of ideals. Let
\begin{align*}
    c_m & = \lct(X_{\bA^1},  (\cI_m)^\frac{1}{mr};X_0)\\
     & = \sup\{ c\in\bR \,|\, (X_{\bA^1}, (cX_0)\cdot (\cI_m)^\frac{1}{mr}) \text{ is sub log canonical} \} 
\end{align*}
and we can see $c_\infty = \lim_{m\to \infty} c_m$ exists. We then define
\begin{align*}
    \LNA(\cF) & = c_\infty + \frac{e_+}{r} -1, \\ 
    \DNA(\cF) & = \LNA(\cF) - \frac{S(\cF)}{r},\\
    \JNA(\cF) & = \frac{\lambda_{\max}(\cF)-S(\cF)}{r}.
\end{align*}
%It is not hard to see from the definition that 
%\[
%c_\infty\le 1-\frac{\ord_{X_0}(\cI_\bullet)}{r} \le 1-\frac{e_+ - \lambda_{\max}(\cF)}{r},
%\]
%hence $\DNA(\cF)\le \JNA(\cF)$.
\end{defn}
The following construction is first made by Kento Fujita (see \cite[Sec. 4.2]{Fuj18}). 
\begin{say}[K. Fujita's approximation]
Let $\mu_m\colon \cX_m\to X_{\mathbb{A}^1}$ be the normalized blow up of $\cI_m$ with the exceptional Cartier divisor $\mathcal E_m$. Let $\cL_m:=\mu_m^*(-mrK_X)(-\mathcal E_m)$,
which from the construction can be easily seen to be semi-ample.  Therefore $(\cX_m,\cL_m)$ induces a (semi-ample) test configuration for each $m$. Then we have the following statement.
\end{say}
\begin{lem}[{\cite[Sec. 4]{Fuj18}}] \label{l-dingapproximation}
$\lim_{m\to \infty}\Ding(\cX_m, \cL_m)= \DNA(\cF)$. 
\end{lem}
\begin{proof}Since $\mu_m^*(K_{X_{\bA^1}})+\frac{1}{mr}{\mathcal E}_m\sim_{\mathbb Q}-\frac{1}{mr}\cL_m$. We have 
\begin{eqnarray*}
\lct(\cX_m, \cD_{\cX_m,\cL_m};(\cX_m)_0)&=&  \lct(X_{\bA^1},(\cI_m)^\frac{1}{mr};X_0).
%\min_{E}\frac{A_{X_{\bA^1}}(E)-\mult_{\cI_m} (E) }{\mult_E(X_0)}\\
%&=& \lct(X_{\bA^1},(\cI_m)^\frac{1}{mr};X_0),
\end{eqnarray*}
%where $E$ runs through all $\mathbb G_m$-equivariant divisor over $X_{\bA^1}$ contained in the pull back of $X_0$.
Let $\cF_m$ be the filtration induced by the test configuration $(\cX_m, \cL_m)$ (see Example \ref{expl:filtration from tc}). 
Then the leading coefficient of the weight polynomial for the test configuration $(\cX_m,\cL_m)$ is 
$$\frac{ (\frac{1}{r}\bar{\cL}_m)^{n+1}  }{(n+1)(-K_X)^n}=\frac{S(\cF_m)-e_+}{r}.$$ 
Thus by definition \eqref{e-ding},
\begin{eqnarray*}
 \Ding(\cX_m,{\cL_m})&=& -\frac{ (\frac{1}{r}\bar{\cL}_m)^{n+1}  }{(n+1)(-K_X)^n}-1+\lct(\cX_m, \cD_{\cX_m,\cL_m};(\cX_m)_0)\\
  &=& \DNA(\cF_m). 
\end{eqnarray*}
By \cite[Lem. 4.7]{Fuj18}, we know $\lim S(\cF_m)\to S(\cF)$ and  $c_m\to c_{\infty}$, thus we conclude 
$$\lim_{m\to \infty}\Ding(\cX_m, \cL_m)\to \DNA(\cF).$$ 
\end{proof}
\begin{rem}\label{r-futakifiltration}
The conceptual reason that Lemma \ref{l-dingapproximation} holds is that in the definition of $\Ding$-invariants, only the leading term of the asymptotic formula is needed. As a comparison, for generalized Futaki invariants one also has to consider the second term. In \cite{Nys12}, the viewpoint of filtrations was taken to study test configurations, and it was extended to possibly non-finitely generated filtration in \cite{Sze15}. However, as pointed out in \cite{Nys12}, it is difficult to define the generalized Futaki invariants this way, and the definition in \cite[Def. 4]{Sze15} does not behave well e.g., it could change after taking a truncation. It also relates to \cite[Conj. 2.5]{BoJ18} from the non-Archimedean geometry.
\end{rem}
\begin{cor}\label{c-dingfiltration}
If $X$ is Ding-semistable, then $\DNA(\cF)\ge 0$ for any linearly bounded filtration $\cF$. 
\end{cor}
In  \cite{Fuj19b, Li17}),  Theorem \ref{t-valkstable} was deduced from Corollary \ref{c-dingfiltration}.  In our note,  by following \cite{XZ19}, we will conceptualize the argument using the definition of $\beta$-invariant for a filtration.

\begin{defn}[{\cite[Def. 4.1]{XZ19}}]\label{d-filbeta}
Given a filtration $\cF$ of $R$ and some $\delta\in\bR_+$, we define the \emph{$\delta$-log canonical slope} (or simply \emph{log canonical slope} when $\delta=1$) $\mu_{X,\delta}(\cF)$ as
\begin{equation}
\mu_{X,\delta}(\cF) = \sup \left\{t\in\bR\,|\,\lct(X;I^{(t)}_\bullet)\ge \frac{\delta}{r}\right\}
\end{equation}
where $I^{(t)}_\bullet$ is the graded sequence of ideals given by $I^{(t)}_m := I_{m, tm}(\cF)$. Then we define $$\beta_{X,\delta}(\cF):= \frac{\mu_{X,\delta}(\cF)-S(\cF)}{r}.$$
And when $\delta=1$, we will write $\beta_X(\cF)$.  
\end{defn}

In \cite{XZ19}, the above definition was inspired by the study of a specific filtration, namely the Harder-Narashiman filtration (see Step 1 of the proof of Theorem \ref{t-positivity}). 

\begin{thm}[{\cite[Theorem 4.3]{XZ19}}]\label{t-betading}
We have the inequality $\beta_X(\cF) \ge \DNA(\cF)$ for any bounded multiplicative filtration $\cF$.
\end{thm}
\begin{proof}
Denote by $\mu:=\mu_{X,1}(\cF)$, $\lambda_{\max}:=\lambda_{\max}(\cF)$ and we have $\mu\le \lambda_{\max}$. If $\mu=\lambda_{\max}$, then it is clear that $\beta(\cF)=\JNA(\cF)\ge \DNA(\cF)$. Hence we may assume that $\lambda_{\max}>\mu$ in what follows. In particular, $I_{m,\lambda}\neq 0$ for some $\lambda>\mu m$.

For each $m$, the ideal $I_{m,\mu m+\epsilon}$ does not depend on the choice of $\epsilon>0$ as long as $\epsilon$ is sufficiently small and we set $\fa_m=I_{m,\mu m+\epsilon}$ where $0<\epsilon\ll 1$. It is easy to see that $\fa_\bullet$ is a graded sequence of ideals on $X$. By the definition of $\mu$, we have 
\[
m\cdot \lct(X;\fa_m) = m\cdot \lct(X;I_{m,\mu m+\epsilon})\le \lct(X;I_\bullet^{(\mu+\epsilon /m)})\le \frac{1}{r}
\]
for all $m$. It follows that $\lct(X;\fa_\bullet)\le \frac{1}{r}$ and hence by \cite{JM12}*{Theorem A} there is a valuation $v$ over $X$ such that 
\begin{equation} \label{eq:a<=v}
    a:=A_{X}(v)\le \frac{1}{r}v(\fa_\bullet)<\infty.
\end{equation}
For each $\lambda\in \bR$, we set $f(\lambda)=v(I^{(\lambda)}_\bullet)$. Since $\lambda_{\max}>\mu$, there exists some $\epsilon>0$ such that $f(\lambda)<\infty$ for all $\lambda<\mu+\epsilon$. Since the filtration $\cF$ is multiplicative, we know that $f$ is a non-decreasing convex function. It follows that $f$ is continuous on $(-\infty,\mu+\epsilon)$ and from the construction we see that 
$$f(\mu)\le v(\fa_\bullet)\le \lim_{\lambda\to \mu+} f(\lambda)=f(\mu),$$ hence $f(\mu)=v(\fa_\bullet)\ge ar$ by \eqref{eq:a<=v}. We then have 
\begin{eqnarray}\label{e-convex}
f(\lambda)\ge f(\mu)+\xi (\lambda-\mu)\ge ar+\xi(\lambda-\mu) \mbox{\ \ where \ } \xi:=\lim_{h\to 0+}\frac{f(\mu)-f(\mu-h)}{h}
\end{eqnarray} 
for all $\lambda$ by the convexity of $f$. We claim that $\xi>0$. Indeed, it is clear that $\xi\ge 0$ since $f$ is non-decreasing. If $\xi=0$, then $f$ must be constant on $(-\infty,\mu]$; but this is a contradiction since $f(\mu)\ge ar>0$ while we always have $f(e_-)=0$. Hence $\xi>0$ as desired. Replacing $v$ by $\xi^{-1} v$, we may assume that $\xi=1$ and \eqref{e-convex} becomes
\begin{equation} \label{eq:f convex}
    f(\lambda)\ge \lambda +ar-\mu.
\end{equation}

Now let $\tv$ be the valuation on $X\times \bA^1$ given by the quasi-monomial combination of $v$ and $X_0$ with weight $(1,1)$. Using the same notation as in Definition \ref{defn:D^NA}, we have
\begin{align*}
    \tv(I_{m,me_-+i}\cdot t^{me-i}) & \ge m f\left( \frac{me_-+i}{m} \right) +(me-i) \\
     & \ge m\left(\frac{me_-+i}{m}+ar-\mu\right)+(me-i) \\
     & = m(e_++ar-\mu) \quad (\forall i\in\bN)\ ,
\end{align*}
where the first inequality follows from the definition of $f(\lambda)$ and the second inequality follows from \eqref{eq:f convex}. It follows that $\tv(\tI_m)\ge m(e_++ar-\mu)$ and hence by definition of $c_m$ we obtain
\[
c_m \le A_{(X,\Delta)\times \bA^1}(\tv)-\frac{\tv(\tI_m)}{mr}\le a+1-\frac{e_+ +ar-\mu}{r}=\frac{\mu-e_+}{r}+1
\]
for all $m\in\bN$. Thus $c_\infty\le \frac{\mu-e_+}{r}+1$ and we have
\[
\DNA(\cF) = c_\infty + \frac{e_+ - S(\cF)}{r} -1 \le \frac{\mu-S(\cF)}{r} = \beta(\cF)
\]
as desired.

\end{proof}

\begin{prop}\label{l-betading}
Let $v$ be a valuation over $X$ with $A_{X}(v)<\infty$, and $\cF_v$ the induced filtration on $R$, then 
$$\beta_X(v)\ge \beta_X(\cF_v) \ge \DNA(\cF_v).$$ 

In particular, if $X$ is Ding-semistable, then $\beta_X(v)\ge 0$ for any $A_{X}(v)<\infty$.
\end{prop} 
\begin{proof}To see the first inequality, since $v(I_\bullet^{(t)})\ge t$, we know when $t>rA_X(v)$, then $\lct(X;I_\bullet^{(t)})<\frac{1}{r}$, which implies that $\mu_{1}\le rA_X(v)$ and $\beta(v)\ge \beta(\cF_v)$. The second inequality follows from Theorem \ref{t-betading}.

The last statement then follows from Corollary \ref{c-dingfiltration} which implies $\DNA(\cF_v)\ge 0$ if $X$ is Ding-semistable. 
\end{proof}

Combining Theorem \ref{t-d=k} and Proposition \ref{l-betading}, we have the following theorem.

 \begin{thm}[{\cite{Li17,Fuj19b, XZ19}}]\label{t-maintheorem2} 
 We have the following equivalence
\[
(X\mbox{ is K-semistable})\Longleftrightarrow(X \mbox{is Ding-semistable})\Longleftrightarrow (\beta_X(E)\ge 0, \forall E /X),
\]
and the latter is also equivalent to $\beta_X(\cF)\ge 0$ for any linearly bounded multiplicative filtration on $R:=\bigoplus_m H^0(-rmK_X)$.
 \end{thm}

\bigskip

To summarize, we have the following table indicating various invariants, and where they can be defined.  

\vspace{0.1cm}

{\small

\begin{center}
\begin{tabular}{ | c | c | c | c | c |   }
 \hline
           & test configurations  $\cX$  & filtrations $\cF$ & valuations $v$ \\
 \hline          
\Fut     & Def. \ref{d-futaki} (or \cite{Tia97}, \cite{Don02}) & unknown & unknown \\
 \hline
\Ding   & Def. \ref{d-ding} (or \cite{Ber16}) & Def. \ref{defn:D^NA} (or \cite{Fuj18}) & ok \\
 \hline
$\beta_X$   & ok & Def. \ref{d-filbeta} (or \cite{XZ19}) &Def. \ref{d-betadivisor} (or \cite{Fuj19b}, \cite{Li17}) \\
 \hline
\end{tabular}
\end{center}

}

\medskip

%They all coincide on special test configurations (after suitable normalizations).
Under the correspondence of Theorem \ref{t-weakspecial}, in which for a $\bZ$-valued valuation $v$ that is an lc place of a $\bQ$-complement, which precisely corresponds to a weakly special test configuration $\cX$ with an irreducible central fiber, there is an equality
$$\Ding(\cX)=2\cdot \Fut(\cX)=\beta_X(v) .$$

%Then we can define 
%$$\beta_{X}(\cF):=\lct(X; \fa_{\bullet})(-K_X)^n-\int^{\infty}_{0}\vol(\cF^t)dt.$$

%In \cite{BL18}, it was shown that if $X$ is K-semistable, then $\beta_X(\cF^{t})\ge 0$ for any linear bounded multiplicative decreasing filtrations $\cF^{t}$ $(t\in \mathbb{R})$. Moreover, $\delta$ can be computed as 
%\begin{eqnarray}\label{e-filtration}
%\delta(X)=\inf\frac{\lct(X; \fa_{\bullet})(-K_X)^n}{\int^{\infty}_{0}\vol(\cF^t)dt},
%\end{eqnarray}
%where the infimum runs through all filtrations as above. 

\subsection{Revisit uniform stability}
  We already have seen the equivalent definitions of K-semistability and K-stability (see Theorem \ref{t-valkstable} and Theorem \ref{t-d=k}).  We have seen the following.

\begin{thm}\label{t-uniform}
Notation as above. For a $\bQ$-Fano variety $X$, the followings are equivalent:
\begin{enumerate}
\item (uniform K-stablity) there exists some $\eta>0$, such that for any test configuration $(\cX,\cL)$, 
$$\Fut(\cX;\cL)\ge \eta\cdot  \JNA(\cX;\cL);$$
\item (uniform Ding-stablity) there exists some $\eta>0$, such that for any linearly bounded filtration $\cF$, 
$$\DNA(\cF)\ge \eta\cdot  \JNA(\cF);$$
%\item there exists some $\eta>0$, such that for any valuation $v$ over $X$with $A_{X}(v)<\infty$, $\beta_X(E)\ge \eta\cdot \JNA(E)$;
\item $\delta(X)>1$;
\item $\beta_X(v)> 0$ for any quasi-monomial valuation $v$ over $X$;
\item there exists $\delta> 1$ such that $\beta_{X,\delta}(\cF)\ge 0$ for any linearly bounded filtration $\cF$.
\end{enumerate}
\end{thm}
\begin{proof}The equivalence between (1) and (2) follows from a similar argument as for Theorem \ref{t-d=k} (see \cite{BBJ15, Fuj19b}). The equivalence between (2) and (3) is proved in \cite[Theorem 1.4]{Fuj19b}; between (3) and (4) in  Theorem \ref{t-quasimini}. Finally, the equivalence between (3) and (5) follows from in \cite{XZ19}*{Thm. 1.4(2)}
\end{proof}

It is also natural to look for a `uniform' version when $\Aut(X)$ is not discrete. The key is to define a norm $\JNA$ which should module the group action. In \cite{His16},  the reduced $\JNA_T$-functional for a torus group $T\subset \Aut(X)$ is defined, and play the needed role.

Let $X$ be a $\mathbb Q$-Fano variety with an action by a torus $T\cong \bG_m^s$. Fix some integer $r>0$ such that $L:=-rK_X$ is Cartier and as before let $R=R(X,L)$. Let $M=\Hom(T,\bG_m)$ be the weight lattice and $N=M^*=\Hom(\bG_m,T)$ the co-weight lattice. Then $T$ naturally acts on $R$ and we have a weight decomposition $R_m=\bigoplus_{\alpha\in M} R_{m,\alpha}$ where 
\[
R_{m,\alpha}=\{s\in R_m\,|\,\rho(t)\cdot s = t^{\langle \rho,\alpha \ra}\cdot s\text{ for all }\rho\in N\text{ and }t\in k^*\}.
\]
Consider a $T$-equivariant filtration $\cF$ on $R=\bigoplus_mH^0(X,mL)$, i.e., $s\in \cF^\lambda R$ if and only if $g\cdot s\in \cF^\lambda R$ for any $g\in T$. We then have a similar weight decomposition 
\[
\cF^\lambda R_m=\bigoplus_{\alpha\in M} (\cF^\lambda R_m)_\alpha
\]
where $(\cF^\lambda R_m)_\alpha := \cF^\lambda R_m \cap R_{m,\alpha}$.

\begin{defn}\label{d-quotient}
For $\xi\in N_{\bR}=N\otimes_\bZ \bR$, we define the $\xi$-twist $\cF_{\xi}$ of the filtration $\cF$ in the following way: for any $s\in  R_{m,\alpha}$, we have
$$s\in \cF^\lambda_{\xi} R_m \mbox{ if and only if } s\in \cF^{\lambda_0} R_m \mbox{ where }\lambda_0=\lambda-\langle\alpha, \xi \ra,$$
in other words, 
 $$\cF^\lambda_{\xi} R_m = \bigoplus_{\alpha\in M}  \cF^{\lambda-\langle\alpha, \xi \ra}{R}\cap R_{m,\alpha}.$$
\end{defn}

One can easily check that $\cF_{\xi}$ is a linearly bounded multiplicative filtration if $\cF$ is. 

Let $Z=X  /\!\!/_{\rm chow}T$ be the Chow quotient (so $X$ is $T$-equivariantly birational to $Z\times T$). Then the function field $k(X)$ is (non-canonically) isomorphic to the quotient field of $$k(Z)[M]=\bigoplus_{\alpha\in M} k(Z)\cdot 1^\alpha.$$
For any valuation $\mu$ over $Z$ and $\xi\in N_\bR$, one can associate a $T$-invariant valuation $v_{\mu,\xi}$ over $X$ such that
\begin{eqnarray}\label{e-valuation}
v_{\mu,\xi}(f)=\min_{\alpha}(\mu(f_{\alpha})+\langle \xi, {\alpha}\ra)
\end{eqnarray}
for all $f=\sum_{\alpha\in M} f_{\alpha}\cdot 1^\alpha \in k(Z)[M]$. Indeed, every valuation $v\in \Val^T(X)$ (i.e. the set of $T$-invariant valuations) is obtained in this way (see e.g. the proof of \cite{BHJ17}*{Lemma 4.2}) and we get a (non-canonical) isomorphism $\Val^T(X)\cong \Val(Z)\times N_\bR$. For any $v\in\Val^T(X)$ and $\xi\in N_{\bR}$, we can therefore define the twisted valuation $v_{\xi}$ as follows: if $v=v_{\mu,\xi'}$, then
$$ v_{\xi}:=v_{\mu, \xi'+\xi}. $$
One can check that the definition does not depend on the choice of the birational map $X\dashrightarrow Z\times T$. When $\mu$ is the trivial valuation, the valuations $\wt_\xi:=v_{\mu,\xi}$ are also independent of the birational map $X\dashrightarrow Z\times T$.

\begin{defn}
Let $T$ be a torus acting on a $\mathbb{Q}$-Fano variety $X$. For any $T$-equivariant filtration $\cF$ of $R$, its {\it reduced $\bfJ$-norm} is defined as:
$$\JNA_T(\cF):=\inf_{\xi\in N_\bR}\JNA(\cF_{\xi}).$$
The reduced $\bfJ$-norm $\JNA(\cX,\cL)$ of a $T$-equivariant test configuration $(\cX,\cL)$ of $X$ is defined to be the reduced $\bfJ$-norm of its associated filtration (Example \ref{expl:filtration from tc}).
\end{defn}

\begin{defn}\label{d-reduceduniform}
Notation as above. 
%If $\beta(\wt_{\xi})=0$ for any valuation $\wt_{\xi}$ induced by $\xi\in N_{\bR}(T)$, then
We define the reduced $\beta$ for a $T$-equivariant valuation $v$ with $A_X(v)<\infty$,  which is not of the form $\wt_{\xi}$, to be 
$\delta_{X,T}(v)=1+\sup_{\xi\in N_{\bR}(T)} \frac{\beta(v)}{S_X(v_{\xi})}$ and $\delta_T(X)=\inf_v \delta_{X,T}(v)$ where $v$ runs through all such valuations. 
\end{defn}

We can define a $\bQ$-Fano variety $X$ to be {\it reduced uniformly K-stable} if one of the following is true. 

\begin{thm}\label{t-reducedK}
Notation as above. Let $T\subset \Aut(X)$ be a maximal torus. Then the followings are equivalent:
\begin{enumerate}
\item (reduced uniform K-stablity) there exists some $\eta>0$, such that for any $(\cX;\cL)$, 
$$\Fut(\cX;\cL)\ge \eta\cdot  \JNA_{T}(\cX;\cL);$$
\item (reduced uniform Ding-stablity) there exists some $\eta>0$, such that for any linearly bounded filtration $\cF$,  
$$\DNA(\cF)\ge \eta\cdot  \JNA_{T}(\cF);$$
\item there exists $\delta>1$ such that for any linearly bounded filtration $\cF$, we can find $\xi\in N_{\bR}$ such that $\beta_{X,\delta}(\cF_{\xi})\ge 0$, 
\item $X$ is K-semistable and $\sup_{\xi\in N_{\bR}}\beta_{X}(v_{\xi})> 0$ for any quasi-monomial $v$ which is not induced by the torus, %where $\beta_{X,T}(v)$ is the reduced $\beta$-invariant. 
\item $\delta_{T}(X)>1$. 
\end{enumerate}
\end{thm}
\begin{proof}The equivalence between (1) and (2) are given in \cite{Li19} for test configurations. And (2) was extended to general filtrations in \cite{XZ19}. The last three characterizations were proved in \cite[Thm. 1.4(3) and Thm. A.5]{XZ19}.
\end{proof}

Our definition clearly does not depend on the torus $T$ since any two maximal tori are conjugate to each other. When $X$ is K-semistable, and it is not reduced uniformly K-stable, i.e., $\delta_T(X)=1$ for a maximal torus $T\subset \Aut(X)$, then by \cite{XZ19}*{Thm. A.5}, we can find a $T$-equivariant quasi-monomial valuation $v$ which is not on the torus, such that  $\delta_{X,T}(v)=\delta_X(v)=1$. Thus the following conjecture follows from Conjecture \ref{c-finitegeneration}.
\begin{conj}\label{c-reducedk}
 A $\bQ$-Fano variety $X$ is K-polystable if and only it is  reduced uniformly K-stable.
 \end{conj}

It has been shown in \cite{Li19} that a $\bQ$-Fano variety $X$ is reduced uniformly K-stable if and only if it admits a (weak) K\"ahler-Einstein metric.

\section*{Notes on history}

 K-stability was first defined in \cite{Tia97} for Fano manifolds. A key observation by Tian is to consider all $\mathbb C^*$-degenerations realized in the embeddings of $|-rK_X|\colon X\to\mathbb P^{N_r} $ for all large $r$. Then in \cite{Don02}, Donaldson formulated it in algebraic terms and extended it to all polarized projective varieties. In \cite{RT07},  Ross and Thomas investigated the notion from a purely algebraic geometry viewpoint, and compared it with GIT stability notions. The intersection formula of the generalized Futaki invariant was found in \cite{Wan12, Oda13b}. All these are discussed in Section \ref{ss-definition}. Later the uniform version of K-stability was introduced in \cite{BHJ17,Der16} (see Theorem \ref{t-uniform}(1)).

 In \cite{Oda13}, the MMP was first applied to show that any K-semistable Fano varieties have only klt singularities. Then a more systematic MMP process was introduced in \cite{LX14} to study a family of Fano varieties as in Section \ref{ss-mmp}. Theorem \ref{t-specialdegeneration} was proved there and became one of the major ingredients in the latter development of the concept of K-stability. 
 
The next stage of the development of the foundation theory of K-stability is largely around the valuative criterion.
Different proof of Theorem \ref{t-valkstable} were given and took a few intertwining steps.
In \cite{Ber16}, a corresponding notion of Ding-stability was formulated by Berman inspired by the analytic work of \cite{Din88}. After first only considering divisors that appear {\it on} the Fano variety \cite{Fuj16},  an earlier version of $\beta(E)$ for any divisor $E$ {\it over} $X$ was defined in \cite{Fuj18}. Moreover, it was shown there that Ding semi-stability will imply nonnegativity of $\beta(E)$ (this is the direction which needs more input, as the converse follows easily from \cite{LX14} and a straightforward calculation, see the discussion in Section \ref{s-twoinvariants}). During the proof one key property of Ding invariant was  established in \cite{Fuj18}, namely one can define $\Ding(\cF)$ for any bounded multiplicative filtrations $\cF$ and the corresponding $m$-th truncation $\cF_m$ satisfies that $\Ding(\cF_m)$ converges to $\Ding(\cF)$ (see K. Fujita's approximation Lemma \ref{l-dingapproximation}).  Using filtrations to study K-stability questions was initiated in \cite{Nys12} and extended to general filtrations in \cite{Sze15}, however, there is a difficulty to define the generalized Futaki invariants in this generality (see Remark \ref{r-futakifiltration}).
A striking application found in \cite{Fuj18} is that the exact upper bound of the volumes of $n$-dimensional Ding-semistable Fano varieties is shown to be $(n+1)^n$.  The ultimately correct formulation of $\beta(E)$ as in \eqref{e-beta} was found by Fujita in \cite{Fuj19b} and Li in  \cite{Li17}  independently. It was shown there that the various nonnegativity notions of $\beta$ are equivalent to the corresponding different notions of Ding-stability. The latter was also proved to be the same as the corresponding K-stability notions  in \cite{BBJ15, Fuj19b}, by using the argument in \cite{LX14} to reduce to the special test configuration $\cX$ where $\Ding(\cX)$ and $\Fut(\cX)$ are simply the same. See Section \ref{s-Ding}. Another new feature developed by Li in \cite{Li17} is that he related the K-stability of a Fano variety to the study of the minimizer of the normalized volume function (defined in \cite{Li18}) for cone singularities. For this purpose it is natural for him to extend the original setting which was only for divisorial valuations to all valuations (see Section \ref{ss-valuation}). 
In \cite{Li17}, he also found the derivative formula (see Lemma \ref{l-dervolume}). This allows \cite{LX20} to use a specializing process in the local setting for normalized volume as in Section \ref{ss-cone}, and give an alternative proof of the valuative criterion of K-semistability, which does not have to treat general filtrations. In Section \ref{s-specialvaluation}, we use an argument which is based on some latter developments in \cite{BJ20, BLX19}. In particular, a new perspective is introduced in \cite{BLX19} (partly inspired by \cite{LWX18}), which is to use complements as an auxiliary tool to connect degenerations and valuations (see Theorem \ref{t-weakspecial}). 
In \cite{BHJ17, BBJ15, BoJ18}, these invariants are extended to an even more general setting, namely non-Archimedean metrics. We did not discuss these notions here. 

One conceptual output from the valuative approach, as discussed in Section \ref{sss-delta}, is the formulation of the stability threshold $\delta(X)$ in \cite{BJ20}, and it is shown that it is the same as the invariant defined in \cite{FO18}, which can be estimated in many cases (see Section \ref{s-sing}). Then it became clear to develop the theory further, e.g. proving the equivalence between uniform K-stability and K-stability, one needs to understand the minimizing valuations of $\delta_X$, i.e. the valuations computing $\delta(X)$. In \cite{BLX19}, we proved when $\delta(X)\le 1$, they are always quasi-monomial. In \cite{BX19, BLZ19}, one also proved that if it is a divisorial valuation, then it indeed induces a non-trivial degeneration of $X$. Conjecturally, out of a quasi-monomial valuation, we can produce a divisorial one. When the automorphism group is non-discrete, one can define the notion of reduced uniform K-stability following \cite{His16} which is developed in \cite{Li19, XZ19} (see Theorem \ref{t-reducedK}). The analogue definition of $\delta$ and the study of its minimizer was proceeded in \cite{XZ19}, where we obtained  in this setting a similar result as in \cite{BLX19}.

A striking new development is the local K-stability theory, built on Li's definition of the normalized volume function in \cite{Li18} defined on valuations over a klt singularities, and the study of its minimizer. This gives a local model of K-stability theory and many well studied global question for K-stability of Fano varieties can find its local counterpart. As briefly discussed in Section \ref{ss-normv}, the picture was given by the Stable Degeneration Conjecture which consisted of a number of parts formulated in \cite{Li18} and \cite{LX18}. They were pursued by many works, and by now only one part remains to be open (see Conjecture \ref{c-localfinite}). In \cite{Li17, LL19, LX20}, the cone singularity case is thoroughly  investigated. The existence of the minimizer is proved in \cite{Blu18}, then in \cite{Xu19} it is shown that they are all quasi-monomial and in \cite{XZ20}  the uniqueness of the minimizer (up to rescaling) is confirmed. In the divisorial case, both \cite{LX20, Blu18} independently showed the minimizer is given by a Koll\'ar component; and it was shown in \cites{LX18, LX20}  that in general if assuming the associated graded ring is finite generated, then it indeed yields a K-semistable (affine) log Fano cone. In \cite{Liu18}, by generalizing the argument in  \cite{Fuj18} (see Example \ref{e-bounded}), an interesting inequality is shown to connect the volume of a K-semistable Fano variety with the volume of any arbitrary point on it. We will see its applications in Section \ref{s-moduli}.

\smallskip

Finally, let us remark some progress made in the analytic side. After the solution of the Yau-Tian-Donaldson Conjecture for Fano manifolds (see \cite{CDS, Tia15, Sze16} etc.), one may naturally ask the same question for all $\bQ$-Fano varieties. It seems there are essential difficulties to extend the original argument using metric geometry to the singular case. However, in \cite{BBJ15}, Berman-Boucksom-Jonsson initiated a variational approach, which is strongly inspired by a non-Archimedean geometry viewpoint. This is indeed conceptually closely related to the foundational progress that people made to understand K-stability in the algebraic side. As a result, in \cite{LTW19} and \cite{Li19} the approach of in \cite{BBJ15} was fully carried out, and it was proved that a $\bQ$-Fano variety $X$ with $|{\rm Aut}(X)|<\infty$ (resp. positive dimensional ${\rm Aut}(X)$) has a KE metric if and only if $X$ is uniformly K-stable (resp. reduced uniformly K-stable). So for Fano varieties, to complete the solution of the Yau-Tian-Donaldson Conjecture in the singular case, what remains is to show (reduced) uniform K-stability is the same as K-(poly)-stability. 

\clearpage

\part{K-moduli space of Fano varieties}\label{p-moduli}
In this part, we will discuss some questions about K-(semi,polystable)stable Fano varieties. In the author's opinion, currently there are three central topics in algebraic K-stability theory. 
In Part \ref{p-what}, we have intensively discussed the research topic on understanding  K-stability and related notions. 

 Another two topics are using K-stability to construct a project moduli space, called {\it K-moduli}; and {\it verifying} a given example of $\mathbb{Q}$-Fano variety is K-stable or not.   This will be respectively discussed in Part \ref{p-moduli} and Part \ref{p-example}. Naturally, the progress we achieved in the foundation theory, as discussed in Part \ref{p-what}, will help us to advance our understanding of these two questions. 

\smallskip

To give a general framework for intrinsically constructing  moduli spaces of Fano varieties is a challenging question in algebraic geometry, especially if one wants to find a compactification. So when the definition of K-stability from complex geometry (see \cite{Tia97}) and its algebraic formulation (see \cite{Don02}), first appeared in front of algebraic geometers, though the connection with the existence of K\"ahler-Einstein metric provides a philosophic justification, technically it seemed bold to expect such a notion would be a key ingredient in constructing moduli spaces of Fano varieties, as it is remote from any known approaches of constructing moduli. 

We recall that there are two successful moduli constructions that one can consult with. The first one is the moduli space which parametrizes Koll\'ar-Shepherd-Barron (KSB) stable varieties, that is projective varieties $X$ with semi-log-canonical singularities (slc) and ample $\omega_X$ (see \cite{Kol13b}). The main tools involved in the construction is the Minimal Model Program. While it has been worked out in KSB theory for how to define a family of higher dimensional varieties, which in particular solves all the local issues on defining a family of Fano varieties,  however, there are two main differences between moduli of Fano varieties and moduli of KSB-stable varieties:  firstly, if we aim to find a compact moduli space of Fano varieties, we often have to add Fano varieties with infinite automorphism group, which is a phenomena that does not happen in the KSB moduli case; secondly, a more profound issue is that Minimal Model Program often provides more than one limit for a family of Fano varieties over a punctured curve, thus it is unclear how to find a  MMP theory that picks the right limit.

The second moduli problem is the one parametrizing (Gieseker) semistable sheaves on a polarized projective scheme $(X,\mathcal{O}_X(1))$ with fixed Hilbert polynomial (see \cite[Section 4]{HL10}).  This moduli space is given by the Geometric Invariant Theory which, as we have noted, is not clear how to apply to the moduli of Fano varieties. Nevertheless, there have been a lot of recent works (see e.g. \cite{Alp13}, \cite{AFS17}, \cite{AHH18}) to associate a given Artin stack $\mathcal{Y}$ with a good moduli space $\pi\colon \mathcal{Y}\to Y$, such that the morphism $\pi$ has the properties shared by the morphism from the stack of GIT-semistable locus $[X^{\rm ss}/ G]$ to its GIT-quotient $X/\!\!/G$.   This is the framework we will use to construct the moduli of K-(semi,projective)stable Fano varieties.

\medskip

%\subsection{Main Conjecture}\label{ss-moduli}

The first main theorem is the following, which is obtained using algebo-geometric approach, is a combination of the recent progress \cite{Jia17, LWX18, BX19, ABHX19, Xu19, BLX19, XZ20}.
 \begin{theorem*}[K-moduli]\label{c-kmoduli} We have two following moduli spaces:
 \begin{enumerate}
\item  (K-moduli stack) The moduli functor $\mathfrak{X}^{\rm Kss}_{n,V}$ of $n$-dimensional K-semistable $\mathbb{Q}$-Fano varieties of volume $V$, which sends $S\in {\sf Sch}_k$ to
 \[
\mathfrak{X}^{\rm Kss}_{n,V}(S) = 
 \left\{
  \begin{tabular}{c}
\mbox{Flat proper morphisms $X\to S$, whose fibers are }\\
\mbox{$n$-dimensional K-semistable klt Fano varieties with  }\\
\mbox{volume $V$, satisfying Koll\'ar's condition}
  \end{tabular}
\right\}
\]
is represented by an Artin stack $\mathfrak{X}^{\rm Kss}_{n,V}$ of finite type.

\item (K-moduli space) $\mathfrak{X}^{\rm Kss}_{n,V}$ admits a separated good moduli (algebraic) space $\phi\colon \mathfrak{X}^{\rm Kss}_{n,V}\to X^{\rm Kps}_{n,V}$ (in the sense of \cite{Alp13}), whose closed points are in bijection with $n$-dimensional K-polystable $\Q$-Fano varieties of volume $V$. 
\end{enumerate}
\end{theorem*}
 
We call such moduli spaces to be {\it the K-moduli stack of K-semistable $\mathbb{Q}$-Fano varieties} and {\it the K-moduli space of K-polystable $\mathbb{Q}$-Fano varieties}.  
The main remaining part is the following.
 \begin{conj*}[Properness Conjecture]\label{c-kmoduli} 
The good moduli  space  $X^{\rm Kps}_{n,V}$ is proper.
\end{conj*}

We also have a projectivity theorem (see \cite{CP18, XZ19}) which implies the projectivity of $X^{\rm Kps}_{n,V}$, up to the above Properness Conjecture and Conjecture \ref{c-reducedk}.

\begin{theorem*}[Projectivity] Any proper subspace of $X^{\rm Kps}_{n,V}$ whose points parametrize reduced uniformly K-stable Fano varieties, is projective. 
\end{theorem*}

\section{Artin stack $\mathfrak{X}^{\rm Kss}_{n,V}$}\label{s-ksemi}

In this section, we will show that families of K-semistable Fano varieties with a fixed dimension $n$ and volume $V$ are parametrized by an Artin stack  $\mathfrak{X}^{\rm Kss}_{n,V}$ of finite type. 
The local issue  has been solved in KSB theory (see \cite{Kol21}). The boundedness and openness follows from an interplay between the foundation theory of K-stability as in Part \ref{p-what} and the MMP theory e.g. \cite{HMX14, Bir19}.
\subsection{Family of varieties}
It is a subtle issue to give a correct definition of a family of higher dimensional varieties over a general base. 
While for the construction of the moduli space, we need the definition of a family over a general base to determine the scheme structure, once that is achieved, the rest of the difficulties are all for families over a normal (or even smooth) base. Since $\mathbb{Q}$-Fano varieties only have klt singularities, which is a smaller class of singularities than the singularities KSB varieties allow to have, on our luck all the subtleties brought up by giving the correct definition which we have to face in the construction  of $\mathfrak{X}^{\rm Kss}_{n,V}$ have already been addressed (see \cite{Kol08,Kol21}).
Therefore, in this survey, we will only deal with families as in Definition \ref{d-Qfamily}.

\begin{defn}\label{d-Qfamily}
A \emph{$\Q$-Gorenstein family of  $\bQ$-Fano varieties} $\pi:X \to B$ over a normal base $B$ is composed of a flat proper morphism $\pi : X\to B$ satisfying: 
\begin{enumerate}
\item $\pi$ has normal, connected fibers (hence, $X$ is normal as well)
\item $-K_{X/T} $ is a $\pi$-ample $\Q$-Cartier divisor, and
\item $X_t$ is klt for all $t\in C$. 
\end{enumerate}
\end{defn}

In (see \cite[Theorem 11.6]{Kol16}), it is shown that the above condition on $-K_{X/T}$ being $\bQ$-Gorenstein is equivalent to the volume $(-K_{X_t})^n$ being a local constant on $t$. 

\begin{rem}
For a general base $B$, we can also define a family of $\bQ$-Fano varieties $X$ over $B$. Then one should post the Koll\'ar condition, which requires that for any $m\in \Z$ the reflexive power $\omega^{[m]}_{X/S}$ commutes with arbitrary base change (see \cite[24]{Kol08}).  This condition first appeared in the study of families of KSB stable varieties, and now is well accepted as a right local condition for a family of varieties with dimension at least two over a general (possibly non-reduced) base.  See \cite{Kol08,Kol21}.
% So in this survey, when we say a family of $\mathbb{Q}$-Fano varieties, we always assume Koll\'ar's condition holds for this family.
\end{rem}

\begin{defn}\label{d-KsemiQfamily}
We call $\pi:X \to B$  as in Definition \ref{d-Qfamily} a $\Q$-Gorenstein family of \emph{K-semistable} $\bQ$-Fano varieties if
\begin{enumerate}
\setcounter{enumi}{3}
\item for any $t\in B$, $X_t$ is K-semistable. 
\end{enumerate}
\end{defn}
\begin{rem}
We note here $t$ can be a non-closed point. It was subtle to study K-stability notions over a non-algebraically closed field since they a priori could change after taking a base change of the ground field. Nevertheless, the issue was resolved in \cite{Zhu20}*{Theorem 1.1} where it is shown that a Fano variety $X$ defined over a (possibly non-algebraically closed) field $k$ is K-semistable (resp. K-polystable) over $k$ if and only if $X_{\bar{k}}$ is K-semistable (resp. K-polystable) where $\bar{k}$ is the algebraic closure of $k$.
\end{rem}

\subsection{Boundedness} In moduli problems, boundedness is often a deep property to establish. Fortunately, deep boundedness results in  birational geometry have been established in \cite{HMX14, Bir16, Bir19}. Then we can apply it to obtain the boundedness for K-semistable $\mathbb{Q}$-Fano varieties with volume bounded from below. The implication was first settled in \cite{Jia17}.

\begin{thm}[{Boundedness}]\label{t-boundedness}
Fix $n\in \mathbb N$ and $V>0$. 
All $n$-dimensional K-semistable $\mathbb{Q}$-Fano varieties with volume at least $V$, are contained in a bounded family. 
\end{thm} 

The first proof of Theorem \ref{t-boundedness} in  \cite{Jia17} heavily relies on \cite{Bir19}. Here following \cite{XZ20}, we give a different argument using normalized volume, in particular, we will show that the weaker boundedness theorem as in \cite{HMX14}*{Theorem 1.8} is enough for our purpose. 

First, in the course of generalizing \cite{Fuj18} (see Example \ref{e-bounded}), in \cite[Theorem 1]{Liu18}, Yuchen Liu found that there is an inequality connecting the local volume and global volume for a K-semistable Fano variety. Such an inequality was later generalized to an arbitrary $\mathbb{Q}$-Fano variety $X$ in \cite{BJ20}.
 \begin{thm}[{\cite{Fuj18,Liu18, BJ20}}]\label{t-localtoglobal}
 Let $X$ be a $\mathbb{Q}$-Fano variety, then for any $x\in X$, we have
 \[
 \hvol(x, X)\cdot \big(\frac{n+1}{n} \big)^n\ge (-K_X)^n\cdot\delta(X)^n
 \]
 \end{thm}
\begin{proof}See \cite[Theorem D]{BJ20}.
\end{proof}

Therefore, if we bounded $(-K_X)^n$ and $\delta(X)$ from below, we have a lower bound of  $\hvol(x, X)$.

%In \cite[Section 6.2.2]{LLX18}, by bounding the minimal log discrepancy of a singularity from below using its normalized volume (see Section \ref{s-local}), and then applying the BAB Conjecture proved in \cite{Bir16}.

By the proof of the uniqueness of the minimizer in \cite{XZ20} (see Theorem \ref{t-singularityvolume}(2)), we know that if we take a finite cover $f\colon(y\in Y)\to (x\in X)$ which is \'etale incodimension 1, then 
$$\hvol(y,Y)=\deg(f)\cdot \hvol(x,X)$$
(see \cite[Theorem 1.3]{XZ20}).
Applying this locally to the index-1 cover of $K_X$,  since $\hvol(y,Y)\le n^n$ by \cite{LX19}*{Thm. 1.6}, the Cartier index of any point $x\in X$ is bounded from above by ${n^n}/{\hvol(x,X)}$. Thus the Cartier index of $X$ is bounded from above, therefore we know all such $X$ form a bounded family by \cite[Theorem 1.8]{HMX14}. 

\begin{rem}
A  strong conjecture predicts that $x\in X$ with volume bounded from below, always specialize to a bounded family of singularities with a torus action. 
\end{rem}

\subsection{Openness}\label{ss-openness}
By Theorem \ref{t-boundedness}, there exists a positive integer $M$ such that $-MK_{X}$ is a very ample Cartier divisor for any $n$-dimensional K-semistable $\mathbb Q$-Fano variety $X$, i.e. 
$|-MK_X|\colon X\hookrightarrow \mathbb P^N$ for some uniform $N$. 
Thus there is a finite type Hilbert scheme ${\rm Hilb}(\PP^N)$, such that any embedding $|-MK_X|\colon X\hookrightarrow \mathbb P^N$ gives a point in ${\rm Hilb}(\PP^N)$.
Then there is a locally closed subscheme $W\subset {\rm Hilb}(\PP^N)$ such that a map
$T\to W$ factors through $W$ if and only if the pull back family ${\rm Univ}_T$ is a $\mathbb Q$-Gorenstein family of $\mathbb{Q}$-Fano varieties and $\mathcal{O}(-MK_{{\rm Univ}_T/T})\sim_{T}\mathcal{O}(1)$. 

Next, to know that $\mathfrak{X}^{\rm Kss}_{n,V}$ is an Artin stack of finite type, we will show the K-semistability is an open condition in a $\bQ$-Gorenstein family of $\bQ$-Fano varieties. Thus there is an open subscheme $U\subset W$, and $\mathfrak{X}^{\rm Kss}_{n,V}=[U/{\rm PGL}(N+1)]$.

\medskip

 In \cite{Xu19} and \cite {BLX19}, two different proofs respectively using the normalized volume and $\delta$-invariants are given. Both proofs use Birkar's theorem on the existence of bounded complements \cite{Bir19}, respectively in the local and global case. 
\begin{thm}[{Openness, \cite{Xu19, BLX19}}]\label{t-openness}
If $X \to B$ is a $\mathbb{Q}$-Gorenstein family of $\mathbb{Q}$-Fano varieties, then the locus where the fiber is K-semistable is an open set. 
\end{thm} 

We have indeed shown two stronger statements, each of which implies Theorem \ref{t-openness}. % by the lower semicontinous of $\delta$ and $\hvol$ in a family, shown in \cite{BL17, BL18}.

\begin{thm}\label{t-constructible}
For a point $s\in B$. We have the following:
\begin{enumerate}
\item   If $X \to B$ is a $\mathbb{Q}$-Gorenstein family of $\mathbb{Q}$-Fano varieties, then the function 
$$(s\in B) \to \min\{\delta(X_{\bar{s}}),1\}$$ is a constructible, lower-semicontinous function, and
\item  for a family of klt singularities $\pi\colon (B\subset X)\to B$, the function 
$$(s\in B) \to \hvol({s}, X_{{s}})$$ is a constructible,  lower-semicontinous function. 
\end{enumerate}
\end{thm} 

\begin{proof}[Sketch of Proof for (1)]Following \cite{BLX19}, we give a sketch of the proof for the global result, which is similar to the one for Theorem \ref{t-quasimini}. 
 In \cite{BLX19}, (1) is proved if we replace $\delta(X_{s})$ by $\delta(X_{\bar{s}})$, where $\bar{s}$ is the geometric point over $s\in B$. Then by \cite{Zhu20}, we know
 $$\min\{\delta(X_{s}),1\}= \min\{\delta(X_{\bar{s}}),1\}.$$

In \cite{BL18}, it is showed that $\delta(X_t)$ is a lower semi-continuous function in Zariski topology on  $t\in B$ for a $\mathbb{Q}$-Gorenstein family of $\mathbb{Q}$-Fano varieties $X\to B$. Therefore, we only need to show the constructibility of $\delta(X_t)$.

Combining  Proposition \ref{p-deltaapproximation} and Lemma \ref{l-complementbounded}, for any $X_t$, if $\delta(X_t)\le 1$, then $\delta(X_t)=\inf \frac{A_{E_t}(E_t)}{S_{X_t}(E_t)}$ for all $E_t$ which is an lc place of an $N$-complement of $X_t$, where $N$ only depends on the dimension of $X_t$. 

In particular, we can apply the argument of Paragraph \ref{say-stratification} and conclude there exists a finite type scheme $S/B$ (in particular, $S$ has finitely many components) and a family of divisors $\mathcal{D}\subset X_S:=X\times_BS$  over $S$ with $\mathcal{D}\sim_{S}-NK_{X_S/S}$ such that $(X_S,\mathcal{D})$ admits a fiberwise log resolution 
$$f_S\colon \big(Y, \Delta_Y:={\rm Ex}(f_S)+(f_S)_*^{-1}(\mathcal{D})\big)\to (X_S,\mathcal{D}),$$ i.e. restricting over any $s\in S$, $(Y_s,\Delta_{Y_s})\to (X_s, D_s:=\mathcal{D}_S\times_S\{s\})$ is a log resolution, and for any $t\in B$, and any $N$-complement $D\in |-NK_{X_t}|$, there exists a point $s\in S$, such that $(X_s, \mathcal{D}_s)\cong (X_t, D)$.

 Applying the argument as in Paragraph \ref{say-locallyconstant}, using \cite{HMX13}*{Theorem 1.8}, we conclude that for a fixed $s$, $a_i:=\inf\frac{A_{X_s}(E_s)}{S_{X_s}(E_s)}$, where the infimum runs over all divisors $E_s$ corresponding to points on $W_{\bQ}$,  only depends on the component $S_i$ of $S$ containing $s$. And Lemma \ref{p-deltaapproximation} implies that 
$$\delta(X)=\min \{\ a_i\ |\  \mbox{there exists $s\in S_i$ such that $X_s\cong X$}\}.$$
\end{proof}
As we discuss at the beginning of Section \ref{ss-openness}, by combining Theorem \ref{t-boundedness} and \ref{t-openness}, the following theorem holds. 
 \begin{thm}[Moduli of K-semistable $\bQ$-Fano varieties]\label{c-moduli}
The functor $\mathfrak{X}^{\rm Kss}_{n,V}$ is represented by an Artin stack of finite type.
\end{thm}

\section{Good moduli space $X^{\rm kps}_{n,V}$}\label{s-kpoly}

In this section, we discuss the existence of a good moduli space ${X}^{\rm Kps}_{n,V}$ of $\mathfrak{X}^{\rm Kss}_{n,V}$. Recall an algebraic space 
$Y$ is called  {\it a good moduli space} of an Artin stack $\cY$, if there is a quasi-compact morphism $\pi\colon \cY\to Y$ such that
\begin{enumerate}
\item $\pi_*$ is an exact functor on quasi-coherent sheaves; and
\item $\pi_*(\mathcal{O}_{\cY})=\cO_Y$.
\end{enumerate}
(see \cite[Def. 4.1]{Alp13}). A typical example of good quotient arises from the GIT setting: if a reductive group $G$ acts on a polarized projective scheme $(X,L)$, and let $X^{\rm ss}\subset X$ be the semistable locus, then the stack $[X^{\rm ss}/G]$ admits a good moduli space which is the GIT quotient $X/\!\!/G$. 

For an Artin stack, admitting a good moduli space  is a quite delicate property and it carries strong information of the orbit geometry. In a trilogy of works \cites{LWX18, BX19, ABHX19}, this was established for $\mathfrak{X}^{\rm Kss}_{n,V}$, using %the cone construction (see Section \ref{ss-cone}), 
 the abstract theory developed in \cite{AHH18} and tools from the MMP to obtain finite generation. They key is to show that, for special kinds of pointed surfaces $0\in S$, a family of K-semistable Fano varieties over the punctured surface $S\setminus\{0\}$ can be extended to such a family  over the entire surface $S$. 

\subsection{Separated quotient}
In this section, we will discuss one key property that K-stability grants in the construction of the moduli space.
As we have already seen in Example \ref{ex-quadratic}, a family of $\mathbb{Q}$-Fano varieties $X^{\circ}$ over a punctured curve $C^{\circ}=C\setminus\{0\}$ could have many different fillings to be $\mathbb{Q}$-Fano varieties $X$ over $C$. Therefore, we have to only look at the fillings which are K-semistable. Moreover, since a K-polystable $\mathbb{Q}$-Fano could have an infinite automorphism group, i.e., in general we can not expect the extension family is unique, but what one should expect from separatedness of $ M^{\rm Kps}_{n,V}$ is that any two K-semistable fillings are {\it $S$-equivalent}. 

\begin{defn}
Two K-semistable  $\mathbb{Q}$-Fano varieties $X$ and $X'$ are \emph{S}-equivalent  if they degenerate to a common K-semistable log Fano pair via special test configurations. 
\end{defn}

Thus we aim to show for a punctured family of $\bQ$-Fano varieties, the K-semistable filling is unique up to a $S$-equivalence (Theorem \ref{t-main}).
This has been quite challenging for a while. In \cite{LWX18}, the case when $X \to C$ arises from a test configuration was solved, i.e., it was proved that any two K-semistable degenerations $X$ and $X'$ of a same K-semistable $\mathbb{Q}$-Fano variety, are S-equivalent. Therefore, the orbit inclusion relation for an S-equivalence class of a K-semistable $\bQ$-Fano varieties behaves exactly in the nice way as the GIT situation (see \cite[Thm. 3.5]{New78}). 
In particular, this gives the following description of K-polystable $\mathbb{Q}$-Fano variety as the minimal element in the $S$-equivalence class.
\begin{thm}[{\cite{LWX18}}]\label{t-sequivalence}
A  $\mathbb{Q}$-Fano variety $X$ is \emph{K-polystable} if it is K-semistable and any special test configuration $\cX$ of $X$ with  the central fiber $X_0$ being K-semistable satisfies $X\simeq \cX_0$. 
\end{thm}
The argument in \cite{LWX18} later was improved in \cite{BX19} where we show that any two K-semistable fillings of a $\mathbb{Q}$-Gorenstein families of $\mathbb{Q}$-Fano varieties over a smooth curve are S-equivalent. 
%We will discuss this  in more details below. The following statement is the main result in \cite{BX19}. 
%When $X$ and $X'$ are given by two special test configuration of the same K-semistabe $\mathbb{Q}$-Fano variety, with ${\rm Fut}(X)={\rm Fut}(X')=0$,  this is proved in \cite{LWX18}. 
For this more general situation, we have to elaborate the argument in \cite{LWX18} into a relative setting. We will discuss more details in the proof of the theorem below.

\begin{thm}[{\cite{BX19}}]\label{t-main}
Let $\pi: (X,\Delta) \to C$ and $\pi': (X',\Delta')\to C$ be $\Q$-Gorenstein families of log Fano pairs over a smooth pointed curve $0 \in C$. Assume there exists an  isomorphism
\[
\phi: (X,\Delta) \times_C C^\circ  \to (X',\Delta') \times_C C^\circ
\] 
over $C^\circ\colon = C\setminus \{0 \}$.
If $(X_0,\Delta_0)$ and $(X_0',\Delta_0')$ are K-semistable, then they are S-equivalent.
\end{thm}
\begin{proof}We separate the proof into a few steps.
\medskip

\noindent{\it Step 1: Defining the filtrations.}

Let  $\pi: X \to C$ and  $\pi': X'\to C$
 be $\Q$-Gorenstein families of  $\mathbb{Q}$-Fano varieties over a smooth pointed curve $0 \in C$.  
Assume there exists an isomorphism 
\[
\phi: X \times_C C^\circ  \to X' \times_C C^\circ
\] 
over $C^\circ: C\setminus \{0 \}$ that does not extend to an isomorphism $X \simeq X'$ over $C$. 
After shrinking, we may assume $C$ is affine and there exists a local uniformizer $t$. 

 From this setup, we will construct filtrations on the section rings of the special fibers. 
Set 
\[L := -rK_X
\quad \text{ and } \quad
L':=-rK_{X'},
\]
where $r$ is a positive integer so that $L$ and $L'$ are Cartier. 
For each non-negative integer $m$, set
\begin{align*}
\cR_m &:=H^0(X, \cO_X(mL))   
 & \cR'_{m} &:=  H^0(X', \cO_X(mL') )\\
R_m &:= H^0(X_0,  \cO_X(mL_0) )
& R'_m&:= H^0(X'_0, \cO_X(mL_0) ). 
\end{align*} 
Additionally, set  
$$\cR:= \displaystyle  \oplus_{ m} \cR_m, 
\quad 
R:= \displaystyle  \oplus_{ m} R_m,
\quad
\cR':= \displaystyle  \oplus_{ m} \cR'_m, 
\quad \text{ and } \quad 
R': = \displaystyle \oplus_{ m} R'_m.$$

Fix a common log resolution $\widehat{X}$ of $X$ and $X'$ 
\begin{center}
\begin{tikzcd}
  & \widehat{X}  \arrow[dr, "\psi'"] \arrow[dl,swap,"\psi"]& \\
  X  \arrow[rr, dashrightarrow, "\phi"] & & X'
   \end{tikzcd} 
   \end{center}
   and write $\widetilde{X_0}$ and $\widetilde{X'_0}$ for the birational transforms of $X_0$ and $X'_0$ on $\widehat{X}$. 
 Set 
 \begin{eqnarray}\label{e-defa}
a: = A_{X,X_0}(\widetilde{X'_0}) \mbox{\ \ \ and \ \ \ }a' := A_{X', X'_0 }(\widetilde{X_0}).
\end{eqnarray}
Observe that  $\widetilde{X_0}\neq \widetilde{X'_0}$, since otherwise $\phi$ would extend to an isomorphism over $C$ as $-K_{X}$ and $-K_{X'}$ are ample. Moreover, $a, a'>0$ since $X_0$ and $X_0'$ are klt, and we can apply inversion of adjunction. 

For each $p\in \Z$ and $m\in \N$, set
  \[
  \cF^p \cR_m:= \{ s\in  \cR_m \, \vert \, \ord_{\widetilde{X'_0}} (s) \geq p \},  \mbox{\ \ and \ \ }
    \cF'^p \cR'_m := \{ s\in \cR_m \, \vert \, \ord_{\widetilde{X_0}} (s) \geq p \}
    .\]
We define filtrations of $R$ and $R'$ by setting 
\[
\cF^{p} R_m : = \im( \cF^p \cR_m \to R_m) 
 \mbox{\ \ and \ \ }
\cF'^{p} R'_m : = \im( \cF'^p \cR'_m  \to R'_m),
\]
where the previous maps are given by restriction of sections.  It is straightforward to check that $\cF$ and $\cF'$ are filtrations of $R$ and $R'$. 

Note that a section $s \in R_m$ lies in $\cF^p R_m$  if and only if there exists an extension $\tilde{s} \in \cR_m$ of $s$ such that $\tilde{s} \in  \cF^p \cR_m$. The analogous statement holds for $\cF'$.

\medskip

\noindent{\it Step 2: Relating the filtrations.}

 Since $p^*(X_0) = q^*(X'_0)$ have multiplicity one along $\widetilde{X_0}$ and $\widetilde{X'_0}$, we may write 
\[
K_{ \widehat{X}}  = \psi^*(K_X)  + a \widetilde{X'_0} +  F    \mbox{\ \ and \ \ }
K_{ \widehat{X}}  = \psi'^*(K_{X'})  + a' \widetilde{X_0} + F',
\]
where the components of $\Supp(F) \cup \Supp(F')$ are both $\psi$- and $\psi'$-exceptional. 
Now,
\begin{align*}
\cF^p \cR_m& \simeq H^0\Big(
\widehat{X},\cO_{\widehat{X}} \big(m \psi^* L -p \widetilde{X'_0} \big)\Big)  \\
	&= H^0 \left(
	\widehat{X}, \cO_{\widehat{X}} \big(m \psi'^*L' +(mra- p )\widetilde{X'_0}- mra' \widetilde{X_0}+mr(F-F' ) \big) 
	\right). 
\end{align*}
Hence, for $s\in \cF^p  \cR_m$, multiplying $\psi^*s$ by $t^{mra -p}$ gives an element of 
\[
H^0 \Big(\widehat{X},\cO_{\widehat{X}}\big(m\psi'^*L'  -( mr(a+a')- p) \widetilde{X_0}\big) \Big),\]
which can be identified with  $\cF^{mr(a+a')-p}\cR'_m$. 

As described above, for each $p\in \Z$ and $m\in \N$, there is a map
 \[\varphi_{p,m}:
\cF^{p}\cR_m \longrightarrow 
\cF^{mr(a+a')-p} \cR'_m
,\]
which, when  $\cR_m$ and $\cR'_m$ are viewed as submodules of $K(X)=K(X')$,  sends $s\in \cF^{p}\cR_m$ to $t^{mra-p}({\phi^{-1}})^*(s)$. 
 Similarly, there is a map 
 \[\varphi'_{p,m}:
\cF^{p} \cR'_m \longrightarrow 
\cF^{mr(a+a')-p} \cR_m
,\]
which sends $s' \in \cF^{p} \cR'_m$ to $t^{mra'-p}\phi^*(s')$.

\begin{lem}\label{l-Fcompare}
The map $\varphi_{p,m}$ is an isomorphism. Furthermore, given $s\in \cF^{p} \cR_m$
\begin{itemize}
\item[(1)] $s$ vanishes on $X_0$ if and only if $\varphi_{p,m}(s) \in \cF'^{mr(a+a')-p+1} \cR'_m$, and
\item[(2)]    $\varphi_{p,m}(s)$ vanishes on $X'_0$ if and only if $s\in \cF^{p+1} \cR_m$.
\end{itemize}
\end{lem}
\begin{proof}
The map $\varphi'_{mr(a+a')-p,m}$ is the inverse to $\varphi_{p,m}$, since
 $\varphi'_{mr(a+a')-p,m}\circ \varphi_{p,m}$ is multiplication by $t^{mra'-(mr(a+a')-p) }t^{mra-p}=1$. 
 Hence, $\varphi_{p,m}$ is an isomorphism. 
 
For (1), fix $s\in \cF^{p} \cR_m$ and note that $s$ vanishes on $X_0$ if and only if 
$$
\psi^*s \in H^0\Big(\widehat{X}, \cO_{\widehat{X}} \big( m \psi^* L -p \widetilde{X'_0} -  \widetilde{X_0} \big) \Big)
.$$
 The latter holds precisely when
  \[t^{mra-b}\psi^*s
  \in H^0\Big(
  \widehat{X}, \cO_{\widehat{X}} \big(m \psi'^* L' - (mr(a+a')-p +1) \widetilde{X_0} \big)\Big),\]
 which is identified with $\cF'^{mr(a+a')-p +1 }\cR;_m$. Statement (2) follows from a similar argument. 
 \end{proof}

 \begin{prop}\label{p-isomgrF}
 The maps $(\varphi_{p,m})$ induce an isomorphism of graded rings 
 \begin{eqnarray}\label{e-graded}
\bigoplus_{ m\in \N  }  \bigoplus_{p\in \mathbb{Z} }{\rm gr}_\cF^p R_m 
\overset{\varphi}{\longrightarrow}
 \bigoplus_{m \in \N } \bigoplus_{p \in \mathbb{Z}} {\rm gr}_{\cF'}^{p}R'_m ,
 \end{eqnarray}
that sends the degree $(m,p)$-summand on the left to the degree $(m,mr(a+a')-p)$-summand on the right. 
Hence, 
 $ {\rm gr}_\cF^p R_m$ and ${\rm gr}_{\cF'}^p R'_m$ vanish for $p>mr(a+a')$. 
 \end{prop}

   \begin{proof}
Consider the map
   \begin{equation}\label{e-FF'}
   \cF^p R_m \longrightarrow \gr_{\cF'}^{mr(a+a')-p} R'_m. 
   \end{equation}
 defined as follows.
Given an element of $s\in \cF^p R_m$, choose $\tilde{s} \in \cF^p \cR_m$ such that $\tilde{s}$ is an extension of $s$. 
Now, send $s$ to the image of $\varphi_{p,m}(\tilde{s})$ under the composition of maps 
\[
\cF'^{mr(a+a') - p } \cR'_m \to \cF'^{mr(a+a') - p } R'_m \to \gr_{\cF'}^{mr(a+a')-p} R'_m.
\]
This map can be easily seen well defined.

Using Lemma \ref{l-Fcompare}, we see that \eqref{e-FF'} is surjective and has its kernel equal to $\cF^{p+1} R_m$. Indeed, the surjectivity follows from the fact that $\varphi_{p,m}$ is an isomorphism. The description of the kernel is a consequence of Lemma \ref{l-Fcompare}.2. Therefore, 
$$\gr_{\cF}^p R_m \to \gr_{\cF'}^{mr(a+a')-p} R'_m$$ 
is an isomorphism. The previous  isomorphism induces an isomorphism of graded rings, since $\varphi_{p_1,m_1}(s_1)  \varphi_{p_2,m_2}(s_2)  = \varphi_{p_1+p_2,m_1+m_2}(s_1  s_2)$ for $s_1\in \cF^{p_1}\cR_{m_1}$ and $s_2 \in \cF^{p_2} \cR_{m_2}$.  

To see the vanishing statement, observe that 
 $ {\rm gr}_\cF^p R_m$ and ${\rm gr}_\cF'^pR'_m$ vanish for $p<0$. Hence, the isomorphism of graded rings yields the 
 vanishing for $p>mr(a+a')$. 
 \end{proof}

\begin{rem}\label{r-filtrationcoincide}
The above filtration defined in \cite{BX19} was trying to extend the filtration defined in \cite{BHJ17} (see the proof of  Lemma \ref{l-beta=w}) for test configurations into a more general relative setting. 
In hindsight, this indeed coincides with the canonical filtration introduced in \cite{AHH18} when considering $S$-completeness for vector bundles (see \cite{AHH18}*{Rem. 3.36}). \end{rem}

The connection between the above filtration with K-stability can be seen by the following statement.

\begin{prop}\label{p-TSa}Let 
\[
\beta=ar^n(-K_{X_0})^n-\int^{\infty}_{0} \vol(\cF^tR) dt \mbox{\  and \  } \beta'=a'r^n(-K_{X'_0})^n-\int^{\infty}_{0} \vol(\cF'^tR') dt
\]
Then
$\beta + \beta' = 0$ .
\end{prop}

\begin{proof}
Applying Proposition \ref{p-isomgrF}, we see 
\begin{align*}
\dim \cF^p R_m = \sum_{j=p}^{mr(a+a')} \gr_{\cF}^j R_m 
			&= \sum_{j=0}^{mr(a+a')-p} \gr_{\cF'}^{j} R'_m \\
			& = \dim R_m  - \dim \cF^{mr(a+a')-p+1}R'_m ,
\end{align*} 
for  $p \in \{0, \ldots , mr(a+a')+1\}$.
Therefore, 
\[
\sum_{p=0 }^{mr(a+a')} \dim \cF^p R_m
+
\sum_{p=0 }^{mr(a+a')} \dim \cF'^p R'_m
=mr(a+a') \dim R_m
.\]
Then we conclude by dividing both sides by $\frac{1}{n!}m^{n+1}$ and let $m\to \infty$.  \end{proof}

Let $\fa_{\bullet}(\cF)$ be the base ideal sequences for $\cF$ on $X$, i.e., $\fa_{p}=\fa_p(\ord_{\widetilde{X'_0}})$;  and $\fb_{\bullet}(\cF)$ the restriction of $\fa_{\bullet}(\cF)$ on  $X_0$.  Then the inversion of adjunction implies that
$\lct(X,X_0; \fa_{\bullet}(\cF))=\lct(X_0;\fb_{\bullet}(\cF))$. Since we have $a\ge \lct(X,X_0; \fa_{\bullet}(\cF))$ and $\fb_{p}\supset I_{m,p}:=I_{m,p}(\cF)$ (see Definition \ref{defn:D^NA} for the definition of $I_{m,p}$) for any $m$ (and the equality holds for $m\gg 0$), we have
 $$a\ge \lct(X_0;\fb_{\bullet}(\cF))\ge \mu_{X_0}(\cF).$$
 Similarly, we can define  and get $a'\ge \lct(X'_0;\fb_{\bullet}(\cF'))\ge \mu_{X'_0}(\cF')$.  Now since $X_0$ and $X'_0$ are K-semistable, and $\beta\ge \beta_{X_0}(\cF)\ge 0$ and $\beta'\ge 0$. Thus $\beta = \beta'=0$ and $a= \lct(X_0;\fb_{\bullet}(\cF))= \mu_{X_0}(\cF)$, $a'= \lct(X'_0;\fb_{\bullet}(\cF'))= \mu_{X'_0}(\cF')$.

 \medskip
 
 \noindent{\it Step 3: Finite generations.}
 
 The remaining part of the proof is to show that the graded ring in \eqref{e-graded} is finitely generated and yields normal test configurations, since then by \cite{LX14}, this will imply that the Proj of the graded ring yields a 
K-semistable $\mathbb{Q}$-Fano variety.  In \cite{BX19}, this is  the most involving part of the proof, using the cone construction.  
Based on a better understanding of filtrations for K-stability problems, %e.g. Section \ref{s-specialvaluation}, 
now we have an argument which significantly simplifies the technical one in \cite{BX19} as follows, though the underlying strategy remains the same. \footnote{This argument is suggested by Harold Blum.}

We know that there is a sufficiently large $m$ and sufficiently small $\epsilon$, such that $\lct(X,X_0;\frac{1}{m}I_{m,(a-\epsilon)m})\ge 1$. Thus for a general divisor $D\in \cF^{(a-\epsilon)m}\cR_m$,  $(X,X_0+\frac{1}{mr}D)$ is log canonical. On the other hand,
$$A_{X,X_0+\frac{1}{mr}D}(\widetilde{X'_0})=a-\frac{1}{mr}\ord_{X_0'}(D)\le \epsilon.$$

Thus from \cite{BCHM10}, we know that there exists a model $\mu\colon Z\to X$, which precisely extract $X_0'$, and 
the ring 
$$\bigoplus_{m\in \mathbb N}\bigoplus_{p\in \mathbb N}H^0(Z,\mu^*(-mrK_X)-pX_0')$$ 
is finitely generated. Its restriction on $R$ yields $\oplus_p\cF^pR$, which is finitely generated and the graded ring in \eqref{e-graded} yields test configurations $Y$ and $Y'$ of $X_0$ and $X_0'$.
\begin{claim} $Y$ and $Y'$ are normal test configurations. 
\end{claim}
\begin{proof}To see the claim, we want to use the construction in \ref{sss-Scompleteness}. We know that the above argument yields a family 
$$ \mathsf{X}:={\rm Proj}\bigoplus_m\big(i_*(\pi^{\circ}_*(-rmK_{ \mathsf{X}^{\circ}}))\big)$$
over $\overline{\rm ST}(R)$ (see \eqref{e-STR}). Then the two test configurations $Y$ (resp. $Y'$) are given by $s=0$ (resp. $t=0$). Denote by $\mathsf D$ the closure of $D$ on $\mathsf X$. Since  $\phi$ yields a family $\pi^{\circ}\colon \mathsf{X}^{\circ}\to \overline{\rm ST}(R)\setminus 0$, which is normal, and $ \mathsf{X}\setminus \mathsf{X}^{\circ}$ is of codimension 2, we conclude that $\mathsf{X}$ is normal. 
We consider $\mathsf{X}'$ the family over $\overline{\rm ST}(R)$ obtained by the trivial isomorphism $X\to X$, and $\mathsf{D}'$ the divisor on $ \mathsf{X}'$ which is obtained by gluing $D$ on $X$ using the trivial isomorphism. In particular,  $(\mathsf X', \frac 1m\mathsf D')$ is a $\mathbb{Q}$-Gorenstein family of log canonical CYs over $\overline{\rm ST}(R)$. 

On $\mathsf X'$,  $t=0$ and $s=0$ are two divisors $Z$ and $Z'$ (isomorphic to $X_0\times \mathbb{A}^1$). Let 
$$\epsilon':=A_{\mathsf X',\frac 1m \mathsf  D'+Z+Z'}(Y')=A_{X,\frac 1mD+X_0}(\widetilde{X'_0})\le \epsilon,$$
then $(\mathsf X,\frac 1m \mathsf  D+Y+(1-\epsilon')Y')$ is crepant birationally equivalent to $(\mathsf X',\frac 1m \mathsf  D'+Z+Z')$, which in particular implies that $(\mathsf X,Y+(1-\epsilon')Y')$ is log canonical. 
As $\epsilon'\to 0$, we know $(\mathsf X,Y+Y')$ is log canonical, which implies $Y$ and $Y'$ are normal.
\end{proof}
Since $0=\beta(\cF)\ge \Ding(Y)$ by Theorem \ref{t-betading}, it implies $\Ding(Y)=0$. Thus $Y$ is a special test configuration by Theorem \ref{t-specialdegeneration}, and the special fiber $Y_0$ is K-semistable (see \cite[Lemma 3.1]{LWX18}). Similarly, $X_0'$ degenerates $Y_0$ via the special test configuration $Y'$.
\end{proof}

\subsection{The existence of a good moduli space}
In this section, we will sketch the argument that $\mathfrak X^{\rm Kss}_{n,V}$ admits a good moduli space. For smoothable $\bQ$-Fano varieties, this is done \cite{LWX19}, using the criterion in \cite{AFS17}.
In \cite{AHH18}, two elegant valuative criteria were formulated. This makes checking that a stack of finite type admits a good moduli space a lot more transparent. 

Let $R$ be a DVR with $\eta={\rm Spec}(K)$ the generic point. Let $\cY$ be an Artin stack over $\mathbbm{k}$.
\begin{defn}[{$S$-completeness, see \cite[Def. 3.37]{AHH18}}]
Fix a uniformizer $\pi$ of $R$.  Following \cite[(3.6)]{AHH18}, denote by 
\begin{eqnarray}\label{e-STR}
\overline{\rm ST}(R) := [{\rm Spec}(R[s, t]/(st - \pi))/\mathbb{G}_m],
\end{eqnarray}
where the action is $(s,t)\to (\mu\cdot s, \mu^{-1}\cdot t)$. Let $0=[(0,0)/\mathbb{G}_m]$, then $\overline{\rm ST}(R)\setminus 0$ is isomorphic to the double points curve ${{\rm Spec}(R)}\cup_{{\rm Spec}(K)}{\rm Spec}(R)$. Then a stack $\cY$ is called to be {\it $S$-complete} if any morphism $\pi^{\circ}\colon \overline{\rm ST}(R)\setminus 0\to \cY$ can be uniquely extended to a morphism  $\pi\colon \overline{\rm ST}(R)\to \cY$. 
\end{defn}
\begin{defn}[{$\Theta$-reductivity, see \cite[Def. 3.37]{AHH18}}] Let $\Theta:= [\bA^1/\bG_m]$ with the multiplicative action. Set $0 \in \Theta_R$ to be the unique closed point. 
Then we say $\cY$ is {\it $\Theta$-reductive} if a morphism $\Theta_R \setminus 0 \to \cY$ can be uniquely extended to a morphism $\Theta_R\to \cY$. 
\end{defn}
\begin{thm}[{\cite{AHH18}*{Thm. A}}]Let $\cY$ be an Artin stack of finite type over $\mathbbm{k}$, then $\cY$ admits a good moduli space if $\cY$ is $S$-complete and $\Theta$-reductive.
\end{thm}

In \cite{ABHX19}, we put the argument in \cite{BX19} in the context of \cite{AHH18}, then the results are enhanced, so that one can show any K-polystable $\mathbb{Q}$-Fano has a reductive automorphism group, and moreover, any $S$-closed finite substack of $\mathfrak X^{\rm Kss}_{n,V}$ has a good moduli space. 

\subsubsection{$S$-completeness}\label{sss-Scompleteness}

  We will explain that the $S$-completeness of $\mathfrak X^{\rm Kss}_{n,V} $ essentially follows from Theorem \ref{t-main}, as observed in \cite{ABHX19}.

Let $R$ be a DVR with $\eta={\rm Spec}(K)$ the generic point. Consider two families $X$ and $X'$ of K-semistable $\mathbb{Q}$-Gorenstein Fano varieties over ${\rm Spec}(R)$ such that there is an isomorphism 
\[
\phi\colon X\times_{{\rm Spec}(R)} {\rm Spec}(K) \cong X'\times_{{\rm Spec}(R)  }{\rm Spec}(K).
\]
Thus $\phi$ yields a family $\pi^{\circ}\colon \mathsf{X}^{\circ}\to \overline{\rm ST}(R)\setminus 0$. We want to show this can be indeed extended to a family 
$\pi\colon\mathsf{X}\to \overline{\rm ST}(R)$, which is precisely the claim of S-completeness for the functor of K-semistable Fano varieties with fixed numerical invariants.

Denote by $i\colon \overline{\rm ST}(R)\setminus 0\subset \overline{\rm ST}(R)$ the open inclusion, then since $\pi^{\circ}_*(-rmK_{ \mathsf{X}^{\circ}})$ is a vector bundle on $\overline{\rm ST}(R)\setminus 0$ and $0$ in $\overline{\rm ST}(R)$ is of codimension 2, then
$i_*(\pi^{\circ}_*(-rmK_{ \mathsf{X}^{\circ}}))$ is a vector bundle over $\overline{\rm ST}(R)$. Moreover, a key calculation (see \cite{ABHX19}*{Proposition 3.7}) shows that for any $m$ we have
$$i_*(\pi^{\circ}_*(-rmK_{ \mathsf{X}^{\circ}}))|_{0}\cong  \bigoplus_{p\in \mathbb{Z} }{\rm gr}_\cF^p R_m,$$
(see Proposition \ref{p-isomgrF}). Thus we can define 
$$ \mathsf{X}:={\rm Proj}\bigoplus_m\big(i_*(\pi^{\circ}_*(-rmK_{ \mathsf{X}^{\circ}}))\big),$$
and the rest is identical to Step 3 of the proof of Theorem \ref{t-main} (see Remark \ref{r-filtrationcoincide}). 

An important consequence of S-completeness is the following theorem.
\begin{thm}[{\cite{ABHX19}}]\label{t-reductive}
For any $K$-polystable Fano variety, ${\rm Aut}(X)$ is reductive.
\end{thm}
\begin{proof}This follows from the S-completeness. In fact, if we apply the above discussion to ${\rm Aut}(X)(K)$, then any $g\in {\rm Aut}(X)(K)$ can be used to glue two trivial families $X\times {\rm Spec}(R)$, to get a family $ \mathsf{X} $ over $\overline{\rm ST}(R)$. The special fiber over $0\cong [{\rm Spec}(k)/\mathbb{G}_m]$  is isomorphic to $X$ as it is K-polystable together with a morphism $\lambda\colon G_m\to {\rm Aut}(X)$. We can use $\lambda$ to cook up a trivial torsor $\mathsf{X} _{\lambda}$.
Moreover, we can show   $\mathsf{X} $ and  $\mathsf{X} _{\lambda}$ are isomorphic torsors over $\overline{\rm ST}(R)$, which exactly says there are two elements $a$ and $b$ in ${\rm Aut}(X)(R)$ such that $g=a\cdot \lambda \cdot b$. In other words,  the Iwahori decomposition 
$${\rm Aut}(X)(K)={\rm Aut}(X)(R)\cdot {\rm Hom}(\mathbb{G}_m, {\rm Aut}(X)) \cdot {\rm Aut}(X)(R),$$
holds for ${\rm Aut}(X)$, but this implies that ${\rm Aut}(X)$ is reductive.  
\end{proof}

When $X$ is smooth with a KE metric, the above theorem was proved by Matsushima \cite{Mat57}. When $X$ is a $\mathbb{Q}$-Fano variety with a weak KE metric, this is an important step in the proof of Yau-Tian-Donaldson Conjecture (see \cite{CDS, Tia15}, also see \cite{BBEGZ19}). Theorem \ref{t-reductive} gives a completely algebraic treatment. 

\subsubsection{$\Theta$-reductivity}

To check $\Theta$-reductivity, we need to establish the following 
\begin{thm}[{\cite{ABHX19}, Thm. 5.2}]\label{t-thetareductive}
Let $R$ be a DVR of essentially finite type and $\eta$ the generic point of $\Spec(R)$.
For any $\mathbb{Q}$-Gorenstein family of K-semistable Fano varieties $X_R$ over  $R$, any special K-semistable degeneration $\cX_{\eta}/\bA^1_{\eta}$ of the generic fiber $X_{\eta}$ can be extended to a family of K-semistable degenerations $\cX_R/\bA^1_R$ of $X_R$. \end{thm}
This is proved in \cite{ABHX19}*{Sec. 5} by generalizing the arguments developed in \cite{LWX18}, using a local method. Here we sketch a global argument.
\begin{proof}[Sketch of the proof] Denote by $k$ the residue field of $R$.
Let 
$$\cR:=\bigoplus_m \cR_m=\bigoplus_mH^0(X_R, -rmK_{X_R})$$ for a sufficiently divisible $r$.
The special test configuration $\cX_{\eta}$ induces a special divisor $E_{\eta}$, which yields a filtration $ \cF_{\eta}:=\cF_{E_{\eta}}$ on $R_{K}=\cR\otimes K(R)$.  For each $m$ and $i$, there is a unique extension of $ \cF^{i}\cR_m\subset \cR_m$ of $\cF^{i}_{\eta}(R_K)_m$ as an $R$-submodule, such that  $ \cR_m/\cF^i\cR_m$ is a free $R$-module. Denote  by $\cF_k^{\bullet}R_k$ the restricting filtration of $\cF^{\bullet}$ on $R_k=\cR\otimes k$, i.e. 
$$\cF_k^iR_k={\rm Im}(\cF^i\cR_k\to \cR_k\to R_k).$$  Then $\cF_k^{\bullet}$ yields a multiplicative filtration on $R_k$.

Since $\Fut(\cX_{\eta})=0$, $\mu_{\cF_{\eta}}=S(\cF_{\eta})$. On the other hand, we have $\mu_{\cF_k}\le \mu_{\cF_{\eta}}$ and $S(\cF_k)=S(\cF_{\eta})$, we have $\mu_{\cF_k}\le S(\cF_k)$, which implies $\mu_{\cF_k}=S(\cF_k)$ as $X_k$ is K-semistable (see Theorem \ref{t-maintheorem2}). In particular, $\mu_{\cF_{\eta}}=\mu_{\cF_k}$. 

Then we can mimic the proof of Step 3 of Theorem \ref{t-main} to produce a test configuration $\cX$ which extends $\cX_K$ as follows. For any arbitrary positive $\epsilon$, we can find a divisor $D\in H^0(X_R, -rmK_{X_R})$  such that $(X_R, X_k+\frac{1}{mr}D)$ is a log canonical pair and $A_{X_{\eta},\frac{1}{mr}D_{\eta}}(E_{\eta})<\epsilon$. From this, we can easily conclude that the closure $E_R$ of $E_{\eta}$ is a dreamy divisor over $X_R$, whose induced filtration coincides with $\cF$.  Therefore, we can produce such a test configuration $\cX$. 
\end{proof}
%\begin{rem}
%In hindsight, in \cite{LWX18} we were addressing the problem of extending the family over $\Theta^2\setminus 0$ to a (unique) family over $\Theta^2$, which is a special case for verifying both $S$-completeness and $\Theta$-reductivity.  
%\end{rem}
To summarize, applying the main theorem of \cite{AHH18}, we conclude the existence of {\it the K-moduli space}. 
\begin{thm}[{\cite{ABHX19}}]\label{t-goodmoduli} 
The finite type artin stack $\mathfrak{X}^{\rm Kss}_{n,V}$ admits a separated good moduli space $\phi\colon \mathfrak{X}^{\rm Kss}_{n,V}\to X^{\rm Kps}_{n,V}$.
\end{thm}

We denote by $\mathfrak X^{\rm Kss, sm}_{n,V}$ to be the open locus where the the corresponding K-semistable Fano varieties are smooth, and $\overline{X^{\rm Kps, sm}_{n,V}}$ to be the closure of $\phi(\mathfrak X^{\rm Kss, sm}_{n,V})$ in  $X^{\rm Kps}_{n,V}$, which is the locus parametrizing K-polystable Fano varieties that can be smoothable in a $\mathbb Q$-Gorenstein deformation.

%\begin{rem}
%The isomorphism in Proposition \ref{p-isomgrF} can be viewed as a geometric consequence of Iwahori's Theorem applied to ${\rm PGL}$ on the ambient projective space. 
%\end{rem}

\section{Properness and Projectivity}\label{s-proj}
In this section, we discuss briefly the properness and projectivity of $ X^{\rm Kps}_{n,V}$. 

\subsection{Properness}

The following statement is equivalent to the properness of the good quotient moduli space. 

\begin{conj}[Properness]\label{c-proper}
Any family of K-semistable Fano varieties over a punctured curve $C^{\circ}=C\setminus \{0\}$, after a possible finite base change, can be filled in over $0$ to a family of K-semistable Fano varieties over $C$. 
\end{conj}

%This is a probably the most difficult step in the entire construction. 
When the K-polystable Fano varieties $X_t$ are all smooth for $t\in C^{\circ}$ it follows from that \cite{DS14} a K-polystable/KE limit exists. Similar arguments appear in \cite{CDS, Tia15} in a log setting. It follows that $\overline{X^{\rm Kps, sm}_{n,V}}$ is proper.
In fact, it can be shown that the limit $X_0$ for $t\to 0$ is the Chow limit of $[X_t]\in {\rm Chow}(\mathbb{P}^N)$ induced by {\it Tian's embedding}, i.e.,  embeddings $X_t\to \mathbb{P}^N$ given by $|-mK_{X_t}|$ for $m\gg 0$ with the orthonormal bases under the K\"ahler-Einstein metrics on $X_t$ (see \cite{CDS, Tia15, LWX19}). 

To give a completely algebraic treatment, we can follow the strategy of \cite{AHH18}*{Section 6}, which is an abstraction of the Langton's argument of proving the properness of the moduli space of stable sheaves (see \cite{Lan75}). As a consequence, in \cite{BHLX20} we show that if Conjecture \ref{c-maxidd} is true, then one can define a unique optimal degeneration with respect to a lexicographical order $(\frac{\rm Fut(\cX)}{\lVert \cX \rVert_{\rm m} }, \frac{\rm Fut(\cX)}{\lVert \cX \rVert_{\rm 2} })$, and prove that it yields a \emph{$\Theta$-stratification} (see \cite{HL14}) on the stack of all $\mathbb Q$-Fano varieties.  Then  Conjecture \ref{c-proper} would follow from it.  

\subsection{Projectivity}
On $X^{\rm Kps}_{n,V}$ there is a natural  $\mathbb{Q}$-line bundle, called  the {\it Chow-Mumford line bundle} or the {\it CM-line bundle}. People expect $\lambda_{\rm CM}$  to be positive on  $X^{\rm Kps}_{n,V}$, because when the family parametrizes smooth fibers, the curvature form of the Quillen metric on the CM line bundle is given by the Weil-Peterson form (see e.g. \cite{FS90}). This differential geometry approach is pushed further in \cite{LWX18a} to show $\lambda_{\rm CM}$ is big and nef on $\overline{X}^{\rm Kps, sm}_{n,V}$ and ample on ${X^{\rm Kps, sm}_{n,V}}$. Later in \cite{CP18}, an algebraic approach to study the positivity of $\lambda_{\rm CM}$ was introduced. It is proved that  $\lambda_{\rm CM}$  is nef on any proper space of $\overline{X}^{\rm Kps, sm}_{n,V}$, and ample if all points of the proper spacel parametrize uniformly K-stable Fano varieties.  In \cite{XZ19}, we developed a number of new tools to enhance the strict positivity result of \cite{CP18} to a version which allows the fibers to have a non-discrete automorphism group. 
\begin{defn}[{see e.g. \cite{PT09, LWX18, CP18}}] Let $f : X \to T$ be a
proper flat morphism of varieties of relative dimension $n$ such that the general fiber is normal, and $L$ an $f$-ample $\bQ$-Cartier divisor on $X$. 
Consider the Mumford-Knudsen expansion of $\cO_X(rL)$ for a sufficiently divisible $r$:
$$ {\rm det} f_*(\cO_X (qrL)) \cong \bigotimes^{n+1}_{i=0}\cM^{(^q_i)}_i$$
where $M_i$ are uniquely determined bundles on $T$. 

We define {\it the CM line bundle} to be  
$$\lambda_{f,rL}:= \cM_{n+1}^{n(n+1)+\mu(rL)}\otimes\cM_n^{-2(n+1)},$$
where $\mu(rL)=\frac{-nK_{X_t}\cdot (rL_t)^{n-1}}{(rL_t)^n}$ for a general fiber $X_t$, and $\lambda_{f,L}:=\frac{1}{r^n}\lambda_{f,rL}$ as a $\bQ$-line bundle. It clearly does not depend on the choice of $r$. 

If both $X$ and $T$ are normal and projective, we can write it as an intersection
$$\lambda_{f,L} := f_*\big(\mu(L)\cdot L^{n+1}+ (n + 1)L^n\cdot K_{X/T}\big),$$
and if $L=-K_{X/T}$, we denote $\lambda_{f}:=-f_*(-K_{X/T})^{n+1}$. 

When $T=\mathfrak X^{\rm Kss}_{n,V}$ and $X\to T$ is the universal family, we can descend $\Lambda_{\rm CM}$ to get the CM ($\bQ$)-line bundle on $T=\mathfrak X^{\rm Kss}_{n,V}$, and one can show that $\lambda_{\rm CM}$ can be descent to $X^{\rm Kps}_{n,V}$. 
 \end{defn}

\begin{thm}[{Projectivity, \cites{CP18,XZ19}}]\label{t-positivity} The restriction of $\lambda_{\rm CM}$ on any proper subspace of $X^{\rm Kps}_{n,V}$ whose points parametrize reduced uniformly K-stable Fano varieties, is ample. In particular,  $\lambda_{\rm CM}|_{\overline{X}^{\rm Kps, sm}_{n,V}}$ is ample. 
\end{thm}

Putting together  Properness Conjecture \ref{c-proper} and Conjecture \ref{c-reducedk}, Theorem \ref{t-positivity} predicts that $\lambda_{\rm CM}$ is ample on $X^{\rm Kps}_{n,V}$, i.e. we have
the following implications:

{\small

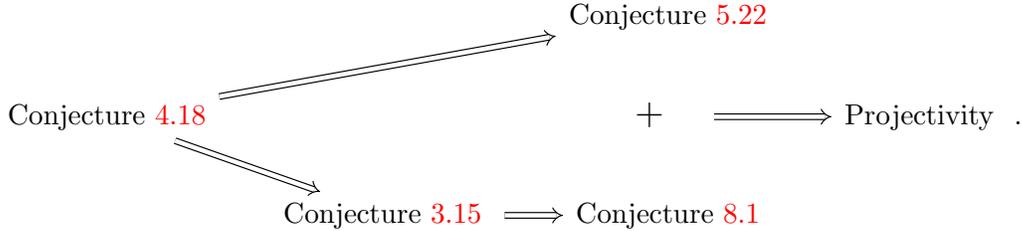
\begin{tikzcd}[column sep=scriptsize]
 & &\mbox{Conjecture \ref{c-reducedk}} \\
\mbox{Conjecture \ref{c-finitegeneration}} \arrow[urr, Rightarrow]  \arrow[rd, Rightarrow] &  & \text{\large$ + $\ \ \ }  \arrow[r, Rightarrow] & {\rm Projectivity \ \  .}\\
&\mbox{Conjecture \ref{c-maxidd} \arrow[r, Rightarrow]}   &  \mbox{Conjecture  \ref{c-proper}}
\end{tikzcd}
}

\begin{proof}[Outline of the proof of Theorem \ref{t-positivity}] We give a sketch of the main ingredients in the proof.

\medskip

 \noindent{\it Step 1: Harder-Narashimhan filtration.}

Consider a family of $\bQ$-Fano varieties $f\colon X\to C$ over a smooth projective curve. One remarkable idea in \cite{CP18} is to relate the positivity of $\lambda_f$ to the  Harder-Narashimhan filtration of $f_*(-mrK_{X/C})$. This becomes more transparent in \cite{XZ19}. One can consider the Harder-Narasimhan filtration $\cF^i_{\rm HN}$ of the $\mathcal{O}_C$-algebra $\cR:=\bigoplus_m f_*(-mrK_{X/C})$, where $\cF^{\la}_{\rm HN}\subset f_*(-mrK_{X/C})$ is defined to be the union of all subbundles with slope at least $\la$. Restricting to a fiber $X_0$, and putting all $\la$ and $m$ together, we obtain {\it the Harder-Narashimhan filtration $\cF^{\la}_{\rm HN}$} on $R=\bigoplus_{m\in \bN}H^0(-mrK_{X_0})$, which can be easily proved to be a multiplicative linearly bounded filtration. 

 In \cite{CP18},  the proof uses the characterization of $\delta$-invariants as the limit of the log canonical thresholds of basis type divisors (see Equation \eqref{e-delta}). In \cite{XZ19}, we turned their argument into an inequality $\deg(\lambda_{f})\ge \beta(\cF_{\rm HN})$, which is equivalent to saying $\mu_{X_0}(\cF_{\rm HN})\le 0$. 
 In fact, for any $t> 0$ and any section of $D_0\in \cF^{tm}_{\rm HN}R_m$ (for $m$ sufficiently large with $tm\ge 2g+1$) can be extended to a section 
 $$D\in |-rmK_{X/C}-f^*P|.$$ Therefore $(X_0,\frac{1}{mr}D_0)$ can not be log canonical since otherwise $ K_{X/C}+D\sim_{\bQ} -\frac{1}{mr}f^*P$ and the pushforward of its multiple would violate the semi-positivity of the pushforward of the log canonical class. This implies that the log canonical slope $\mu_{X_0}(\cF_{\rm HN})\le 0$. 

\medskip

\noindent {\it Step 2: Twist the family.}

This is an extra step which is special for a family of reduced uniformly K-stable $\bQ$-Fano varieties, i.e., when there is a positive dimensional torus acting fiberwisely on the family. 
 
Let $f\colon X\to C$ be a $\bQ$-Gorenstein family of $\bQ$-Fano varieties $f\colon X\to C$ where a general fiber is reduced uniformly K-stable with respect to a torus $T$. Fixed a general fiber $X_t$. Then there exists a $\delta>1$ which only depends on $X_t$ such that for the Harder-Narashimhan filtration  $\cF_{\rm HN}$, by Theorem \ref{t-reducedK} there exists a twist $\xi$ such that $\beta_{\delta}((\cF_{\rm HN})_{\xi})\ge 0$. A somewhat subtle property is such $\xi$ can be chosen to be in $N(T)_{\bQ}$ (see \cite[Proposition 5.6]{XZ19}). After a further finite base change $C'\to C$, we can assume $\xi\in N(T)$.
Consider 
$$\cR:=\bigoplus_{m\in \bN}\cR_m=\bigoplus_{m\in \bN}f_*(-mrK_{X/C})$$
which can be decomposed into weight spaces
$$\cR:=\bigoplus_{m\in \bN}\cR_m=\bigoplus_{m\in \bN,\alpha\in M}\cR_{m,\alpha}.$$
Now we can pick up any point $c\in C$, and construct the twisting family $f_{\xi}\colon X_{\xi}:={\rm Proj}_C(\cR_{\xi})\to C$ where
$\cR_{\xi}=\bigoplus_{m\in \bN,\alpha\in M}\cR_{m,\alpha}\otimes \mathcal{O}_{C}(\langle \xi,\alpha \rangle\cdot c).$ It is easy to check $\lambda_{f}\sim_{\bQ}\lambda_{f_{\xi}}$  as a general fiber is K-semistable, and the Harder-Narashimhan filtration of the twisting family $f_{\xi}\colon X_{\xi}\to C$ is $(\cF_{\rm HN})_{\xi}$ (see \cite[Corollary 5.3]{XZ19}).

Then from $\beta_{\delta}((\cF_{\rm HN})_{\xi})\ge 0$, we can show $-K_{X_{\xi}/C}+\frac{\delta}{(n+1)(-K_X)^n(\delta-1)}f^*_{\xi}(\lambda_f)$ is nef (see \cite[Proposition 4.9]{XZ20}).

\medskip

\noindent{\it Step 3: Ampleness lemma.}

In this step we want to show that if there is a $\bQ$-Gorenstein family $f\colon X\to T$ of $\bQ$-Fano varieties with maximal variation such that general fibers are reduced uniformly K-stable,  then $\lambda_{f}$ is big.
By \cite{BDPP13}, it suffices to show that there exists an $\bQ$-ample line bundle $H$ such that for any covering family of curves $\{C\}$ of $T$, $\lambda_{f}\cdot C\ge H\cdot C$.

Using Koll\'ar's ampleness lemma in \cite{Kol90, KP17} and the product trick, it is shown in \cite{CP18} that this is true if there is a uniform constant $c$ such that $-K_{X_C/C}+c f^*\lambda_{f}$ is nef. (Moreover, they show  if $X_t$ is uniformly K-stable, such uniform constant $M$ exists.) 

In our case, applying the construction of the Step 2,  we can twist the family to get  $X_{C,\xi}\to C$ such that $\beta_{\delta}((\cF_{\rm HN})_{\xi})\ge 0$ for a uniform constant $\delta>1$ (see Theorem \ref{t-reducedK}), which only depends on $X/C$. Then we can take $c=\frac{\delta}{(n+1)(-K_X)^n(\delta-1)}$ and we know that $-K_{X_{C,\xi}/C}+c f^*\lambda_{f}$ is nef.

To conclude, we have to track the proof of the ampleness lemma to strengthen it to a version that includes all twists. For more details, see \cite[Section 6]{XZ19}.
\end{proof}

\begin{rem}[Case of log pairs] Unlike Part \ref{p-what}, in various steps of the moduli theory, to treat the case of log pairs could post substantial new difficulty. Among them, maybe the most difficult one is to give an appropriate definition of a family of log pairs over a general base. This is settled in \cite{Kol19}, by considering {\it K-flat} pairs over the base. While to generalize the boundedness and openness to the log pair case, the argument is essentially the same. To get the good moduli space, extra care has to been took for the subtle degeneration behaviour for the divisors. And for projectivity, it needs one more step to get the positivity of the CM line bundle from the ampleness lemma.  In the KSB case, this is worked out in \cite{KP17}. For K-moduli, see \cite{Pos19} and \cite{XZ19}*{Sec. 7}. 
\end{rem}

\medskip

In the below, we give a comparison of the ingredients appeared in the construction of KSB and K-moduli spaces. 

{\small

\begin{center}
%\mbox{KSB and K-moduli spaces}
\vspace{4mm}

\begin{tabular}{ | c | c | c | c | c |   }
 \hline
      & KSB moduli & K-moduli \\
 \hline       
Local closedness &\cite{Kol08, AH11, Kol19}  &  Contained in the KSB case\\
 \hline       
 Boundedness &\cite{HMX18}  &  \cite{Jia17} after \cite{Bir19} \\     
  & & or \cite{XZ20} after \cite{HMX14}\\
    &  &  (see Theorem \ref{t-boundedness})\\
 \hline
 Openness & --- & \cite{Xu19} or \cite{BLX19} \\
     &  &  (see Theorem \ref{t-openness})\\
 \hline
Separatedness & easy & \cite{BX19} (see Theorem \ref{t-main})\\
 \hline
Existence of the  good & \cite{Kol97, KM97}&\cite{AHH18}, \cite{ABHX19}\\
   moduli  &  &  (see Theorem \ref{t-goodmoduli})\\
 \hline
Properness & \cite{BCHM10, HX13, Kol13}& unknown \\
 \hline
Projectivity and Positivity&\cite{Kol90, Fujino18, KP17} &  \cites{CP18, XZ19} \\
 of CM line bundles  & and \cite{PX17} &  (see Theorem \ref{t-positivity})\\
 \hline
\end{tabular}

\end{center}

}

\section*{Notes on history}

It has been a long mystery for algebraic geometers to find an intrinsic way to construct moduli spaces of Fano varieties. An appropriate definition of families of arbitrary dimensional varieties (or even log pairs)  is made in \cite{Kol21,Kol19}. This was originally worked out for the construction of the KSB moduli, but one can use them to define a family of $\bQ$-Fano varieties (or even a log Fano pairs) as well. However,  more global conditions are needed than only the canonical class being anti-ample, since otherwise, the geometry of the moduli space would be too pathological (e.g. highly non-separated).
Given the canonicity of the K\"ahler-Einstein metric (see \cite{BM87}), one might naturally wonder whether they are parametrized by a nicely behaved moduli space, i.e. the K-moduli (of Fano varieties). In the deep analytic works, e.g. \cite{Tia90,Tia97, Don01, DS14, CDS, Tia15}, one could see some main ingredients of the construction of K-moduli already appeared, but only for smooth KE Fano manifolds and their degenerations.

Built on Donaldson's theorem that all Fano manifolds with a KE metric and a finite automorphism group are asymptotically GIT stable \cite{Don01} as well as \cite{DS14}, one can show such Fano manifolds can be parametrized by a quasi-projective variety  (see \cite{Oda12b, Don15}). Then the K-moduli conjecture for all smoothable K-(semi,poly)stable Fano varieties was intensively studied in \cite{LWX19} (also see \cite{SSY16,Oda15} for results on partial steps), and it was confirmed  except the projectivity which was established a few years later in \cite{XZ19}. All these heavily depend on the analytic tools developed in the solution of Yau-Tian-Donaldson Conjecture in \cite{CDS, Tia15}. 

Another direction started from \cite{Tia92,MM93} is for specific examples, writing down a moduli space which pointwisely parametrizes KE/K-polystable Fano varieties, i.e. constructing explicitly examples of K-moduli spaces. This has been achieved in \cite{MM93, OSS16} for all smoothable del Pezzo surfaces, even before the general construction of $X^{\rm Kps}_{n,V}$ was known. Later it is also done in \cite{SS17} for intersection of two quadratics, and in \cite{LX19, Liu20} for cubic threefolds/fourfolds. In \cite{ADL19, ADL20}, moduli spaces for log Fano pairs were considered, where the authors invented the framework of varying coefficients of the boundary divisor to establish wall-crossings connecting various compactifications. Especially, the examples of low degree plane curves were worked out in details. See Section \ref{s-moduli}.

\medskip

The progress on a purely algebraic method becomes attainable once the valuation criterion was established (see Part \ref{p-what}), combining with main achievements in other branches of algebraic geometry, e.g. MMP, moduli theory etc. In \cite{Jia17}, it was observed that the boundedness of K-semistable Fano varieties with a fixed volume follows from the very general boundedness results established in \cite{Bir19, Bir16}. Later the boundedness is also shown in \cite{XZ20}, which only uses \cite{HMX14} together with the uniqueness of the minimizer of the normalized volume function (see Theorem \ref{t-boundedness}). Then in \cite{BLX19, Xu19}, it is realized that the existence of bounded complements proved in \cite{Bir19} combining with the invariance of log plurigenera proved in \cite{HMX13}, can be used to deduce the openness of the K-semistable locus in the base parametrising a family of Fano varieties (see Theorem \ref{t-openness}). All these results together yield the K-moduli stack $\mathfrak X^{\rm Kss}_{n,V}$ as an Artin stack of finite type, which is indeed a global quotient stack. 

To construct a good moduli space of $\mathfrak X^{\rm Kss}_{n,V}$, we need to understand its orbital geometry, which was established in a trilogy. First in \cite{LWX18}, we prove that the original definitions of K-semistability/K-polystability indeed coincide with the semistability/polystability in the sense of $S$-equivalence classes (see Theorem \ref{t-sequivalence}). In \cite{BX19}, this is generalized to the case of  a relative family and the uniqueness of the $S$-equivalence class of a K-semistable degeneration is proved, which amounts to saying that the good moduli space if exists has to be separated (see Theorem \ref{t-main}). In both papers, the valuative criterion of $\beta$ was adopted, to conclude MMP type facts. Finally, in \cite{ABHX19}, it was realized that the criteria found in \cite{AHH18}  to guarantee the existence of a good moduli, namely $S$-completeness and $\Theta$-reductivity,  can be verified for $\mathfrak X^{\rm Kss}_{n,V}$ based on improvements of the arguments in \cite{BX19} and \cite{LWX18}. As a result, the moduli space $X^{\rm Kps}_{n,V}$ as a separated algebraic space is constructed purely algebraically (see  Theorem \ref{t-goodmoduli}). 

Partial results on positivity of the CM line bundle for families of smoothable KE Fano varieties were  established in \cite{LWX18a} in an analytic manner, using the fact that the curvature form of  Deligne's metric of the CM line bundle is given by a current extending the Weil-Petersson form on the open locus parametrising smooth Fano manifolds.  It was first shown in \cite{CP18}, with an algebro-geometric argument, that the CM line bundle is nef. More precisely, in the current language, for a family $X$ of $\bQ$-Fano varieties over a curve $C$, if we restrict the Harder-Narashimhan $\cF_{\rm HN}$ filtration to a general fiber,  the non-negativity of the CM degree can be deduced from the K-semistability of the fiber, together with the classical results on the positivity of the push forward of pluri-log canonical classes. This implication was analyzed further in \cite{XZ19}, and indeed it inspired the definition for $\beta$ of an arbitrary filtration (see Definition \ref{d-filbeta}). Moreover, in \cite{CP18}, the strict positivity of the CM line bundle was obtained for complete families of uniformly K-stable Fano varieties with maximal variation, where the condition of $\delta>1$ for a general fiber is used to get a uniform nef threshold for $-K_{X/C}$ with respect to $f^*(P)$ so that one can apply a version of Koll\'ar's ampleness lemma as in \cite{Kol90, KP17}. The log case of \cite{CP18} was also treated in \cite{Pos19} (see Theorem \ref{t-positivity}). All these results were extended to families of reduced uniformly K-stable varieties in \cite{XZ19}, where to get the uniform nef threshold,  the key construction of twisting the family, which does not change the CM line bundle on the base, is invented.
\clearpage

\part{Explicit  K-stable Fano varieties}\label{p-example}

Telling whether an explicit given $\mathbb{Q}$-Fano variety is K-(semi,poly)stable is a quite challenging question.  The case of smooth surfaces was solved by Tian in \cite{Tia90} decades ago, but in higher dimension, the knowledge is  incomplete. 
In \cite{CDS}, the authors noted

\medskip

{\it `` On the other hand, we should point out that as things stand at present the result is of very limited use in concrete
cases, so that there is no manifold $X$ known to us, not covered by other existence
results and where we can deduce that $X$ has a K\"ahler-Einstein metric. This is
because it seems a very difficult matter to test K-stability by a direct study of all
possible degenerations. However, we are optimistic that this situation will change
in the future, with a deeper analysis of the stability condition.''}

\medskip

As one will see, the situation indeed has changed a lot since then. With the development of foundational theories, verifying K-stability for Fano varieties becomes a rapidly moving forward subject. 

One guiding question is to look at all smooth hypersurfaces. Since the Fermat hypersurfaces are known to have KE metrics by \cite{Tia00}*{P 85-87} (an algebraic approach is given in \cite{Zhu20}), we know a general smooth hypersurface is K-stable. 
The following folklore conjecture is then natural.
\begin{ques*}
 Are all smooth degree $d$ hypersurfaces with $3\le d \le n$ in $\mathbb{P}^{n+1}$ K-stable? 
\end{ques*}

The recent progress provides  strong tools to verify the K-stability of concrete Fano varieties. For instance, combing \cite{Fuj19, LX19,  AZ20, Liu20}, it is proved that smooth hypersurfaces of degree at least 3 up to dimension four are all K-stable (see Corollary \ref{c-hypersurface}, Theorem \ref{t-index2}, Theorem \ref{t-cubics} and Theorem \ref{t-cubicfourfold}). When dimension is larger than four, besides the simple case of degree 1 and 2,  all degree $n$ and $n+1$ smooth hypersurfaces in $\mathbb{P}^{n+1}$ are known to be K-stable (see \cite{Fuj19, SZ19} or Corollary \ref{c-hypersurface} and Theorem \ref{t-index2}). %In fact, the novel method of verifying K-stability for degree $n$ smooth hypersurfaces in $\mathbb{P}^{n+1}$, as invented in \cite{AZ20}.

\medskip

In the below, we will discuss two ways of showing Fano varieties are K-(semi,poly)stable. The first one is estimating $\delta(X)$ by studying the singularity in $|-K_X|_{\mathbb{Q}}$, and the second approach is constructing explicit K-moduli spaces. 

Both of them had a long history. The first approach dated back to the invention of $\alpha$-invariant in \cite{Tia87} (and also \cite{OS12}), which gave a sufficient condition. Then after the $\delta$-invariant (see Section \ref{sss-delta}) was defined in a similar fashion in \cite{FO18}, people have developed a number of ways to estimate $\delta(X)$ for a Fano variety $X$, see e.g. \cite{Fuj19, SZ19, AZ20}. In particular, the recent paper by \cite{AZ20} provides a powerful approach, namely proving K-stability by `adjunction'.  The second approach, namely, the moduli method, first appeared implicitly in \cite{Tia90} and then explicitly in \cite{MM93}.
Thereafter, various cases were established in \cite{OSS16, SS17, LX19, Liu20}. A new perspective, which considers the setting of varying coefficients and describes the wall crossing birational maps between moduli spaces, was investigated in \cite{ADL19, ADL20} (see also \cite{GMS18}), and put a number of moduli spaces constructed from other perspective into a uniform framework.

\section{Estimating $\delta(X)$ via $|-K_X|_{\mathbb{Q}}$}\label{s-sing}
For a long time, there have been a very limited number of methods to prove that a given Fano manifold admits a KE metric. Among them, probably the most well known one is Tian's $\alpha$-invariant criterion which says if $\alpha(X)>\frac{n}{n+1}$, then $X$ is K-polystable. See \cite{Tia87} (and also \cite{OS12} for an algebraic treatment). 
It established the philosophy that if members in $|-K_X|_{\mathbb{Q}}$ is not `too singular', then $X$ should be close to be K-stable.

By now, we have the new invariant $\delta(X)$, which is precisely computed by the infimum of log canonical thresholds of basis type divisor in $|-K_X|_{\bQ}$ (see \eqref{e-delta}).  Computing $\delta$ for general $X$ is a challenging problem. Nevertheless,  by the machineries we have developed, there are many new cases including which people have speculated for a long time that now one can verify the K-stability.  

\subsection{Fano varieties with index one}
In this section, we discuss two classes of examples, which one can prove the K-stability in a rather straightforward way. However, historically these examples had been mysterious to people for a while. Only after the development of the foundational theory, e.g. Theorem \ref{t-valkstable} was established, the proof became accessible.  
\subsubsection{Fano manifolds with $\alpha=\frac{n}{n+1}$}
The first class of examples provided by the new theory are Fano manifolds with $\alpha(X)=\frac{n}{n+1}$. It has been known  for a long time that any $\mathbb{Q}$-Fano variety $X$ with $\alpha(X)>\frac{n}{n+1}$ is K-stable and $\alpha(X)\ge \frac{n}{n+1}$ is K-semistable (see \cites{Tia87,OS12} or the proof Theorem \ref{t-alphaequal} and Example \ref{e-alpha} for another proof). In \cite{Fuj19}, Fujita used the valuation criterion to study the equality case, and obtained the following somewhat surprising fact. 
\begin{thm}[{\cite{Fuj19}}]\label{t-alphaequal}
If $X$ is an $n$-dimensional smooth Fano manifold, $\alpha(X)= \frac{n}{n+1}$, then either $X\cong \mathbb{P}^1$ or $X$ is K-stable.
\end{thm}
\begin{proof}%We know $E$ is dreamy by Theorem \ref{t-delta=1}. 
If $X$ is not K-stable, then we know there exists a divisor $E$ over $X$ such that $\delta(E)=\delta(X)=1$. We want to show that this implies $X\cong \mathbb{P}^1$.

Consider the restricted volume function 
\[
Q:=-\frac{1}{n}\frac{d}{dt}\vol(-\mu^*K_X-tE)dt \mbox{ \ \ for } t\in [0,\tau),\]
 where $\tau:=\tau(E)$ is the pseudo-effective threshold of $E$ with respect to $-K_X$. The main property we need about $Q$ is that  it is a concave function which follows from that it is equal to the restricted volume and the latter is log concave (see \cite{ELMNP09}). Moreover, $Q$ is smooth on $[0,\tau)$ and can be extended to a continuous function on $[0,\tau]$.
We have
 \begin{eqnarray}
 \frac{\int^{\tau}_0 tQdt}{\int^{\tau}_0Q dt}=\frac{\frac{1}{n}\int^{\tau}_{0}\vol(\mu^*(-K_X)-tE)dt}{\frac{1}{n}(-K_X)^n}=S_X(E)=A_{X}(E)\ge\frac{n}{n+1}\tau,
 \end{eqnarray}
  as $\beta(E)=0$ and $\alpha(X)=\frac{n}{n+1}$. We denote by $A:=A_{X}(E)$.
  
  By the concavity, we know that $Q(t)\ge (\frac{t}{A})^{n-1}Q(A)$ for $t\in [0,A]$ and $Q(t)\le (\frac{t}{A})^{n-1}Q(A)$ for $t\in [A,\tau].$
  So 
  \begin{eqnarray*}
  0&=&\int^{\tau}_{0}(t-A)Q(t)dt\\
  &\le& Q(A)\int^{\tau}_{0}(t-A)(\frac{t}{A})^{n-1}dt\\
  &=&\frac{Q(A) \tau^n}{A^{n-1}}(\frac{\tau}{n+1}-\frac{A}{n})\\
  &\le &0.
  \end{eqnarray*}
This implies that $A=\frac{n}{n+1}\tau$ and $Q(t)= (\frac{t}{A})^{n-1}Q(A)$ for $t\in[0,\tau]$.

To proceed, since $E$ computes $\delta(X)=1$, we know we can find a model $\mu\colon Y\to X$ only extracting $E$, so we know that for $t\ll 1$,
$$Q=-\frac{1}{n}\frac{d}{dt}\vol(-\mu^*K_X-t E)|_{t=t_0}=E\cdot(-\mu^*K_X-tE)^{n-1}.$$
Compared to $Q(t)= (\frac{t}{A})^{n-1}Q(A)$, we know $E^i\cdot (\mu^*K_X)^{n-i}=0$ and all $0\le i \le n-1$, which in particular implies $E$ is mapped to a point.   Moreover, $Q=t^{n-1}(-E|_E)^{n-1}$, which implies for $t\in [0,\tau]$,
$$\vol(-K_X-tE)=n\int^{\tau}_{t}Qdu=(\tau^n-t^n)E(-E)^{n-1}=(-K_X)^n-t^nE(-E)^{n-1}.$$
 This implies $\tau=\epsilon$, which is the nef threshold of $E$ with respect to $-K_X$. 
 
 %Since $X$ is smooth $A\ge n$, as $X$ is smooth, so $\tau\ge n+1$. But this implies $X\cong \mathbb{P}^n$ by \cite{LZ18}. Then $\alpha(\mathbb{P}^n)=\frac{1}{n+1}$.
  
  Since $E$ is dreamy by Theorem \ref{t-delta=1}, we know $\mu^*(-K_X)-\tau E$ is semi-ample but not big. Since $\mu^*(-K_X)-\tau E$ is ample on $E$, we know a sufficiently divisible multiple of $\mu^*(-K_X)-\tau E$ will give a fibration struction $\rho\colon Y\to Z$, whose restrict on $E$ is finite. Thus a general fiber of $\rho$ is a curve $l$. Since $Y$ is normal, $l$ is in the smooth locus. Thus $K_Y\cdot l=-2$ and $0=(\mu^*(-K_X)-\tau E)\cdot l=2-(1+\frac{1}{n+1}\tau) E\cdot l$, which implies $E\cdot l=1$, $\tau=n+1$ and $A=n$. So $Y\to X$ is the blow up of the smooth point, and we know the Seshadri constant at this point is $n+1$, therefore $X\cong \mathbb{P}^n$ by \cite{LZ18}. Then $\alpha(\mathbb{P}^n)=\frac{1}{n+1}$.
\end{proof}
\begin{rem}The smooth assumption in Theorem \ref{t-alphaequal} is indeed necessary. In \cite{LZ19}, they found a class of singular $\mathbb{Q}$-Fano varieties $X$ with $\alpha(X)=\frac{n}{n+1}$, but they are only strictly K-semistable. 
\end{rem}
\begin{cor}\label{c-hypersurface}
For $n\ge 2$, smooth degree $n+1$ hypersurfaces in $\mathbb{P}^{n+1}$ are all K-stable.
\end{cor}
\begin{proof}For any such hypersurface $X$, we have $\alpha(X)\ge\frac{n}{n+1}$ (see e.g. \cite{Che01}), thus we can apply Theorem \ref{t-alphaequal}.
\end{proof}

\subsubsection{Birationally Superrigid $\mathbb{Q}$-Fano varieties}

A special class of Fano varieties, called {\it birationally superrigid Fano varieties}, has been studied for a long time, dated back to Gino Fano.  
\begin{defn}
A Fano variety $X$ with terminal $\bQ$-factorial singularities, Picard number one  is said to be {\it birationally superrigid }if
 every birational map
$f \colon X \dasharrow Y$ from $X$ to a Mori fiber space $Y$ is an isomorphism. 
\end{defn}

The {\it Noether-Fano method} is a criterion to determine a Fano variety to be birationally superrigid.  We have the following equivalent characterization of birational superrigidity.
\begin{thm}[Noether-Fano method, {see \cite[Section 5]{KSC04}}]\label{t-NFmethod}Let $X$ be a terminal Fano variety, which is $\mathbb{Q}$-factorial with Picard number one. Then it is birationally
superrigid if and only if for every movable boundary $M \sim_{\mathbb{Q}} -K_X$ on $X$, the pair $(X, M)$ has canonical singularities.
\end{thm}

By the above characterization, we know being birationally superrigid posts a very strong restriction on the mobile linear subsystem of $|-K_X|_{\mathbb{Q}}$. So it is natural to ask whether all of them are K-stable. After some early partial result in \cite{OO13}, in a recent work \cite{SZ19}, the question is affirmatively answered under a mild assumption. 
\begin{thm}[\cite{SZ19}]Let $X$ be a birationally superrigid. Assume $\alpha(X)>\frac{1}{2}$ (resp. $\alpha(X)\ge \frac{1}{2}$), then $X$ is K-stable (resp. K-semistable). 
\end{thm}
\begin{proof}The analysis is similar to the one in the proof of Theorem \ref{t-alphaequal}. By Theorem \ref{t-NFmethod}, we know there is only one $\bQ$-divisor $D\sim_{\mathbb{Q}}-K_X$ with irreducible support, such that $\ord_{F}D>A_{X}(F)=:A$. 
Assume $F$ such that $\beta(F)<0$. We denote the pseudoeffective threshold by $\tau=\ord_{F}(D)$.

We define $Q$ the restricted volume function as in Theorem \ref{t-alphaequal}, and by our assumption, we have
$$b:= \frac{\int^{\tau}_0 tQdt}{\int^{\tau}_0Q dt}={S_X(F)}>A.$$
Thus we have $0< A< b < \tau$. 
Using the concavity again (see Theorem \ref{t-alphaequal}), we know for $t\in [0,A]$, $Q(t)\ge (\frac{t}{A})^{n-1}Q(A)$. For $t\in [A,\tau]$,  $Q(t)=\big(\frac{\tau-t}{\tau-A}\big)^{n-1}Q(A)$
since 
$$-\mu^*(K_X)-tF=\frac{\tau-t}{\tau-A}(-\mu^*(K_X)-A\cdot F)+\frac{t-A}{\tau-A}(-\mu^*(K_X)-\tau\cdot F).$$
Thus 
\begin{eqnarray*}
  0&=&\int^{\tau}_{0}(t-b)Q(t)dt\\
  &\le & \int^{A}_{0}(t-b) (\frac{t}{A})^{n-1}Q(A)dt+\int^{\tau}_{A}(t-b)\big(\frac{\tau-t}{\tau-A}\big)^{n-1}Q(A)dt\\
  &=& Q(A)\tau(\frac{\tau-2A}{n(n+1)}+\frac{A-b}{n} )<0.
\end{eqnarray*}
since  $\frac{A}{\tau}\ge \alpha(X)\ge \frac{1}{2}$. And the K-stable case is proved in the same way. 
\end{proof}

Combining with \cite{deF16}, this gives a different proof on Corollary \ref{c-hypersurface}. So far all known birationally superrigid Fano varieties satisfy $\alpha(X)>\frac{1}{2}$.

\subsection{Index 2 Fano hypersurfaces}
In this section, we will discuss the work in \cite{AZ20} which introduced an induction framework to verify K-stability of Fano varieties, using some ideas from \cite{Zhu20a}. This is a quite  powerful method. While we will start with the general strategy,  by the end we will focus on one main case, which is given in the following theorem. 
\begin{thm}[{\cite{AZ20}}]\label{t-index2}
Any degree $n$ smooth hypersurface $X$ in $\mathbb{P}^{n+1}$ is K-stable.
\end{thm}
This was only previously known for $n\le 3$ (see \cite{Tia90, LX19}). In the rest of this section, we will give a sketch of the proof of Theorem \ref{t-index2}, following \cite{AZ20}.

\subsubsection{Double filtrations}

The starting point in \cite{AZ20} is the following useful observation that if we can always choose basis which are compatible for two filtrations.  
\begin{lem}\label{l-doublefil}
Let $V$ be a finite dimensional vector space, and $\cF$ and $\cG$ are two filtrations of $V$ by subspaces, then we can picture up a basis $\{e_1,...,e_m\}$ of $V$ compatible with both $\cF$ and $\cG$, i.e. for any $i$ (resp. $j$), we can find a subset of $\{e_1,...,e_m\}$  which forms a basis of $\cF^iV$ (resp. $\cG^j V$). 
\end{lem}
\begin{proof}Left to the reader. 
\end{proof}
So when we compute $S(\cF)$, we can always choose an auxiliary filtration $\cG$.
\begin{exmp}[{\cite[Lem. 4.2]{AZ20}}]\label{e-alpha}
We choose $\cG$ to be the filtration induced by a general element $H\in |-rK_X|$ for a sufficiently divisible $r$. Then for any $m$-basis type divisor $D_m\sim_{\mathbb Q}K_X$ compatible with $\cG$, we can write it as $D_m=a_mH_m+\Gamma_m$ where ${\rm Supp}(\Gamma)$, where $a_m\to \frac{1}{r}\int_0^{1}(1-x)^ndx=\frac{1}{r(n+1)}$. 

This implies that if $\cF$ is induced by an $\ord_E$ for some divisor, then choosing $H$ not containing the center of $E$, we have $\ord_{E}(D_m)=\ord_{E}(\Gamma_m)\le (1-ra_m)T(E)$, thus $S(E)\le \frac{n}{n+1}T(E)$. In particular, 
$$\delta(X)=\inf_E\frac{A_{X}(E)}{S(E)}\ge\frac{n+1}{n} \inf_E\frac{A_X(E)}{T(E)}=\frac{n+1}{n} \alpha(X),$$ 
which gives a proof of Tian's $\alpha$-invariant criterion.  
\end{exmp}

\begin{defn}For a subvariety $Z\subset X$, we define $\delta_Z(X)=\inf_E \frac{A_X(E)}{S(E)}$ where the infimum runs through over all divisors whose center on $X$ contains $Z$. In fact, we allow $Z$ to be reducible, then we define 
$$ \delta_{Z,m}(X):=\sup\{\lambda\ | \ Z\nsubseteq \mbox{ Nlc}(X,\lambda D_m \mbox{ for any $m$-basis type divisor}\} $$
and $\delta_{Z}(X)=\limsup_{m\to \infty}\delta_{Z,m}$.
\end{defn}
Similar to \eqref{e-delta}, we can prove when $Z$ is irreducible and reduced, the above two definitions coincide. 

Then we have the following lemma, which is a combination of various estimates.
\begin{lem}[{see \cite[Theorem 1.6]{Zhu20b}}]\label{l-multiplier}
Let $X$ be an $n$-dimensional Fano manifold. Then $X$ is K-stable if
\begin{enumerate}
\item   $\delta_Z(X)\ge\frac{n+1}{n}$ for any positive dimensional variety $Z$,
\item $\beta(E_x)>0$ for any ordinary blow up of $x\in X$.
\end{enumerate}
\begin{proof}For any divisor $F$ over $X$, let $\cF=\cF_F$ be the induced filtration and $\cG$ as chosen in the notation of Example \ref{e-alpha}. We want to show $\beta(F)>0$. By our first assumption, it suffices to assume that the center of $F$ is a smooth close point $x$ and $F\neq E_x$.  Using the notation there, we know for we can choose a sequence of constant $\mu_m\to 1$, such that $(X, \mu_m\frac{n+1}nD_m)$ is klt in a punctured neighborhood of $x$ and $\mu_m(1-ra_m)<\frac n{n+1}$. We claim the following fact, which implies what we need.
\begin{claim}$S_X(F)<A_X(F)$ for any $F\neq E_x$.
\end{claim}
\begin{proof}
 %It suffices to show that $(X,\Gamma_m)$ is klt as $H$ is general. 
 Since $(X,\fm^n_x)$ is plt with the only lc place $E_x$, for any $F\neq E_x$, $b:=\frac{A_{X}(F)}{\ord_F(\fm_x)}>n$.  Let $B:=a_m\frac{n+1}nD_m$. If $(X,B)$ is klt then $\ord_F(B)\le A_X(F)$. Otherwise, $(X,B)$ is not klt and  the multiplier ideal $ \mathcal{I}(X,B)$ cosupports on $x$. Since $H^0(X,\mathcal{O}_X)\to H^0(X, \mathcal{O}_X/ \mathcal{I}(X,B))$ is surjective, as $-K_X-B$ is ample, which implies the multiplier ideal $\mathcal{I}(X,B)=\fm_x$. In particular, by the definition of the multiplier ideal, $b^{-1}A_X(F)=\ord_F(\fm_x)\ge \ord_{F}B-A_{X}(F)$. 
 This implies $\ord_F(B)\le \frac{b+1}{b}A_X(F)$.  
 
 Thus we have
 $$\ord_F(D_m)= \frac{n}{(n+1)\cdot \mu_m} \ord_F(B) \le\frac{n(b+1)}{(n+1)b\mu_m}A_X(F) \mbox{\ \  for $m\gg 0$}.$$
Since  $b>n$, $S_X(F)=\limsup_m \ord_F(D_m)<A_X(F)$.
% Since $\beta(E_x)>0$, then $A_X(E_x)>\ord_{E_x}(D_m)$ for $m\gg 0$, this implies that $(X,D_m)$ is klt.
 \end{proof}
\end{proof}
\end{lem}
\subsubsection{Adjunction}\label{sss-adjunction} Next we will introduce a key idea from \cite{AZ20}, namely taking the adjunction of a filtration, and apply inversion of adjunction to get an estimate of $\delta$. 

Given two prime divisors where $E$ is a Cartier divisor on $X$ and $F$ either a Cartier divisor or a divisor which is the exceptional divisor from a plt blow up $\pi\colon Y\to X$. Let $V_m:=H^0(-mK_X)$ for some sufficiently divisible $m$. For a fix $j$, the following linear systems on $F$
$$(W_m)_{i,j}=\cF^j_F(\cF^i_E(V_m))/\cF_F^{j+1}(\cF^i_E(V_m))\cong $$
$${\rm Im}\Big( |\pi^*(-mK_X-iE)-jF)| \to  |\pi^*(-mK_X-iE)-jF)|_F\Big),$$
form a decreasing filtration of $W_{m,j}:=(W_m)_{0,j}$ indexed by $i\in \bN$.  

Therefore, if we denote $R:=\bigoplus_m H^0(-mrK_X)$ and $R_F:=\bigoplus_m W_m$ where $W_m=\bigoplus_j W_{m,j} $, we obtain an $\bN^2$-filtration on each $W_m$ (by $i$ and $j$), which is clearly multiplicative if we put all $m\in \bN$ together.   Then we can show (see \cite[Section 2]{AZ20})
\begin{defn-lem}For each $m$, we can define 
$$c_1(W_m):=\frac{1}{m\cdot \dim(W_m)}\sum_{(i,j)\in \bN^2}D_{i,j},$$ where $\{D_{i,j}\}$ is a basis type divisor compatible with the $\bN^2$-filtration of $W_m$. And we define $S_m(R;F):=\frac{1}{m\cdot \dim V_m}\sum_{i,j} j\cdot \big(\dim (\cF^j_FV_m)-\dim(\cF^{j+1}_FV_m)\big)(=rS_m(F)).$

Then  $c_1(W_{\bullet}):=\lim_m c_1(W_m)$, $S(R;F):=\lim_m S_m(R;F)$ exist and we have
$$ (-rK_X-S(R;F)F)|_F\sim_{\bR} c_1(W_{\bullet}).$$
\end{defn-lem}
\begin{proof}\cite{AZ20}*{(3.1)}.
\end{proof}
For $W_{\bullet}$ as above and any prime divisor $G$ over $F$, we can define $S(W_{\bullet},G)=\lim_m (\sup_{D_m} \ord_{G} D_m)$, where $D_m$ is an $m$-basis type divisor of $W_m$, and similarly define $\delta(W_{\bullet})=\inf_G \frac{A_{X}(G)}{S(W_{\bullet},G)}$ and the local version $\delta_{Z'}(W_{\bullet})$ for subset $Z'\subset F$.

The following (inversion of) adjunction type result is a key to establish the estimate.
\begin{prop}\label{p-kadjunction}
Notation as above, 
$$\delta_Z(X)\ge \min \{\frac{A_X(F)}{S(F)}, \delta_{\pi(F)\cap Z}(W_{\bullet}) \}.$$
\end{prop}
\begin{proof} For any prime divisor $G$ over $X$, by Lemma \ref{l-doublefil}, it suffices to prove
$$\lim_m\inf_{D_m}\lct(X,D_m)\ge \min \{\frac{A_X(F)}{S(F)}, \delta_{\pi(F)\cap Z}(W_{\bullet})\},$$
where $D$ is running through all $m$-basis type divisor compatible with the filtration induced by $F$. 

Fix such an $m$-basis type divisor $D_m$. Then $\pi^*D_m=a_mF+\Gamma_m$ and $\Gamma_m|_{F}$ gives a basis type divisor of $W_m$.
Thus $\pi^*(K_X+\lambda D_m)=K_Y+(1-A_{X}(F)+\lambda\cdot a_m)F+\lambda \Gamma_m$. From our assumption, if we choose $\lambda< \delta_{\pi(F)\cap Z}(W_{\bullet})$ and $m\gg 0$, we know that  $(F, {\rm Diff}_F(\lambda\cdot \Gamma_m)|_F)$ is klt along the preimage of the generic points of $\pi(F)\cap Z $, which implies that $(Y,\lambda\Gamma_m+F)$ is klt along the preimage of the generic point of $\pi(F)$. If moreover, $\lambda< \frac{A_X(F)}{S(F)}$, then for $m\gg 0$, $1-A_{X}(F)+\lambda\cdot a_m<1$ as $a_m\le S_m(F)$. Thus $(X,\lambda D_m)$ is klt. This is true for any $m$-basis type divisor as long as $m$ is sufficiently large, therefore, we conclude. 
\end{proof}
\begin{rem}The restricted linear system $(W_m)_{i,j}$ is often not a complete linear system, therefore to compute the corresponding $\delta$-invariant could be involving. Nevertheless, there are cases that the asymptotic invariant $\delta(W_{\bullet})$ is easy to compute. See \ref{say-flag}.
\end{rem}
\subsubsection{Flags} To proceed, we need to inductively apply the discussion in \ref{sss-adjunction} to consecutively cut down the dimension.
\begin{say}[Induction and flags]\label{say-flag} We choose a flag $Y_1\supset Y_2\supset \cdots \supset Y_{n}$. (We will impose more assumptions later.) We can inductively applying the adjunction result Proposition \ref{p-kadjunction} (for $X=Y_i$ and $F=Y_{i+1}$). Moreover, we replace the filtration by $E$ of the linear system by a $\bN^n$-filtration given by the vanishing order along the flags, which we can still apply Lemma \ref{l-doublefil}. Thus for any $\overrightarrow{a}=(a_1,...,a_i)\in \bN^i$, we can consider the the linear system 
$$W_{m,\overrightarrow{a}}^{i}\subset \Big(\cdots \big((-mK_X-a_1Y_1)_{Y_1}-a_2Y_2\big)_{Y_2}-\cdots  - a_iY_i\Big)_{Y_i}$$ on $Y_i$.

One issue in general is that  $W_{m,\overrightarrow{a}}^{i}$ is not a complete linear system. However, one will see in our specific case, if we put all $\overrightarrow{a}\in \bN^i$ together and let $m\to \infty$,  the limit $W^i_{\bullet}$ is {\it almost complete} in the sense that $c_1(W^i_{\bullet})$ is the same as the one if we replace $W_{m,\overrightarrow{a}}^{i}$  by the complete linear system  $W_{m,\overrightarrow{a}}^{i}\subset \Big(\cdots \big((-mK_X-a_1Y_1)_{Y_1}-a_2Y_2\big)_{Y_2}-\cdots - a_iY_i\Big)_{Y_i}$.
\end{say}

To apply Proposition \ref{p-kadjunction} repeatedly, we know that
\begin{prop}
$\delta_Z(X)\ge \min\{\min_{1\le i \le n}\frac{1}{S_{Y_i}(Y_{i-1})},  \delta_{Z\cap Y_{n-1}}(W_{\bullet})  \}. $
\end{prop}

To apply this in our situation, we choose two distinct points $y_1$ and $y_2$ on $Z$ (as $\dim(Z)\ge 1$). 
\begin{lem}Let $Y\subset \mathbb{P}^n$ be a smooth subvariety and $y_1, y_2\in Y$. Then a general hyperplane section containing $y_1$ and $y_2$ is smooth, except $\dim Y=2$ and the line passing through $y_1$ and $y_2$ is contained in $Y$.
\end{lem}
\begin{proof}Left to the reader. See \cite{AZ20}*{Lemma 4.25}.
\end{proof}

We first assume the line connecting $y_1$ and $y_2$  is not contained in $X$. Then we choose  $Y_1$,..., $Y_{n-1}$ general hyperplane section on on $X$ passing though $y_1$ and $y_2$ and $Y_n$ a general point on $Y_{n-1}$. We easily see the graded linear system $W^i_{\bullet}$ is almost complete for all $i$. Now we calculate,  for $1\le i\le n-1$,
$$c_1(W^i_{\bullet})=(1-\frac{1}{n-i+2})c_1(W^{i-1}_{\bullet})|_{Y_j}=2(1-\frac{i}{n+1})\mathcal{O}(1)|_{Y_i},$$
 and $S(W^{i-1}_{\bullet}, Y_i)=\frac{2}{n+1}$.
Taking $i=n-1$, we have $c_1(W^{n-1}_{\bullet})=\frac{4}{n+1}\mathcal{O}(1)|_{Y_{n-1}}$ which has degree $\frac{4n}{n+1}$.
Then a calculation on curve shows that
$$\delta_{Z\cap Y_{n-1}}(W^{n-1}_{\bullet})\ge\frac{2\cdot \#(Z\cap Y_{n-1})}{\deg(W^{n-1}_{\bullet})}\ge\frac{n+1}{n},$$
hence we can conclude by Lemma \ref{l-multiplier} and the fact that $\beta(E_x)>0$ for any exceptional divisor $E_x$ of the blow up of a smooth point $x\in X$ (see \cite[Lemma 4.24]{AZ20}).

Now we consider the case the line connecting $y_1,y_2$ is contained in $X$. Thus we choose $Y_1,..., Y_{n-2}$ as above, but $Y_{n-1}\cong \mathbb{P}^1$ the line, and $Y_n$ a general point on $Y_{n-1}$ which is distinct from $y_1,y_2$ (if $Y_{n-1}\subset Z$, then we do not need $Y_n$). 
We still can easily see the graded linear system $W^i_{\bullet}$ is almost complete for $i\le n-2$.  To see the same thing for $i=n-1$, we need to make a computation of the invariants on the smooth projective surface $Y_{n-2}$. See the proof of \cite{AZ20}*{Lemma 4.28} for more details.
Then we can make a similarly conclude that $\delta_{Z\cap Y_{n-1}}(W^{n-1}_{\bullet})\ge\frac{n+1}{n}$.
\begin{rem}In \cite{Zhu20}, it is proved that equivariant K-semistability (resp. equivariant K-polystability) for any group $G$ (any reductive group $G$) acts on $X$ implies K-semistability (resp. K-polystability). This provides a powerful tool to verify K-(semi,poly)stability of a Fano variety $X$ when it has a {\it large} symmetry (see e.g. \cite{Del20}).
\end{rem}
 
\section{Explicit K-moduli spaces}\label{s-moduli}
We can also prove a $\mathbb{Q}$-Fano variety is K-(semi,poly)-stable by showing that they appear on a K-moduli stack/space.
Note that although we have not yet known the algebraic construction of the compact K-moduli space, we can use the analytic depended construction \cite{LWX19}  for the closed component parametrising smoothable $\mathbb{Q}$-Fano varieties.  Here the compactness is crucial, since often we prove $X$ is K-(semi,poly)stable by showing it is a limit of a family of K-stable Fano manifolds, and all other possible limits are not K-(semi,poly)stable. This approach first  appeared explicitly in \cite{MM93} where they solved the case of degree 4 del Pezzo surface. Then in \cite{OSS16}, all cases of smoothable surfaces were completed. In \cite{SS17}, the authors settled the case of intersections of two quadrics. In \cite{LX19} and \cite{Liu20}, the compact K-moduli of cubic threefolds and fourfolds are proved to be the same as the GIT moduli of cubic threefolds.   More recently in \cite{ADL19,ADL20}, cases of log surfaces are also studied, and by varying the coefficients of the boundary, a sequence of birational moduli spaces were established (see also \cite{GMS18}). 

\subsection{Deformation and degeneration}\label{ss-compactmoduli}
 We start with a Fano manifold $X_{t^*}$ which is a fiber of a family $\kX\to B$ over an irreducible base, and assume we know that $X_{t^*}$ is K-semistable. Then we know that there is a Zariski open set $B^0\subset B$ such that for any $t\in B^0$, the fiber $X_t$ is K-semistable (see \cite{BLX19, Xu19}). However, usually we do not a priori know how large $B^0$ is. Nevertheless, let  $C\to B$ be a map from a curve  whose  image contains $t_0$. Then by the compactness of the K-moduli space, we know that for any $s\in C$, we can extend the family of $\kX\times (B^0\cap C)\to (B^0\cap C)$ over $s$ such that the special fiber $X_{s}$ is K-semistable.  

To determine what $X_s$ is, the local restriction could be useful:
By Theorem \ref{t-localtoglobal}, we know that any point $x\in X_{s}$ is a smoothable singularity with  
\begin{eqnarray}\label{e-localglobal}
 \hvol(x, X_{s})\ge (-K_{X_{s}})^n\cdot \big(\frac{n}{n+1}\big)^n.
 \end{eqnarray}
When $(-K_{X_{s}})^n$ is large, this local condition could post strong restrictions for possible $X_{s}$. For instance, by \cite{LX18, XZ20}, we know that 
\begin{eqnarray}\label{e-cover}
{\pi}_1^{\rm loc}(x, X_{s})\le \frac{n^n}{\hvol(x, X_{s})}\le \frac{(n+1)^n}{(-K_{X_{s}})^n},
\end{eqnarray}
and the first inequality holds if and only if $x\in X_{s}$ is a quotient singularity. 

For surfaces, the above approach gives a robust method. In \cite{OSS16} this approach has been used to identify the singular surfaces parametrized by the boundary of the moduli space of smooth del Pezzo surfaces of degree 1, 2 or 3.

However, in higher dimension, in general it is much harder to explicitly write down a compact K-moduli space. There are two difficulties: first there are not that many known compact moduli spaces parametrizing Fano varieties, which could be a candidate of K-moduli; secondly, to calculate the normalized volume of  klt singularities of dimension larger or equal to three is difficult. One technical result in dimension three is the following. 

\begin{thm}[\cite{LX19}]\label{t-threefold}
For any three dimension singularity $x\in X$, $\hvol(x,X)\le 16$, and the equality holds if and only if $x\in X$ is the rational double point.  
\end{thm}
The proof of Theorem \ref{t-threefold} relies heavily on the classification of three dimensional canonical and terminal singularities, which is classical in MMP. 
This is a key ingredient to prove the following theorem, which gives the first example of using the explicit K-moduli space to find out new smooth K-stable Fano manifolds in dimension at least three. 
\begin{thm}[{\cite{LX19}}]\label{t-cubics}
The GIT moduli of cubic threefolds is the same as the K-moduli.
\end{thm}
\begin{proof}[Sketch of the proof]By the discussion above, the volume of any $x\in X_s$ is at least $\frac{81}{8}$ (see \eqref{e-localglobal}). So by Theorem \ref{t-threefold},  it is either $\frac{1}{2}(1,1,0)$ or the local fundamental group ${\pi}_1^{\rm loc}(x,X_s)$ is trivial. In each  case, we can easily show that the degeneration of the hyperplane class $\mathcal{O}(1)$ on  $X_s$ is still a Cartier divisor. Then a classical result by Takao Fujita shows that $X_s$ is indeed also a cubic threefold.

Then we know it is GIT (semi,poly)stable if it is K-(semi,poly)stable. 
\end{proof}
In \cite{SS17}, it was shown that the compact  K-moduli space of the intersection of two quadrics in $\mathbb{P}^{n+2}$ is identical to the GIT moduli. In fact, \eqref{e-localglobal} directly yields that for any point $x$ on the K-semistable degeneration $X_s$,  we have $\hvol(x, X_s)\ge \frac{n^n}{2}$, which implies that the degeneration of $\mathcal{O}(1)$ is Cartier, and $X_s$ is also an intersection of two quadrics.

An analogue result of Theorem \ref{t-cubics} for cubics in higher dimensions would follow a similar result of Theorem \ref{t-threefold} in higher dimensions, but the latter seems to be hard to establish in general. Nevertheless, in \cite{Liu20}, by combining Theorem \ref{t-threefold} and the effective non-vanishing theorem (see \cite{Kaw00}) to analyze the linear system $\mathcal{O}(1)$ and $\mathcal{O}(2)$ on $X_s$, Liu also solved the 4 dimensional case. 

\begin{thm}[\cite{Liu20}]\label{t-cubicfourfold}
The GIT moduli of cubic fourfolds is the same as the K-moduli.
\end{thm}

\begin{rem}From a computational viewpoint, an explicit description of the boundary point of the K-moduli should often be easier to reach than the KSB moduli case, since the limit has to be klt, in particular normal. 
\end{rem}

\begin{rem} While the approach in Section \ref{s-sing} is mostly for Fano varieties with a small volume, the approach in this section is mainly for Fano varieties with a big volume. However, it seems there is still a gap between these two approaches, and we do not know how to deal with general Fano varieties with intermediate volumes. See Remark \ref{e-dim3}.
\end{rem}

\subsection{Wall crossings and log surfaces}
One interesting new direction is to consider the moduli space of log Fano varieties $(X,t\Delta)$, and vary the coefficient $t$. This reminisces the work in \cite{Has03} which studies a similar setting for KSB stable log curves $(C,tD)$.

In \cite{ADL19}, the authors develop the framework: first they show that the moduli space $\overline{\cM}^{\rm sm, Kss}_{n,c,V,r}$ which parametrizes $n$-dimensional K-semistable log Fano pairs $(X,cD)$ with $ (-K_X)^{n}=V$  that can be deformed to smooth Fano manifolds $X_t$ with a smooth divisor $D_t\in |-rK_{X_t}|$ for some $r\in \bQ$, is represented by an Artin stack of finite type. Moreover, it admits a proper good moduli space $\overline{M}^{\rm sm, Kps}_{n,c,V,r}$. While this is a straightforward generalization of \cite{LWX19} into the log situation, the main new structure, established in \cite{ADL19}, is that there is a wall crossing if we vary $c$ in the following sense.
\begin{thm}[{\cite{ADL19}*{Thm. 3.2}}]There exists a sequence of rational numbers 
$$0=c_0<c_1<c_2<...<c_k=\min \{1,\frac{1}{r}\},$$ such that $\overline{\cM}^{\rm sm, Kss}_{n,c,V,r}$ is the same for any $c\in (c_{i-1},c_i)$. And for any $0<j<k$, we have
$$\overline{\cM}^{\rm sm, Kss}_{n,c_j-\epsilon,V,r}\varhookrightarrow \overline{\cM}^{\rm sm, Kss}_{n,c_j,V,r} \varhookleftarrow \overline{\cM}^{\rm sm, Kss}_{n,c_j+\epsilon,V,r},$$
which when pass to good moduli space gives projective morphisms admitting a local VGIT description (see \cite{AFS17}*{Definition 2.4}).
\end{thm}  
One natural case to consider is the compactification of the family containing $(\mathbb P^n, cD)$ where $D$ is a degree $d$ hypersurface. It can be easily shown that for a general $D\in |\mathcal{O}(d)|$, and $c<\min\{1,\frac{n+1}{d}\}$, $(\mathbb{P}^n,cD)$ is K-semistable.  Then it follows for the different choice of $c$, we get  compactifications which are birational to each other.  Moreover, when $0<c\ll 1$, the compactification is the same as the GIT moduli of degree $d$ hypersurfaces (see \cite[Theorem 1.4]{ADL19}). 
When one increases from $c\ll 1$ to find the walls, in general, such a computation will be difficult. In \cite{ADL19}, the authors worked out all walls for the case $n=2$ and $d=4,5,6$, using the strategy in Section \ref{ss-compactmoduli} but in the log setting. 

Using a similar framework,  in \cite{ADL20} the authors worked out the case for $(4,4)$ curves in $\mathbb{P}^1\times \mathbb{P}^1$ and showed the moduli spaces are identical to the ones constructed in \cite{LO18} constructed from a variation of GIT process. 

\begin{rem}[Fano threefolds]\label{e-dim3} 
The K\"ahler-Einstein problem for smooth del Pezzo surfaces has been completely answered by Tian \cite{Tia90}. So it is natural to take a look at all smooth Fano threefolds which are classified by Iskovskih when $\rho(X)=1$ (see \cite{Isk78}) and Mori-Mukai when $\rho(X)\ge 2$ (see \cite{MM81}).

 When $\rho(X)=1$, there are nineteen families. Among them,  the one with genus 12 (i.e. $(-K_X)^3=22$) which contains the famous Mukai-Umemura manifold has attract lots of interests (see \cite{CS18} for recent progress).

%\begin{itemize}
%\item $\mathbb{P}^3$, which is K-polystable;
%\item the smooth quadric $Q\subset \mathbb{P}^4$, which is K-polystable;
%\item $V_5$ a codimension 3 linear section of ${\rm Gr}(2,5)\subset \mathbb{P}^9$, which is K-polystable by \cite{CS09};
%\item $V_4$ a  smooth intersections of two quadrics in $\mathbb{P}^5$, which is K-stable by \cite{AGP06};
%\item $V_3$ a smooth cubic in $\mathbb{P}^4$, which are all K-stable by \cite{LX19};
%\item $V_2$ a double cover of $\mathbb{P}^3$ with ramification degree 4, which is K-stable by \cite[Theorem 1.1]{Der16a}.
%\item $V_1$ a degree 6 hypersurface in $\mathbb{P}(1,1,1,2,3)$. See \cite{AZ20}*{Coro. 4.9}. 
%\end{itemize}

%There are currently five families with $r(X)=1$ that we completely know the K-(semi,poly)stability. They are
%\begin{itemize}
%\item $X_2$ a double cover of $\mathbb{P}^3$ with ramification degree 6, which is K-stable by \cite[Theorem 1.1]{Der16a};
%\item $X_4$ a double cover of a quadric in $\mathbb{P}^4$  ramified in a quartic section, which is K-stable by \cite[Theorem 1.1]{Der16a};
%\item $X_4$ a quartic in $\mathbb{P}^4$, which is K-stable by \cite{Fuj19c};
%\item $X_6$ a smooth intersection of a quadric and a cubic in $\mathbb{P}^5$, which is K-stable by \cite{Zhu20b};
%\item $X_{10}$ a double cover of $V_5$, which is K-stable by \cite[Theorem 1.1]{Der16a}.
%\end{itemize}

% For the rest seven families on the classification list,  we can not completely characterise the K-(semi,poly)stability for all smooth members.   

 Let 
$$r(X):=\{r | -K_X\sim -rH, \mbox{ for $H$ generator of ${\rm Pic}(X)$} \}$$ be the Fano index. 
There are seven families with $r(X)\ge 2$ and they are all known to be K-polystable or K-stable. 

It is tempting to speculate
\begin{conj}\label{c-threefold}
All smooth Fano threefolds with $\rho(X)=1$ are K-semistable. 
\end{conj}

Using the classification of smooth Fano threefolds, it is indeed an active research topic to completely determine all their K-(semi,poly)stability.  For $\rho(X)\ge 2$, in \cite[Section 10]{Fuj16}, $\beta$-invariants for divisors on $X$ itself are computed, and many cases are shown to be K-unstable. 

\begin{ques}

For Fano threefolds with $\rho(X)=1$ and $r(X)=1$,  it will be very interesting to connect various compactifications of moduli of K3 surfaces of genus $g$ ($1\le g \le 10$ and $g=12$) and the K-moduli of the corresponding Fano threefolds with the volume $2g-2$, i.e. to extend the well studied correspondence to the boundary of the moduli spaces. 

\end{ques}
 
The higher dimensional analogue of Conjecture \ref{c-threefold}, as originally proposed in \cite[Conjecture 5.1]{OO13} which predicted that smooth Picard number 1 Fano  manifolds in any dimension  are K-semistable, was \emph{disproved} in \cite{Fuj17a}, where it is shown that a del Pezzo manifold $X$ of degree five with dimension four or five are K-unstable.

\end{rem}

\section*{Notes on history}

For a long time, Tian's $\alpha$-invariant criterion, established in \cite{Tia87} (see also \cite{OS12}), and its equivariant version was the main tool that people used to verify the existence of a KE metric on a given Fano manifold. There is a long list of literature, in which people estimate the lower bound of the $\alpha$-invariant for various Fano manifolds. See e.g. \cite{Che01, Che08, CS09, CS18} etc.

However, the $\alpha$-invariant criterion is only a sufficient condition, and there are many Fano manifolds which can not be treated using it.  The Fujita-Li's criterion (see \cite{Fuj19b, Li17} or Theorem \ref{t-valkstable}) provides a necessary and sufficient condition. Using it Fujita shows that a Fano manifold $X$ with $\alpha=\frac{\dim(X)}{\dim(X)+1}$ is always K-stable if $\dim(X)\ge 2$ (see \cite{Fuj19}). As far as I know, this is the first result that one can show certain Fano manifolds are K-stable, before differential geometers find KE metric on them. Then the proof of the K-stability of a birational superrigid Fano manifold $X$ if $\alpha(X)\ge\frac{1}{2}$ (see \cite{Zhu20b, SZ19}) gives another example that the new criterion can be adapted to verify cases that people have wondered for a while. 

The basis type divisor introduced in \cite{FO18} gives an explicit way to check the Fujita-Li's criterion. Quickly after it was invented, it has been used to estimate the $\delta$-invariant of a number of families of del Pezzo surfaces, hence reverify their K-stability algebraically (see \cites{PW18, CZ19}) etc.. Then in the remarkable work \cite{AZ20},  a powerful method, using inversion of adjunction for the basis type divisor,  was established, and it is used to prove the K-stability of all smooth degree $n$-Fano hypersurfaces in $\mathbb{P}^{n+1}$ ($n\ge 3$).

\smallskip

The moduli method was already implicitly contained in \cite{Tia90}, when he tried to establish the KE metric on a smooth del Pezzo surface (of degree at most 4) by continuously extending the KE metrics on a sequence of nearby del Pezzo surfaces. It was first explicitly applied in \cite{MM93} to find all K-(semi,poly)stable degree 4 smoothable del Pezzo surfaces. Then in \cite{OSS16}, the case of K-(semi,poly)stable smoothable del Pezzo surfaces was completely solved. In \cite{ADL19, ADL20} (see also \cite{GMS18}), the log surface cases were studied. In particular, the cases of $\mathbb{P}^2$ and $\mathbb P^1\times \mathbb P^1$ together with low degree curves were systematically studied. As a result, the wall crossing phenomena arose when one varies the coefficients, and put a natural framework to connect various moduli spaces, many of which have already appeared in the literatures by construction from different theories.  

In higher dimensions, there are few examples which one can identify the entire compact moduli space, including the singular ones. The known examples include the intersection of two quadratics (see \cite{SS17}), cubic threefolds/fourfolds (see \cite{LX19, Liu20}), and $(\mathbb{P}^n, cD)$ when $c\ll 1$ (\cite{ADL19}). 

%\begin{say}[Conic trick]Let $(X,\Delta)$ be a log Fano pair, with a special TC $(\cX,\Theta)$. The generalize Futaki invariant can be computed using the  
%\end{say}

\begin{comment}

\part{Open questions}

\section{Open questions}

\begin{conj}[]
\end{conj}
\begin{conj}
\end{conj}
\clearpage
\end{comment}

\bigskip

\end{document}